\documentclass[a4paper] {article}
[12pt]
\makeatletter
\renewcommand*\l@section{\@dottedtocline{1}{1.5em}{2.3em}}
\makeatother
\usepackage{amsfonts}
\usepackage{amssymb}
\usepackage[colorlinks=true,
            linkcolor=blue,
            anchorcolor=blue,  
            citecolor=blue,        
            ]{hyperref}
\usepackage[T1]{fontenc}

\usepackage{tikz}\usepackage{float}
\usetikzlibrary{calc}

\usepackage{CJK}
\usepackage{amsmath}

\usepackage{amsfonts}
\usepackage{amssymb}
\usepackage{amsthm}
\usepackage{amssymb}
\usepackage{enumerate}
\usepackage[calc]{picture}\usepackage{graphicx}\usepackage{subfigure}\usepackage{cite}\usepackage{booktabs}
\usepackage[all,cmtip]{xy}

\usepackage{eqlist}

\usepackage{color}
\usepackage{abstract}

\usepackage[top=3.0 cm, bottom=3.0cm, left=3.5cm, right= 3.5cm]{geometry}

\usepackage{cite}

\theoremstyle{plain}
\newtheorem{theorem}{Theorem}[section]
\newtheorem{proposition}[theorem]{Proposition}
\newtheorem{lemma}[theorem]{Lemma}
\newtheorem{corollary}[theorem]{Corollary}

\theoremstyle{definition}
\newtheorem{definition}{Definition}[section]

\theoremstyle{example}
\newtheorem{example}{Example}[section]

\theoremstyle{remark}
 \newtheorem{remark}{Remark}[section]

\usepackage{etoolbox}

\numberwithin{equation}{section}
\numberwithin{theorem}{section}

\begin{document}
    \begin{CJK*}{GBK}{kai}
    \CJKtilde

\begin{center}
{\Large {\textbf {On the Cayley-persistence algebra
}}}
 \vspace{0.58cm}

Wanying Bi, Jingyan Li, Jian Liu, Jie Wu*

\bigskip

\bigskip

    \parbox{24cc}{{\small
{\textbf{Abstract}.}
In this paper, we introduce a persistent (co)homology theory for Cayley digraph grading. We give the algebraic structures of Cayley-persistence object. Specifically, we consider the module structure of persistent (co)homology and show the decomposition of a finitely generated Cayley-persistence module. Moreover, we introduce the persistence-cup product on the Cayley-persistence module and study the twisted structure with respect to the persistence-cup product.
As an application on manifolds, we show that the persistent (co)homology is closely related to the persistent map of fundamental classes.
}}
\end{center}

\vspace{1cc}



\footnotetext[3] {* corresponding author. }
\section{Introduction}
Topological data analysis (TDA) has become one of the most successful applications of topology in data science. The persistent homology introduced by  G. Carlsson, A. Zomorodian et al. \cite{carlsson2005persistence,Zomorodian2005} is still one of the most important methods in TDA and has become a classic tool for extracting topological features. Nowadays, persistent homology has been widely used in various aspects of image science \cite{carlsson2008local,bonis2016persistence,edelsbrunner2013persistent}, neuroscience \cite{carlsson2007topological,baas2017concept,stolz2014computational}, biomolecular \cite{cang2015topological,gameiro2015topological,xia2014persistent}, machine learning \cite{clough2019topological,bae2017beyond,pun2018persistent,giansiracusa2017persistent}, materials science \cite{jiang2021topological,chen2020topology,ichinomiya2017persistent}, etc. These applications have led to updates and developments of persistent homology from theory to computation.

Different variants of persistent homology were developed over the years. In \cite{carlsson2009computing,carlsson2009theory}, G. Carlsson, A. Zomorodian et al. introduced the multidimensional (or multi-parameter) persistent homology, which is used to deal with multidimensional filtration of simplicial complexes.  And then G. Carlsson and V. D. Silva \cite{carlsson2010zigzag} developed the theory of zigzag persistence for studying the persistence of topological features across a family of spaces or point-cloud data sets. It is from the philosophy that the decomposition theory of graph representations is somewhat independent of the orientation of the graph edges.
In \cite{de2011persistent}, the persistent cohomology (or copersistence) was introduced to identify candidates for significant circle-structures in the data. These variants of persistence have also been widely concerned and applied in different fields \cite{xia2015multidimensional,tausz2011applications,liu2021hypergraph,corbet2019kernel}.

We compare the above outstanding variants of persistence as follows.
\begin{table}[H]
\centering
\begin{tabular}{c|c|c}
    \toprule
 Different persistences & Gradings & The corresponding spaces\\
  \midrule
  The usual persistence & $\mathbb{Z}$ (or $\mathbb{R}$) & $\mathbb{Z}$-graded vector spaces \\
  \hline
  Multidimensional persistence& $\mathbb{Z}^{n}$ &$\mathbb{Z}^{n}$-graded vector spaces \\
\hline
  Zigzag persistence & Zigzag sequence & Zigzag diagrams of vector spaces  \\
\hline
  Copersistence & $\mathbb{Z}$ & $\mathbb{Z}$-graded dual vector spaces \\
  \bottomrule
\end{tabular}
\end{table}

In this paper, we consider the persistence based on group grading. The group graded algebra \cite{bahturin2001group,bahturin2002group} has been studied for many years. To give the persistence on a group, we recall the Cayley digraph \cite{xu1998automorphism,dobson2017cayley}, which endows a group with a direction in some sense. Let $G$ be a group and $S$ be a subset of $G$. A \emph{Cayley digraph} $\mathrm{Cay}(G,S)$ is a digraph with the elements of $G$ as vertices and the pairs $(a,b)\in G\times G$ satisfying $ba^{-1}\in S$ as arcs\footnote{The condition $ba^{-1}\in S$ in this  paper is a little different from that described in other papers.}. We may regard a Cayley digraph $\mathrm{Cay}(G,S)$ as a category, denoted by $\mathrm{cat}(\mathrm{Cay}(G,S))$, with the vertices as objects and directed paths as morphisms.
We assume that $S$ is a monoid. Indeed, if $S$ is a subset of $G$, let $\langle S\rangle$ be the monoid generated by $S$. Then we have (see Lemma \ref{lemma:monoid})
\begin{equation*}
  \mathrm{cat}(\mathrm{Cay}(G,S))=\mathrm{cat}(\mathrm{Cay}(G,\langle S \rangle)).
\end{equation*}
Let $\mathfrak{C}$ be a category. A \emph{Cayley-persistence object} is a functor $\mathcal{F}:\mathrm{cat}(\mathrm{Cay}(G,S))\rightarrow \mathfrak{C}$ from the category  $\mathrm{cat}(\mathrm{Cay}(G,S))$ to $\mathfrak{C}$.
Let $l_{x}:a\rightarrow xa$ be a $G$-graded map for $a\in G,x\in S$ in the category $\mathrm{cat}(\mathrm{Cay}(G,S))$. Then we have a morphism
\begin{equation*}
  \mathcal{F}(l_{x}):\mathcal{F}_{a}\rightarrow \mathcal{F}_{xa},\quad a\in G,x\in S
\end{equation*}
in category $\mathfrak{C}$. Moreover, the set $\{\mathcal{F}(l_{x})\}_{x\in S}$ can be regarded as a monoid with composition as multiplication, denoted by $\mathcal{L}_{S}$.

From now on, let $\mathbf{k}$ be a field. Let $M_{\mathcal{F}}=\bigoplus\limits_{x\in G}\bigoplus\limits_{\alpha\in \mathcal{F}_{x}}\mathbf{k}\alpha$ be a $\mathbf{k}$-linear space generated by the elements in $\mathcal{F}_{x},x\in G$. Then $M_{\mathcal{F}}$ is a $\mathbf{k}[\mathcal{L}_{S}]$-module, where $\mathbf{k}[\mathcal{L}_{S}]$ is the monoid ring of $\mathcal{L}_{S}$ over $\mathbf{k}$. Our first result is expressed as follows (see Theorem \ref{thm:module_decomposition}).
\begin{theorem}\label{theorem:main1}
Let $M_{\mathcal{F}}$ be a finitely generated $\mathbf{k}[\mathcal{L}_{S}]$-module and $S$ be a finitely generated monoid. Then we have a finite direct sum decomposition
\begin{equation*}
  M_{\mathcal{F}}\cong\bigoplus_{i=1}^{k} \mathbf{k}[\mathcal{L}_{S}]\cdot e_{i}^{x_{i}}\oplus \left(\left(\bigoplus_{j=1}^{l}\mathbf{k}[\mathcal{L}_{S}]\cdot\varepsilon_{j}^{y_{j}}\right)/N\right)
\end{equation*}
for some $k,l$, where $e_{i}^{x_{i}}\in \mathcal{F}_{x_{i}},\varepsilon_{j}^{y_{j}}\in \mathcal{F}_{y_{j}}$ and $N$ is a finitely generated $\mathbf{k}[\mathcal{L}_{S}]$-module generated by the elements of the form $\mathcal{F}(l_{y})\varepsilon_{t}^{y_{t}}-\mathcal{F}(l_{yy_{t}y_{s}^{-1}})\varepsilon_{s}^{y_{s}}$ for some $1\leq s,t\leq l,y\in S$.
\end{theorem}
In \cite{carlsson2009theory}, G. Carlsson and A. Zomorodian considered classification of the multidimensional persistence module. As a corollary of Theorem \ref{theorem:main1}, we have a Cayley-persistence structure of modules based on the finitely generated module (see Corollary \ref{corollary:decomposition}).
\begin{corollary}
Let $\mathcal{F}:\mathrm{cat}(\mathrm{Cay}(G,S))\rightarrow \mathbf{Vec}_{\mathbf{k}}$ be a Cayley-persistence $\mathbf{k}$-linear space.
If $\mathcal{F}_{G}=\bigoplus\limits_{x\in G}\mathcal{F}_{x}$ is a finitely generated $\mathbf{k}[\mathcal{L}_{S}]$-module and $S$ is a finitely generated monoid, then we have a finite direct sum decomposition
\begin{equation*}
   \mathcal{F}_{G} \cong \bigoplus_{i=1}^{k} \mathbf{k}[\mathcal{L}_{S}]\cdot e_{i}^{x_{i}}\oplus \left(\left(\bigoplus_{j=1}^{l}\mathbf{k}[\mathcal{L}_{S}]\cdot\varepsilon_{j}^{y_{j}}\right)/N\right)
\end{equation*}
for some $k,l$, where $e_{i}^{x_{i}}\in \mathcal{F}_{x_{i}},\varepsilon_{j}^{y_{j}}\in \mathcal{F}_{y_{j}}$ and $N$ is a finitely generated $\mathbf{k}[\mathcal{L}_{S}]$-module generated by the elements of the form $\mathcal{F}(l_{y})\varepsilon_{t}^{y_{t}}-\mathcal{F}(l_{yy_{t}y_{s}^{-1}})\varepsilon_{s}^{y_{s}}$ or $\mathcal{F}(l_{y})\varepsilon_{t}^{y_{t}}$ for some $1\leq s,t\leq l,y\in S$.
\end{corollary}
This can be interpreted that the Cayley-persistence generators prefer to survive in some areas, not just in some intervals. We always hope that there are finitely many generators and the survival spaces can be computed. However, the condition that $M_{\mathcal{F}}$ or $\mathcal{F}_{G}$ is a finitely generated $\mathbf{k}[\mathcal{L}_{S}]$-module does not always hold. We then give a ``lower bounded and noetherian''  condition for $M_{\mathcal{F}}$ being a finitely generated $\mathbf{k}[\mathcal{L}_{S}]$-module in Theorem \ref{theorem:finitely}. Moreover, the multidimensional persistent homology considered in application always satisfies this condition.

Let $\mathcal{K}:\mathrm{cat}(\mathrm{Cay}(G,S))\rightarrow \mathbf{Simp}$ be a Cayley-persistence simplicial complex such that the homology $H_{\ast}(\mathcal{K}_{a};\mathbf{k})$ is of finite dimension  for all $a\in G$. Then the (reduced) cohomology induces a Cayley-copersistence module
\begin{equation*}
  H^{\ast}(\mathcal{K};\mathbf{k}):\mathrm{cat}(\mathrm{Cay}(G,S))\rightarrow \mathbf{Vec}_{\mathbf{k}}.
\end{equation*}
Let $S\subseteq G$ be a monoid such that the identity element $e$ is the unique invertible element in $S$. Then the group $G$ can be regarded as a poset with partial order given by $a\leq b$ if $ba^{-1}\in S$.
Suppose that the category $\mathrm{cat}(\mathrm{Cay}(G,S))$ has finite product, for example, $G$ is a lattice group \cite{anderson2012lattice,birkhoff1987lattice}. Let $\mathbf{H}=\bigoplus\limits_{a\in G}H^{\ast}(\mathcal{K}_{a};\mathbf{k})$. The morphism
\begin{equation*}
  L_{x}=H^{\ast}(\mathcal{K}_{a,xa};\mathbf{k}):H^{\ast}_{xa}\rightarrow H^{\ast}_{a},\quad a\in G,x\in S.
\end{equation*}
induced by $a\to xa$ gives a right action on $\mathbf{H}$. Note that the set $\mathbb{L}_{S}=\{L_{x}\}_{x\in S}$ is a monoid with multiplication $L_{x}\cdot L_{y}=L_{y}\circ L_{x}$ for $x,y\in S$. Let $\mathbf{k}[\mathbb{L}_{S}]$ be a monoid ring of $\mathbb{L}_{S}$ over $\mathbf{k}$, then $\mathbf{H}$ is a right $\mathbf{k}[\mathbb{L}_{S}]$-module.
Moreover, we introduce the persistence-cup product on $\mathbf{H}$ in Section \ref{section:algebra}, which leads to the persistence $\mathbf{k}[\mathbb{L}_{S}]$-twisted algebra (see Theorem \ref{theorem:twist}).
\begin{theorem}\label{theorem:main2}
Let $G$ be an abelian group.
The persistence-cup product on $\mathbf{H}$ is uniquely determined by the persistence-cup product of $\mathbf{k}[\mathbb{L}_{S}]$-module generators. Moreover, $(\mathbf{H},\cdot)$ is a $G$-graded $\mathbf{k}[\mathbb{L}_{S}]$-twisted algebra.
\end{theorem}
Now, we consider the case that $G$ is an ordered abelian group. By endowing $S$  with an interval topology, we have a topological monoid $S$. There is a decomposition theorem of persistent cohomology for ordered group grading (see Theorem \ref{thm:decomclousure}), which is essentially from the idea of A. Zomorodian and G. Carlsson.
\begin{theorem}\label{theorem:main3}
Let $G$ be an ordered abelian group. Let $\mathbf{H}$ be a finitely generated $\mathbf{k}[\mathbb{L}_{S}]$-module.
Then we have a finite direct sum decomposition
\begin{equation*}
   \mathbf{H}\cong \bigoplus_{i=1}^{k}  e^{i}_{x_{i}} \cdot \mathbf{k}[\mathbb{L}_{S}]\oplus \bigoplus_{j=1}^{l}\frac{\varepsilon^{j}_{y_{j}}\cdot \mathbf{k}[\mathbb{L}_{S}]}{\mathbf{k}[I_{j}]}
\end{equation*}
for some $k,l$. Here, $\overline{I_{j}}$ is the closure of $I_{j}$ and  $\mathbf{k}[\overline{I_{j}}]=\varepsilon^{j}_{y_{j}}\cdot L_{w_{j}}\cdot R$ for some $w_{j}\in S,j=1,\dots,l$.
\end{theorem}
Note that $\mathbf{k}[I_{j}]$ does not have to be a finitely generated $\mathbf{k}[\mathbb{L}_{S}]$-module.
The structure of (co)persistence module can be described by the barcode, which reads the survival time of generators of the (co)persistence module. For an element $\alpha\in \mathbf{H}$, we denote
\begin{equation*}
  I(\alpha)=\mathrm{supp}(\alpha)=\{x\in G|\alpha_{x}\neq 0\}
\end{equation*}
the support of $\alpha$ in grading $G$. In Section \ref{section:algebra}, we introduce the persistence-cup product.  Let $\alpha,\beta$ be two of the generators of $\mathbf{H}$ as $\mathbf{k}[\mathbb{L}_{S}]$-module in Theorem \ref{theorem:main3}.
As an application of Theorems \ref{theorem:main2} and \ref{theorem:main3}, We have a description of barcode for the persistence-cup product elements (see Theorem \ref{theorem:main}).
\begin{theorem}\label{theorem:main4}
Let $G$ be an ordered abelian group. Let .
If $\alpha\cdot\beta$ is nontrivial, then
\begin{equation*}
\begin{split}
  \min(\sup (I(\alpha)),\sup (I(\beta)))&\leq \sup(I(\alpha\cdot\beta)), \\
  \max(\inf (I(\alpha)),\inf (I(\beta)))&\leq \inf(I(\alpha\cdot\beta)).
\end{split}
\end{equation*}
\end{theorem}
At last, we consider the Cayley-persistence manifold $\mathcal{M}:\mathrm{cat}(\mathrm{Cay}(G,S))\rightarrow \mathbf{Mani}$. All the manifolds considered are assumed to be compact orientable $n$-manifolds without boundary. Recall that the Poincar\'{e} duality
\begin{equation*}
  D:H^{p}(M; \mathbf{k})\rightarrow H_{n-p}(M; \mathbf{k})
\end{equation*}
defined by $D(\alpha) = \omega\cap \alpha$ is an isomorphism for all $p$. Here, $\omega$ is the fundamental class in $H_{n}(M;\mathbf{k})$.
Let $G$ be an abelian group, and  let $S\subseteq G$ be a monoid such that the identity element $e$ is the unique invertible element in $S$.
Let
\begin{eqnarray*}
  H^{a,b}_{\ast} &=& \mathrm{im}(H^{a}_{\ast}(\mathcal{M};\mathbf{k})\to H^{b}_{\ast}(\mathcal{M};\mathbf{k})), \\
  H_{a,b}^{\ast} &=& \mathrm{im}(H_{b}^{\ast}(\mathcal{M};\mathbf{k})\to H_{a}^{\ast}(\mathcal{M};\mathbf{k})),  \\
  P_{a,b}^{\ast} &=& \mathrm{im}(H_{b}^{\ast}(\mathcal{M};\mathbf{k})\times H_{b}^{\ast}(\mathcal{M};\mathbf{k})\stackrel{\cup}{\rightarrow} H_{b}^{\ast}(\mathcal{M};\mathbf{k})\rightarrow H_{a}^{\ast}(\mathcal{M};\mathbf{k}))
\end{eqnarray*}
be the $(a,b)$-persistent homology, cohomology, and cup-space, respectively. We show that the persistent cohomology of the Cayley-persistence manifold is influenced by the map of fundamental classes (see Theorem \ref{theorem:duality}).
\begin{theorem}\label{theorem:main5}
For $a,b\in G$ with $ba^{-1}\in S$, let $f^{a,b}_{\ast}:H_{\ast}^{a}\rightarrow H_{\ast}^{b}$ be a map induced by $\mathcal{M}_{a}\rightarrow \mathcal{M}_{b}$ and $f^{a,b}_{n}(\omega^{a})=\lambda^{a,b}\omega^{b},\lambda^{a,b}\in \mathbf{k}$. Here, $\omega^{a},\omega^{b}$ are the fundamental classes in $H^{a}_{n},H^{b}_{n}$, respectively.
\begin{itemize}
  \item[$(i)$] If $\lambda^{a,b}\neq 0$, then the map $$D^{a,b}=f^{a,b}_{\ast}\circ D: H^{p}_{a,b}\rightarrow H_{n-p}^{a,b}$$ is an isomorphism for all $p$. Moreover, we have
  $\beta_{a,b}^{p}=\beta_{a,b}^{n-p}$.
  \item[$(ii)$] If $\lambda^{a,b}=0$, then the  $(a,b)$-persistent cup-space $P_{a,b}^{n}=0$. Moreover, we have
  \begin{equation*}
  \beta_{p}^{a,b}+\beta_{n-p}^{a,b}\leq \beta_{p}^{a},\quad \beta^{p}_{a,b}+\beta^{n-p}_{a,b}\leq \beta^{p}_{a}.
  \end{equation*}
\end{itemize}
\end{theorem}
In this paper, we present a few examples to illustrate our ideas and results. The paper is organized as follows. In the next section, we give some preliminary knowledge and recall some classic results. In Section \ref{section:module}, we show the module structure of a Cayley persistence object. In Section \ref{section:algebra}, we study the product on the Cayley-persistence module. In the last section, we show Theorem \ref{theorem:main5}.

\section{Preliminaries}

\subsection{Persistent homology and cohomology}

Let $(X,\preceq)$ be an ordered set, where $\preceq\subseteq X\times X$ gives an order. We can regard $X$ as a category with elements in $X$ as objects and the pairs $(a,b)\in \preceq$ as morphisms. We denote the category by $\mathrm{cat}(X,\preceq)$.
\begin{definition}
Let $\mathfrak{C}$ be a category. A \emph{persistence object} is a functor $\mathcal{F}:\mathrm{cat}(X,\preceq)\rightarrow \mathfrak{C}$ from the category $\mathrm{cat}(X,\preceq)$ to $\mathfrak{C}$. Dually, a \emph{copersistence object} is a contravariant functor $\mathcal{F}:\mathrm{cat}(X,\preceq)\rightarrow \mathfrak{C}$ from the category $\mathrm{cat}(X,\preceq)$ to $\mathfrak{C}$.
\end{definition}
The ring of integers $\mathbb{Z}$ is an ordered set with the order given by $\leq$. Usually, a \emph{persistence module} is a functor $\mathcal{F}:\mathrm{cat}(\mathbb{Z},\leq)\rightarrow \mathbf{Vec}_{\mathbf{k}}$ from the category $\mathrm{cat}(\mathbb{Z},\leq)$ to the category of $\mathbf{k}$-linear spaces. Let $\mathcal{K}:\mathrm{cat}(\mathbb{Z},\leq)\rightarrow \mathbf{Simp}$ be a persistence simplicial complex, that is, a filtration of simplicial complexes satisfies
\begin{enumerate}
  \item[($i$)]  For each $i\in \mathbb{Z}$, $\mathcal{K}_{i}$ is a simplicial complex. For any integers $i\leq j$, there is a morphism of simplicial complexes $f_{i,j}:\mathcal{K}_{i}\rightarrow \mathcal{K}_{j}$.
  \item[($ii$)] For $i\leq j\leq k$, we have $f_{j,k}\circ f_{i,j}=f_{i,k}$.
\end{enumerate}
There are many ways to obtain a persistence simplicial complex. The Vietoris-Rips complexes and C\v{e}ch complexes are the classical constructions from a data-set \cite{carlsson2006algebraic,ghrist2008barcodes}. More generally, we give a simplicial complex $K$ equipped with a function $f:K\rightarrow \mathbb{R}$ on the simplices of $K$ such that $f(\sigma)\leq f(\tau)$ for any face $\sigma$ of $\tau$. Choosing a sequence of real numbers $a_{0}<a_{1}<\cdots<a_{p}<\cdots$, we have a filtration of simplicial complexes $\{K_{a_{p}}\}$ given by
\begin{equation*}
  K_{a_{p}}=\{\sigma\in K|f(\sigma)\leq a_{p}\}.
\end{equation*}
By taking $\mathcal{K}_{p}=K_{a_{p}}$, we obtain a persistence simplicial complex $\mathcal{K}:\mathrm{cat}(\mathbb{Z},\leq)\rightarrow \mathbf{Simp}$.

Note that the homology $H_{\ast}(-;\mathbf{k}):\mathbf{Simp}\rightarrow \mathbf{Vec}_{\mathbf{k}}$ is a functor from the category of simplicial complexes to the category of $\mathbf{k}$-linear spaces. One has that the functor
\begin{equation*}
  H_{\ast}(\mathcal{K};\mathbf{k}):\mathrm{cat}(\mathbb{Z},\leq)\rightarrow \mathbf{Vec}_{\mathbf{k}},\quad i\mapsto H_{\ast}(\mathcal{K}_{i};\mathbf{k})
\end{equation*}
is a persistence module. For $i\leq j$, the \emph{$(i,j)$-persistent homology} is defined by
\begin{equation*}
 H_{\ast}^{i,j}=\mathrm{im}(H_{\ast}(\mathcal{K}_{i};\mathbf{k})\rightarrow H_{\ast}(\mathcal{K}_{j};\mathbf{k})).
\end{equation*}
For the sake of simplicity, we denote $H_{\ast}^{i}=H_{\ast}(\mathcal{K}_{i};\mathbf{k})$.
Let $\mathbf{H}=\bigoplus\limits_{i\in \mathbb{Z}}H_{\ast}^{i}$ be a graded $\mathbf{k}$-linear space. The morphism $t:H_{\ast}(\mathcal{K}_{i};\mathbf{k})\rightarrow H_{\ast}(\mathcal{K}_{i+1};\mathbf{k})$ induces a morphism
\begin{equation*}
  t: \mathbf{H}\rightarrow \mathbf{H}
\end{equation*}
of degree $1$. Consider the polynomial ring $\mathbf{k}[t]$. For any $f(t)=\sum\limits_{k=0}^{n}a_{k}t^{k}\in \mathbf{k}[t]$, we have a morphism
\begin{equation*}
  f(t):\mathbf{H}\rightarrow \mathbf{H}
\end{equation*}
given by $f(t)(\alpha)=\sum\limits_{k=0}^{n}a_{k}\cdot\overbrace{t\circ t\circ\cdots\circ t}^{k}(\alpha)$. This shows that $\mathbf{H}$ is a graded left $\mathbf{k}[t]$-module given by
\begin{equation*}
  \mathbf{k}[t]\times \mathbf{H}\rightarrow \mathbf{H},\quad (f(t),\alpha)\mapsto f(t)(\alpha).
\end{equation*}
A. Zomorodian and G. Carlsson \cite{Zomorodian2005} proved that
\begin{theorem}
If $\mathbf{H}$ is a finitely generated $\mathbf{k}[t]$-module, then we have a finite direct decomposition
\begin{equation*}
  \mathbf{H}\cong \left(\bigoplus_{i=1}^{k}\mathbf{k}[t]\cdot e^{b_{i}}_{i}\right)\oplus\left(\bigoplus_{j=1}^{l}\frac{\mathbf{k}[t]}{\mathbf{k}[t]\cdot t^{s_{j}}}\cdot \varepsilon^{r_{j}}_{j}\right),
\end{equation*}
where $e^{b_{i}}_{i},\varepsilon^{r_{j}}_{j}$ are generators of degree $b_{i},r_{j}$, respectively.
\end{theorem}
Note that $\mathbf{k}[t]$ is a principal ideal domain (PID). The proof of this decomposition mainly depends on the structure theorem of finitely generated modules over a PID.

Dually, for $i\leq j$, the \emph{$(i,j)$-persistent cohomology} is defined by
\begin{equation*}
 H^{\ast}_{i,j}=\mathrm{im}(H^{\ast}(\mathcal{K}_{j};\mathbf{k})\rightarrow H^{\ast}(\mathcal{K}_{i};\mathbf{k})).
\end{equation*}
Let $H^{\ast}_{i}=H^{\ast}(\mathcal{K}_{i};\mathbf{k})$ and $\mathbf{H}^{\sharp}=\bigoplus\limits_{i\in \mathbb{Z}}H^{\ast}_{i}$. The morphism $x:H^{\ast}(\mathcal{K}_{i};\mathbf{k})\rightarrow H^{\ast}(\mathcal{K}_{i-1};\mathbf{k})$ induces a morphism $x: \mathbf{H}^{\sharp}\rightarrow \mathbf{H}^{\sharp}$ of degree $-1$. Then, for any $f(x)\in \mathbf{k}[x]$, we have a morphism
\begin{equation*}
  f(x):\mathbf{H}^{\sharp}\rightarrow \mathbf{H}^{\sharp},
\end{equation*}
which makes $\mathbf{H}^{\sharp}$ a graded right $\mathbf{k}[x]$-module  given by
\begin{equation*}
  \mathbf{H}^{\sharp}\times \mathbf{k}[x]\rightarrow \mathbf{H}^{\sharp},\quad (\alpha,f(x))\mapsto f(x)(\alpha).
\end{equation*}
Similarly, $\mathbf{H}^{\sharp}$ has a decomposition if it is a finitely generated $\mathbf{k}[x]$-module. In Section \ref{section:module}, we will study the module structures of persistent homology for group grading.


\subsection{The monoid ring}
In this section, we introduce the topological monoids \cite{bourbaki1966elements} which will be used to deal with the non-discrete group graded structure of the modules.

\begin{lemma}\label{lemma:topology}
Let $S,T$ be monoids and $\phi:S\rightarrow T$ be a morphism of monoids. If $S$ is a topological monoid, then the image $\phi(S)$ is a topological monoid with the quotient topology induced by $\phi$.
\end{lemma}
\begin{proof}
Define an equivalence relation $\mathcal{R}$ on $S$ by $x\mathcal{R}y$ if $\phi(x)=\phi(y)$. Then $\mathcal{R}$ is a  congruence relation for the product on $S$ since $x\mathcal{R}y,x'\mathcal{R}y'$ imply $(xx')\mathcal{R}(yy')$. Thus the quotient $S/\mathcal{R}$ is a monoid \cite{hungerford1974categories}. The map \begin{equation*}
  \bar{\phi}:S/\mathcal{R}\rightarrow T
\end{equation*}
induced by $\phi$ is defined as follows. For any $\bar{x}\in S$, we choose an element $x$ in the class $\bar{x}$ and define $\bar{\phi}(\bar{x})=\phi(x)$.
It can be directly verified that the map is well defined and a morphism of monoids. Then $\bar{\phi}:S/\mathcal{R}\rightarrow \phi(S)$ is a surjection. Given any pair $(K,i:K\to S/\mathcal{R})$ with $\bar{\phi}\circ i(a)=e_{T}$ for all $a\in K$, we have $\bar{\phi}\circ i(a)=\phi(e_{S})$.
\begin{equation*}
  \xymatrix{
 K\ar@{->}[r]^{i}\ar@{.>}[dr]_{1}& S/\mathcal{R}\ar@{->}[d]^{\bar{\phi}}\\
 &\phi(S)
 }
\end{equation*}
It follows that $i(a)=\bar{e}_{S}$ in $S/\mathcal{R}$, where $\bar{e}_{S}$ is the congruence class of $e_{S}$. One has that $\bar{\phi}$ is a monomorphism. Thus $\bar{\phi}:S/\mathcal{R}\rightarrow \phi(S)$ is a bijection. We endow $\phi(S)$ with a quotient topology induced by $\phi$. Since $S$ is a topological monoid, $\phi(S)$ is a topological monoid.
\end{proof}
\begin{remark}
Let $\mathbf{Monoid}$ be the category of monoids with objects given by monoids and morphisms given by morphisms of monoids. Here, a morphism of monoids is a map $\phi:S\rightarrow T$ satisfying
\begin{equation*}
  \phi(x)\phi(y)=\phi(xy),\quad \phi(e_{S})=\phi(e_{T}).
\end{equation*}
The kernel of $\phi$ is the kernel of $\phi$ in the category $\mathbf{Monoid}$. For other related researches on monoids, the readers can refer to \cite{fiedorowicz1984classifying,rhodes1989kernel}.
\end{remark}
Let $R$ be a commutative ring with unit, and let $S$ be a  monoid. We will recall the construction of the monoid ring \cite{gilmer1984commutative} given by $S$ over $R$. Let $R[S]$ be the set of maps from $S$ to $R$ of finite support in $S$. Recall that a map $\psi$ is of finite support if
the set
\begin{equation*}
  \mathrm{supp}(\psi)=\{z\in S|\psi(z)\neq 0\}
\end{equation*}
is finite. Then $R[S]$ is a ring  with addition
\begin{equation*}
  (\lambda\cdot\phi+ \mu\cdot\varphi)(x)=\lambda\phi(x)+\mu\varphi(x),\quad \lambda,\mu\in R,x\in S
\end{equation*}
and multiplication
\begin{equation*}
  (\phi\cdot \varphi)(x)=\sum_{x_{1}x_{2}=x}\phi(x_{1})\varphi(x_{2}),\quad x\in S.
\end{equation*}
It can be verified $(\phi\cdot \varphi)\cdot \psi=\phi\cdot (\varphi\cdot \psi)$.
\begin{remark}\label{remark:monoid}
Another equivalent definition of the monoid ring is given by the free module $\bigoplus\limits_{x\in S}Rx$ with multiplication given by the multiplication on $S$.  Indeed, let $\phi_{x}(z)=\left\{
                                                                   \begin{array}{ll}
                                                                     1, & \hbox{$z=x$;} \\
                                                                     0, & \hbox{otherwise.}
                                                                   \end{array}
                                                                 \right.
$ It can be verified that the map $\bigoplus\limits_{x\in S}Rx\rightarrow R[S],x\rightarrow \phi_{x}$ gives an isomorphism of rings.
\end{remark}

A monoid $S$ is \emph{finitely generated} if there is a finite set that generates $S$.

\begin{lemma}\label{lemma:noetherian}
Let $S,T$ be monoids and $\phi:S\rightarrow T$ be a morphism of monoids. If $S$ is finitely generated and $R$ is noetherian, then the monoid ring $R[\phi(S)]$ is noetherian.
\end{lemma}
\begin{proof}
If $S$ is finitely generated, there exists a finite set $X=\{x_{1},\dots,x_{n}\}$ generates $S$. Then $\phi(x_{1}),\dots,\phi(x_{n})$ are generators of $\phi(S)$. By \cite[Theorem 7.7]{gilmer1984commutative}, the monoid ring $R[\phi(S)]$ is noetherian.
\end{proof}

\section{The Cayley-persistence module}\label{section:module}

\subsection{Cayley grading and Cayley-persistence objects}
Let $G$ be a group. A ring $R$ is \emph{$G$-graded} \cite{bahturin2001group} if there is a decomposition of $R$ as an additive group into the direct sum of subgroups
\begin{equation*}
  R=\bigoplus_{g\in G}R_{g}
\end{equation*}
such that $R_{g}R_{h}\subseteq R_{gh}$ for any $g,h\in G$.  A (left) $R$-module $M$ is $G$-graded if there is a direct sum decomposition
\begin{equation*}
  M=\bigoplus_{g\in G}M_{g}
\end{equation*}
 of $M$ such that $R_{g}M_{h}\subseteq M_{gh}$.
\begin{example}
\begin{itemize}
  \item[$(i)$] The usual $\mathbb{Z}^{n}$-graded ring is a group graded ring with group $G=\mathbb{Z}^{n}$.
  \item[$(ii)$] A $\mathbb{Z}/2$-graded ring is a group graded ring with group $G=\mathbb{Z}/2$.
  \item[$(iii)$] Let $R$ be a commutative ring with unit, then $R[G]$ is a group ring. There is a natural $G$-grading on $R[G]$.
\end{itemize}
\end{example}
A  digraph $D=(V,E)$ can be regarded as a category with
\begin{itemize}
  \item[($i$)] objects: the vertices in $V$;
  \item[($ii$)] morphisms: for every $x,y$ in $V$,
  \begin{equation*}
    \mathrm{Hom}(x,y)=\left\{
               \begin{array}{ll}
                \{f_{xy}:x\rightarrow y\}, & \hbox{if there exists a directed path from $x$ to $y$;} \\
                 \emptyset, & \hbox{otherwise.}
               \end{array}
             \right.
  \end{equation*}
\end{itemize}
We denote $\mathrm{cat}(D)$ the category given by the digraph $D$. Let $G$ be a group and $S$ be a subset of $G$. We have a Cayley digraph $\mathrm{Cay(}G,S)$.
\begin{lemma}\label{lemma:monoid}
Let $\langle S \rangle$ be the semigroup generated by $S$ in $G$. Then $\mathrm{cat}(\mathrm{Cay}(G,S))=\mathrm{cat}(\mathrm{Cay}(G,\langle S \rangle))$.
\end{lemma}
\begin{proof}
Obviously, $\mathrm{cat}(\mathrm{Cay}(G,S))$ is a subcategory of $\mathrm{cat}(\mathrm{Cay}(G,\langle S \rangle))$ with objects given by $G$. For a morphism $f:a\rightarrow b$ in $\mathrm{cat}(\mathrm{Cay}(G,\langle S \rangle))$, we have $ba^{-1}\in \langle S \rangle$. It follows that $ba^{-1}=s_{1}s_{2}\cdots s_{n}$ for some $s_{1},\dots,s_{n}\in S$. Thus we have a directed path
\begin{equation*}
  a\rightarrow s_{n}a \rightarrow s_{n-1}s_{n}a \rightarrow\cdots \rightarrow s_{2}\cdots s_{n-1}s_{n}a \rightarrow b
\end{equation*}
in $\mathrm{Cay}(G,S)$. It follows that $f\in \mathrm{cat}(\mathrm{Cay}(G,S))$.
\end{proof}

From now on, we always assume that $S$ is a monoid with the identity element $e$.
\begin{definition}
Let $\mathrm{Cay}(G,S)$ be a Cayley digraph, and let $\mathfrak{C}$ be a category. A \emph{Cayley-persistence object} is a functor $\mathcal{F}:\mathrm{cat}(\mathrm{Cay}(G,S))\rightarrow \mathfrak{C}$ from the category  $\mathrm{cat}(\mathrm{Cay}(G,S))$ to the category $\mathfrak{C}$. Dually, the \emph{Cayley-copersistence object} is a contravariant functor $\mathcal{F}:\mathrm{cat}(\mathrm{Cay}(G,S))\rightarrow \mathfrak{C}$ from the category  $\mathrm{cat}(\mathrm{Cay}(G,S))$ to the category $\mathfrak{C}$.
\end{definition}
The Cayley-persistence object $\mathcal{F}:\mathrm{cat}(\mathrm{Cay}(G,S))\rightarrow \mathfrak{C}$ can be also written as follows.
\begin{enumerate}
  \item[($i$)] $\mathcal{F}_{a}\in \mathfrak{C}$ for $a\in G$. For $ba^{-1}\in S$, there is a morphism $f_{a,b}:\mathcal{F}_{a}\rightarrow \mathcal{F}_{b}$ in $\mathfrak{C}$.
  \item[($ii$)] For $ba^{-1},cb^{-1}\in S$, we have $f_{b,c}\circ f_{a,b}=f_{a,c}$.
\end{enumerate}
Let $\mathcal{F}:\mathrm{cat}(\mathrm{Cay}(G,S))\rightarrow \mathfrak{C}$ be a Cayley-persistence object. Given $x\in S$, we have a $G$-graded morphism $l_{x}:a\rightarrow xa$ for each $a\in G$. Then we have $l_{x}l_{y}=l_{xy}$, which implies that
\begin{equation*}
  \mathcal{F}(l_{x})\mathcal{F}(l_{y})=\mathcal{F}(l_{xy})
\end{equation*}
by the functorial property. Let $\mathcal{L}_{S}$ be the monoid given by the elements $\mathcal{F}(l_{x}),x\in S$ with the multiplication $\mathcal{F}(l_{x})\cdot\mathcal{F}(l_{y})=\mathcal{F}(l_{x})\mathcal{F}(l_{y})$ and the identity element $\mathcal{F}(l_{e})$. Then the monoid ring $R_{\mathcal{F}}=\mathbf{k}[\mathcal{L}_{S}]$ is a $G$-graded ring with the grading given by $\deg \mathcal{F}(l_{x})=x$. Let $M_{\mathcal{F}}=\bigoplus\limits_{x\in G}\bigoplus\limits_{\alpha\in \mathcal{F}_{x}}\mathbf{k}\alpha$ be a $\mathbf{k}$-linear space generated by the elements in $\mathcal{F}_{x},x\in G$. The elements in $\mathcal{F}_{x}$ are endowed with a $G$-grading $x$. Moreover, we have a direct sum decomposition
\begin{equation*}
  M_{\mathcal{F}}=\bigoplus_{x\in G}\mathbf{k}[\mathcal{F}_{x}],
\end{equation*}
where $\mathbf{k}[\mathcal{F}_{x}]$ is the $\mathbf{k}$-linear space generated by the elements in $\mathcal{F}_{x}$. Note that
\begin{equation*}
  \mathcal{F}(l_{x})\mathcal{F}_{y}\subseteq \mathcal{F}_{xy},
\end{equation*}
one has that $M_{\mathcal{F}}$ is a $G$-graded $R_{\mathcal{F}}$-module. In a word, we can obtain a group graded module from a Cayley-persistence persistence object.

\begin{example}\label{exam:Cayley}
\begin{itemize}
  \item[$(i)$] Let $e_{1},\dots,e_{n}$ be a standard basis of  $\mathbb{R}^{n}$, where $e_{i}=(0,\dots,0, \overset{i}{1},0,\dots,0)$. Let $S=\mathbb{Z}_{\geq 0}^{n}$ be the monoid generated by $e_{1},\dots,e_{n}$. Then $(\mathbb{Z}^{n},S)$ is a Cayley digraph. A functor $\mathcal{F}:\mathrm{cat}(\mathrm{Cay}(G,S))\rightarrow \mathbf{Sets}$ from $\mathrm{cat}(\mathrm{Cay}(G,S))$ to the category of sets gives a multi-graded structure on sets. Let
      \begin{equation*}
      l_{e_{i}}:\mathbb{Z}^{n}\rightarrow \mathbb{Z}^{n},\quad v\mapsto e_{i}+v.
      \end{equation*}
   Consider the morphism $L_{i}=\mathcal{F}(l_{e_{i}}):M_{\mathcal{F}}\rightarrow M_{\mathcal{F}}$
  given by $L_{i}m=\mathcal{F}(l_{e_{i}})(m)\in \mathcal{F}_{e_{i}+x}$ for $m\in \mathcal{F}_{x}$. Since $L_{i}L_{j}=L_{j}L_{i}$, one has that $\mathbf{k}[L_{1},\dots,L_{n}]$ is a polynomial ring and $M_{\mathcal{F}}$ is a $\mathbb{Z}^{n}$-graded $\mathbf{k}[L_{1},\dots,L_{n}]$-module.

  \item[$(ii)$]  Let $G=S=\mathbb{Z}/p$. Then $(\mathbb{Z}/p,\mathbb{Z}/p)$ is a Cayley digraph. A functor $$\mathcal{F}:\mathrm{cat}(\mathrm{Cay}(\mathbb{Z}/p,\mathbb{Z}/p))\rightarrow \mathbf{Top}$$ from the category $\mathrm{cat}(\mathrm{Cay}(\mathbb{Z}/p,\mathbb{Z}/p))$ to the category of topological spaces gives a family of topological spaces $\mathcal{F}_{i},i=0,1,\dots,p-1$ which are homeomorphic to each other. Note that $Z/p$ gives a group action on $\mathcal{F}_{\mathbb{Z}/p}=\{\mathcal{F}_i\}_{i\in \mathbb{Z}/p}$ and $\mathrm{Aut}(\mathcal{F}_{\mathbb{Z}/p})=\mathbb{Z}/p$. Consider the map $\mathcal{F}(l_{j}):M_{\mathcal{F}}\rightarrow M_{\mathcal{F}}$ obtained by the homeomorphisms $\mathcal{F}(l_{j}):\mathcal{F}_{i}\rightarrow \mathcal{F}_{i+j}$ for $i\in \mathbb{Z}/p$. Let $\mathcal{L}_{\mathbb{Z}/p}$ be the monoid generated by $\mathcal{F}(l_{0}),\mathcal{F}(l_{j}),\dots,\mathcal{F}(l_{p-1})$. Then $M_{\mathcal{F}}=\bigoplus\limits_{i=0}^{p-1}\mathbf{k}[\mathcal{F}_{i}]$ is a $\mathbb{Z}/p$-graded  $\mathbf{k}[\mathcal{L}_{\mathbb{Z}/p}]$-module. Here, $\mathbf{k}[\mathcal{F}_{i}]$ is a $\mathbf{k}$-linear space generated by the elements in $\mathcal{F}_{i}$.
  \item[$(iii)$]  Let $G=F[x_{1},\dots,x_{n}]$ be a free group generated by $x_{1},\dots,x_{n}$. Let $S=F^{+}[x_{1},\dots,x_{n}]$ be the free monoid generated by  $x_{1},\dots,x_{n}$. Then $(G,S)$ is a Cayley digraph. Let $\mathcal{F}:\mathrm{cat}(\mathrm{Cay}(G,S))\rightarrow \mathfrak{C}$ be a Cayley-persistence object. Denote
 \begin{equation*}
  l_{u}:F[x_{1},\dots,x_{n}]\rightarrow F[x_{1},\dots,x_{n}],\quad l_{u}v=uv.
 \end{equation*}
 One has $\mathcal{F}(l_{u})\mathcal{F}(l_v)=\mathcal{F}(l_{uv})$  for $u,v\in S$. Then the monoid $\mathcal{L}_{S}$ is generated by $\mathcal{F}(l_{x_{1}}),\mathcal{F}(l_{x_{2}}),\dots,\mathcal{F}(l_{x_{n}})$. Moreover, $M_{\mathcal{F}}=\bigoplus\limits_{x\in G}\mathbf{k}[\mathcal{F}_{x}]$ is an $F[x_{1},\dots,x_{n}]$-graded $\mathbf{k}[\mathcal{L}_{S}]$-module.
\end{itemize}
\end{example}

\subsection{The structure of $M_{\mathcal{F}}$}
With the notions before, in this section, we give the structure of $M_{\mathcal{F}}$ as an $R_{\mathcal{F}}$-module. The module structure is essentially from the monoid action of $\mathcal{L}_{S}$ on $M_{\mathcal{F}}$.

\begin{theorem}\label{thm:module_decomposition}
Let $\mathcal{F}:\mathrm{cat}(\mathrm{Cay}(G,S))\rightarrow \mathfrak{C}$ be a Cayley-persistence object.
If $M_{\mathcal{F}}$ is a finitely generated $R_{\mathcal{F}}$-module, then we have a finite direct sum decomposition
\begin{equation*}
  M_{\mathcal{F}}\cong\bigoplus_{i=1}^{k} R_{\mathcal{F}}\cdot e_{i}^{x_{i}}\oplus \left(\bigoplus_{j=1}^{l}R_{\mathcal{F}}\cdot\varepsilon_{j}^{y_{j}}\right)/N
\end{equation*}
for some $k,l$, where $e_{i}^{x_{i}}\in \mathcal{F}_{x_{i}},\varepsilon_{j}^{y_{j}}\in \mathcal{F}_{y_{j}}$ and $N$ is an $R_{\mathcal{F}}$-module generated by the elements of the form $\mathcal{F}(l_{y})\varepsilon_{t}^{y_{t}}-\mathcal{F}(l_{yy_{t}y_{s}^{-1}})\varepsilon_{s}^{y_{s}}$ for some $1\leq s,t\leq l,y\in S$. Moreover, if $S$ is a finitely generated monoid, then $N$ is a finitely generated $R_{\mathcal{F}}$-module.
\end{theorem}
\begin{proof}
Since $M_{\mathcal{F}}$ is a finitely generated $R_{\mathcal{F}}$-module, there exists a finite generating set of $M_{\mathcal{F}}$. Let
\begin{equation*}
  S_{gen}=\{s_{1},\dots,s_{p}\},\quad s_{1},\dots,s_{p}\in M_{\mathcal{F}}
\end{equation*}
be a finite generating set. Since the elements in $\mathcal{F}_{x},x\in G$ give a Hamel basis of the $\mathbf{k}$-linear space $M_{\mathcal{F}}$, we can find  a $\mathbf{k}$-linear basis $m_{1}^{x_{1}},\dots,m_{n}^{x_{n}}$ for $S_{gen}$, where $m_{i}^{x_{i}}\in \mathcal{F}_{x_{i}}$ for $i=1,\dots,n$.
One has
\begin{equation*}
  M_{\mathcal{F}}\subseteq R_{\mathcal{F}}m_{1}^{x_{1}}+\cdots+R_{\mathcal{F}}m_{n}^{x_{n}}.
\end{equation*}
Considering the $a$-graded component of $M_{\mathcal{F}}$, we have
\begin{equation*}
  M_{\mathcal{F}}^{a}\subseteq \sum_{ax_{i}^{-1}\in S} \mathbf{k}\cdot \mathcal{F}(l_{ax_{i}^{-1}})m_{i}^{x_{i}}.
\end{equation*}
This shows that $M_{\mathcal{F}}^{a}$ is a finite dimensional $\mathbf{k}$-linear space. Let $E=\bigcup\limits_{a\in G}E_{a}$ be a $G$-graded set, where $E_{a}=\{\mathcal{F}(l_{ax_{i}^{-1}})m_{i}^{x_{i}}\}_{ax_{i}^{-1}\in S}$ is a finite set. Then $M_{\mathcal{F}}^{a}\subseteq \mathbf{k}[E_{a}]$, where $\mathbf{k}[E_{a}]$ is the $\mathbf{k}$-linear space generated by $E_{a}$. For each element $\mathcal{F}(l_{ax_{i}^{-1}})m_{i}^{x_{i}}\in E_{a}$, one has $ax_{i}^{-1}\in S$ and
\begin{equation*}
  \mathcal{F}(l_{z})\mathcal{F}(l_{ax_{i}^{-1}})m_{i}^{x_{i}}=\mathcal{F}(l_{zax_{i}^{-1}})m_{i}^{x_{i}}\in E_{za}
\end{equation*}
since $zax_{i}^{-1}\in S$. Let $\mathcal{L}_{S}$ be the monoid given by the elements $\mathcal{F}(l_{x}),x\in S$.
Thus there is a monoid action of $\mathcal{L}_S$ on $E$ given by
\begin{equation*}
 \mathcal{F}(l_{x})\cdot e_{a} =\mathcal{F}(l_{x})e_{a},\quad x\in S,e_{a}\in E.
\end{equation*}
Let $\mathcal{L}_{S}^{\ast}$ be the set given by all the elements $\mathcal{F}(l_{x})$ such that $x,x^{-1}\in S$. It can be verified $\mathcal{L}_{S}^{\ast}$ is a group.
Let $\mathcal{L}_{S}^{+}$ be the complementary set of $\mathcal{L}_{S}^{\ast}$ in $\mathcal{L}_{S}$. It is obvious that $\mathcal{L}_{S}^{+}$ is a semigroup. We assert that there is a group action of $\mathcal{L}_{S}^{\ast}$ on $E\backslash\mathcal{L}_{S}^{+}E$. Indeed, if $e_{a}\in E\backslash\mathcal{L}_{S}^{+}E$ and the element $x\in S$ is invertible, it suffices to show
\begin{equation*}
  \mathcal{F}(l_{x})e_{a}\in E\backslash\mathcal{L}_{S}^{+}E.
\end{equation*}
Assume that $\mathcal{F}(l_{x})e_{a}\in \mathcal{L}_{S}^{+}E$. Then $e_{a}=\mathcal{F}(l_{x^{-1}})\mathcal{F}(l_{x})e_{a}\in \mathcal{L}_{S}^{+}E$, a contradiction. Moreover, each element in $E\backslash\mathcal{L}_{S}^{+}E$ must be of the form $\mathcal{F}(l_{z})m_{i}^{x_{i}}$ for some $m_{i}^{x_{i}}$ and invertible $z$ in $S$, otherwise, $z$ would be not invertible in $S$ and $\mathcal{F}(l_{z})m_{i}^{x_{i}}\in \mathcal{L}_{S}^{+}E$.
Now, there is a decomposition of $E\backslash\mathcal{L}_{S}^{+}E$ into different orbits of $\mathcal{L}_{S}^{\ast}$-action. Each orbit is given by $\mathcal{L}_{S}^{\ast}m_{i}^{x_{i}}$ for some $m_{i}^{x_{i}}$.
For each orbit in $E\backslash\mathcal{L}_{S}^{+}E$, we choose an element $m_{i}^{x_{i}}$ in $E\backslash\mathcal{L}_{S}^{+}E$. Let $B$ be the set of the chosen elements. Then $B$ is finite with the number of elements
\begin{equation*}
  |B|\leq n.
\end{equation*}
In particular, if the identity element is the unique invertible element in $S$, then $|B|=|E\backslash\mathcal{L}_{S}^{+}E|$. Let
\begin{equation*}
  B_{free}=\{b\in B|\mathcal{L}_{S}^{+}b\cap \mathcal{L}_{S}^{+}(B\backslash\{b\})=\emptyset\},\quad B_{tor}=B\backslash B_{free}.
\end{equation*}
We write  $k=|B_{free}|$ and $l=|B_{tor}|$. Then we list the elements in $B$ as
\begin{equation*}
  b_{1}^{x_{1}},\dots,b_{k}^{x_{k}}\in B_{free},\quad \beta_{1}^{y_{1}},\dots,\beta_{l}^{y_{l}}\in B_{tor}.
\end{equation*}
Consider the morphism
\begin{equation*}
  \pi:\bigoplus_{i=1}^{k} R_{\mathcal{F}}\cdot e_{i}^{x_{i}}\oplus \bigoplus_{j=1}^{l}R_{\mathcal{F}}\cdot\varepsilon_{j}^{y_{j}}\longrightarrow M_{\mathcal{F}}
\end{equation*}
from a free $R_{\mathcal{F}}$-mudule to $M_{\mathcal{F}}$ given by
\begin{equation*}
  \pi(e_{i}^{x_{i}})=b_{i}^{x_{i}},\quad  \pi(\varepsilon_{j}^{x_{j}})=\beta_{j}^{x_{j}},\quad i=1,\dots,k,j=1,\dots,l.
\end{equation*}
By the definition of orbit space, we have $E\backslash\mathcal{L}_{S}^{+}E=\mathcal{L}_{S}^{\ast}B$. Note that
\begin{equation*}
  E\subseteq \mathcal{L}_{S}(E\backslash\mathcal{L}_{S}^{+}E)\subseteq E,
\end{equation*}
we obtain $\mathcal{L}_{S}\mathcal{L}_{S}^{\ast}B=E$, thus we have $\mathcal{L}_{S}B=E$. It follows that
\begin{equation*}
  M_{\mathcal{F}}\subseteq \mathbf{k}[E]= \mathbf{k}[\mathcal{L}_{S}B]\subseteq R_{\mathcal{F}}[B].
\end{equation*}
Here, $R_{\mathcal{F}}[B]$ denotes the $R_{\mathcal{F}}$-modules generated by the elements in $B$, that is, $R_{\mathcal{F}}[B]=\sum\limits_{b\in B}R_{\mathcal{F}}b$. Hence, each element in $M_{\mathcal{F}}$ can be written in the form
\begin{equation*}
 m= \sum_{i=1}^{k}\mathcal{F}(l_{u_{i}})b_{i}^{x_{i}}+\sum_{j=1}^{l}\mathcal{F}(l_{v_{j}})\beta_{j}^{y_{j}},\quad u_{i},v_{j}\in S.
\end{equation*}
Then $m$ has a preimage $\sum_{i=1}^{k}\mathcal{F}(l_{u_{i}})e_{i}^{x_{i}}+\sum_{j=1}^{l}\mathcal{F}(l_{v_{j}})\varepsilon_{j}^{y_{j}}$ in $\bigoplus_{i=1}^{k} R_{\mathcal{F}}\cdot e_{i}^{x_{i}}\oplus \bigoplus_{j=1}^{l}R_{\mathcal{F}}\cdot\varepsilon_{j}^{y_{j}}$, which shows that $\pi$ is a surjection. Next, we will show that $\mathcal{F}(l_{y})b_{t}^{x_{t}}$ is different from all the elements of the following forms
\begin{equation*}
 \mathcal{F}(l_{x})b_{i}^{x_{i}}, \mathcal{F}(l_{x})\beta_{j}^{x_{j}},\quad (i,x)\neq (t,y).
\end{equation*}
Suppose that $\mathcal{F}(l_{y})b_{t}^{x_{t}}=\mathcal{F}(l_{z})c_{s}^{x_{s}}$, where $c_{s}^{x_{s}}$ denotes $b_{i}^{x_{i}}$ or $\beta_{j}^{y_{j}}$.
\begin{itemize}
  \item[($i$)] When one of $y,z$ is invertible in $S$. We may assume $z$ is invertible. Then
  \begin{equation*}
  \mathcal{F}(l_{z^{-1}y})b_{t}^{x_{t}}=c_{s}^{x_{s}}\in B\subseteq E\backslash\mathcal{L}_{S}^{+}E,
  \end{equation*}
which implies that $z^{-1}y$ is invertible and $b_{t}^{x_{t}},c_{s}^{x_{s}}$ are in the same orbit of $\mathcal{L}_{S}^{\ast}$-action on $E\backslash\mathcal{L}_{S}^{+}E$. Thus we have $b_{t}^{x_{t}}=c_{s}^{x_{s}}$.
  \item[($ii$)] When $y,z$ are not invertible in $S$. Since $\mathcal{L}_{S}^{+}b_{t}^{x_{t}}\cap \mathcal{L}_{S}^{+}(B\backslash\{b_{t}^{x_{t}}\})=\emptyset$, one has $b_{t}^{x_{t}}=c_{s}^{x_{s}}$.

\end{itemize}
If $\mathcal{F}(l_{y})b_{t}^{x_{t}}=\mathcal{F}(l_{z})b_{t}^{x_{t}}$, then $yx_{t}=zx_{t}$ by the degree reasons. It follows that $y=z$ by the multiplication elimination law of groups. In a word, we obtain
\begin{equation*}
  \pi\left(\bigoplus_{i=1}^{k} R_{\mathcal{F}}\cdot e_{i}^{x_{i}}\oplus \bigoplus_{j=1}^{l}R_{\mathcal{F}}\cdot\varepsilon_{j}^{y_{j}}\right)=\bigoplus_{i=1}^{k} R_{\mathcal{F}}\cdot b_{i}^{x_{i}}\oplus \pi\left(\bigoplus_{j=1}^{l}R_{\mathcal{F}}\cdot\varepsilon_{j}^{y_{j}}\right).
\end{equation*}
Note that the only $\mathbf{k}$-linear relationship of the elements
\begin{equation*}
 \mathcal{F}(l_{x})b_{i}^{x_{i}}, \mathcal{F}(l_{x})\beta_{j}^{y_{j}},\quad i=1,\dots,k,j=1,\dots,l,x\in S
\end{equation*}
is appeared in $\mathcal{F}(l_{x})\beta_{j}^{y_{j}},j=1,\dots,l,x\in S$ whenever $\mathcal{F}(l_{y})\beta_{t}^{y_{t}}=\mathcal{F}(l_{z})\beta_{s}^{y_{s}}$ for some $1\leq s,t\leq l,y,z\in S$. By the degree reasons, one has $z=yy_{t}y_{s}^{-1}$. Thus the kernel of $\pi$ is generated by
\begin{equation*}
  \mathcal{F}(l_{y})\varepsilon_{t}^{y_{t}}-\mathcal{F}(l_{yy_{t}y_{s}^{-1}})\varepsilon_{s}^{y_{s}},
\end{equation*}
where $s,t,y$ are taken such that $\mathcal{F}(l_{y})\beta_{t}^{y_{t}}=\mathcal{F}(l_{yy_{t}y_{s}^{-1}})\beta_{s}^{y_{s}}$. Denote $N$ the $R_{\mathcal{F}}$-module generated by $\{\mathcal{F}(l_{y})\varepsilon_{t}^{y_{t}}-\mathcal{F}(l_{yy_{t}y_{s}^{-1}})\varepsilon_{s}^{y_{s}}\}_{s,t,y}$ for all the indexes $s,t,y$ satisfying $\mathcal{F}(l_{y})\beta_{t}^{y_{t}}=\mathcal{F}(l_{yy_{t}y_{s}^{-1}})\beta_{s}^{y_{s}}$. Then we have an $R_{\mathcal{F}}$-module isomorphism
\begin{equation*}
  \bigoplus_{i=1}^{k} R_{\mathcal{F}}\cdot e_{i}^{x_{i}}\oplus \left(\bigoplus_{j=1}^{l}R_{\mathcal{F}}\cdot\varepsilon_{j}^{y_{j}}\right)/N\stackrel{\cong}{\longrightarrow} M_{\mathcal{F}}.
\end{equation*}
If $S$ is finitely generated, by Lemma \ref{lemma:noetherian}, the monoid ring $\mathbf{k}[\mathcal{L}_{S}]$ is noetherian. Thus $N$ is finitely generated submodule of the free $\mathcal{R}_{\mathcal{F}}$-module $\bigoplus_{j=1}^{l}R_{\mathcal{F}}\cdot\varepsilon_{j}^{y_{j}}$.
\end{proof}

\begin{remark}
The condition ``finitely generated'' is very important and interesting. Indeed, the condition that $M_{\mathcal{F}}$ is a finitely generated $R_{\mathcal{F}}$-module ensures the existence of the ``born time'' of the generators. The condition that $S$ is a finitely generated monoid makes it possible for us to obtain the finiteness of the time of disappearance and meeting of generators, that is, $N$ is a finitely generated $R_{\mathcal{F}}$-module.
\end{remark}

\begin{remark}\label{remark:meet}
Note that $e_{i}^{x_{i}}\in \mathcal{F}_{x_{1}},\varepsilon_{j}^{y_{j}}\in \mathcal{F}_{y_{j}}$ represent the elements coming into being at $x_{i},y_{j}$, respectively. A generator $\mathcal{F}(l_{y})\varepsilon_{t}^{y_{t}}-\mathcal{F}(l_{yy_{t}y_{s}^{-1}})\varepsilon_{s}^{y_{s}}$ of $N$ shows the information that two elements $\varepsilon_{t}^{y_{t}},\varepsilon_{s}^{y_{s}}$ become the same one at the time $yy_{t}$. We denote $y_{st}=yy_{t}$. Then the information of $N$ can be represented by all the triples $(\varepsilon_{s}^{y_{s}},\varepsilon_{t}^{y_{t}},y_{st})$, which means the two generators appearing at $y_{s},y_{t}$ and meeting at $y_{st}$.
\end{remark}
\begin{example}\label{example:finitely}
Example \ref{exam:Cayley}(i) continued. In view of Hilbert's basis theorem, the polynomial ring $k[L_{1},\dots,L_{n}]$ is noetherian, the $\mathcal{R}_{\mathcal{F}}$-module generated by $\{\mathcal{F}(l_{y})\varepsilon_{t}^{y_{t}}-\mathcal{F}(l_{yy_{t}y_{s}^{-1}})\varepsilon_{s}^{y_{s}}\}_{s,t,y}$ in Theorem \ref{thm:module_decomposition} is finitely generated. However, if $S$ is not finitely generated, $N$ does not have to be finitely generated even if $N$ is abelian. For example, let $\mathcal{F}:(\mathbb{R}^{2},\mathbb{R}_{\geq 0}^{2})\rightarrow \mathbf{Sets}$ be a Cayley-persistence set, where $\mathbb{R}_{\geq 0}=\{x\in \mathbb{R} |x\geq 0\}$. We set $X_{v}=\{v\}$ and $X_{w}=\{w,w'\}$.
\begin{figure}[H]
\centering
    \begin{tikzpicture}[scale=1.5]
        \draw[step=.5cm,gray,ultra thin, dashed] (-0.9 ,-0.9) grid (1.9,1.9);%
        \begin{scope}[thick]

            \draw[->] (-1,0) -- (2,0);
            \draw[->] (0,-1) -- (0,2);
        \end{scope}
\draw[domain=0.5:2,thick] plot(\x,{1/\x})node[right]{$y=\frac{4}{x}$};
\draw[thick](0.5,0.5)--(2.05,0.5);
\draw[thick](0.5,0.5)--(0.5,2.1);
\node [font=\fontsize{8}{6}] (node001) at (0.1,-0.1){$O$};
\node [font=\fontsize{8}{6}] (node001) at (2,-0.1){$x$};
\node [font=\fontsize{8}{6}] (node001) at (0.5,-0.1){$1$};
\node [font=\fontsize{8}{6}] (node001) at (1,-0.1){$2$};
\node [font=\fontsize{8}{6}] (node001) at (1.5,-0.1){$3$};
\node [font=\fontsize{8}{6}] (node001) at (-0.5,-0.1){$-1$};

\node [font=\fontsize{8}{6}] (node001) at (0.1,2){$y$};
\node [font=\fontsize{8}{6}] (node001) at (-0.5,-0.5){$\emptyset$};
\node [font=\fontsize{8}{6}] (node001) at (-0.5,0.5){$\emptyset$};
\node [font=\fontsize{8}{6}] (node001) at (0.5,-0.5){$\emptyset$};
\node [font=\fontsize{8}{6}] (node001) at (0.9,0.7){$C_{2}$};
\node [font=\fontsize{8}{6}] (node001) at (0.45,0.25){$\mathbb{R}_{\geq 0}^{2}\backslash C_{1}$};
\node [font=\fontsize{8}{6}] (node001) at (1.2,1.2){$C_{1}\backslash C_{2}$};
        \end{tikzpicture}\qquad
            \begin{tikzpicture}[scale=1.5]
        \draw[step=.5cm,gray,ultra thin, dashed] (-0.9 ,-0.9) grid (1.9,1.9);%
        \begin{scope}[thick]

            \draw[->] (-1,0) -- (2,0);
            \draw[->] (0,-1) -- (0,2);
        \end{scope}
\draw[domain=0.5:2,thick] plot(\x,{1/\x})node[right]{$y=\frac{4}{x}$};
\draw[thick](0.5,0.5)--(2.05,0.5);
\draw[thick](0.5,0.5)--(0.5,2.1);
\draw[black!60,fill=black!10](2.05,1.25)--(0.8,1.25)--(0.8,2.1)--(2.05,2.1)--(2.05,1.25);\draw[very thick,black!0](0.8,2.1)--(2.05,2.1)--(2.05,1.25);
\node [font=\fontsize{8}{6}] (node001) at (0.7,1.15){$Q$};\node [font=\fontsize{8}{6}] (node001) at (1.3,1.7){$R_{\mathcal{F}}\cdot e_{w'}^{Q}$};

\node [font=\fontsize{8}{6}] (node001) at (0.1,-0.1){$O$};
\node [font=\fontsize{8}{6}] (node001) at (2,-0.1){$x$};
\node [font=\fontsize{8}{6}] (node001) at (0.5,-0.1){$1$};
\node [font=\fontsize{8}{6}] (node001) at (1,-0.1){$2$};
\node [font=\fontsize{8}{6}] (node001) at (1.5,-0.1){$3$};
\node [font=\fontsize{8}{6}] (node001) at (-0.5,-0.1){$-1$};

\node [font=\fontsize{8}{6}] (node001) at (0.1,2){$y$};
\node [font=\fontsize{8}{6}] (node001) at (-0.5,-0.5){$\emptyset$};
\node [font=\fontsize{8}{6}] (node001) at (-0.5,0.5){$\emptyset$};
\node [font=\fontsize{8}{6}] (node001) at (0.5,-0.5){$\emptyset$};
\node [font=\fontsize{8}{6}] (node001) at (0.9,0.7){$w,w'$};
\node [font=\fontsize{8}{6}] (node001) at (0.35,0.25){$v$};
\node [font=\fontsize{8}{6}] (node001) at (1.2,1.1){$w'$};
        \end{tikzpicture}
\end{figure}
Let $C_{1}=\{(x,y)|x,y\geq 1\},C_{2}=\{(x,y)|xy\leq 4,1\leq x\leq 4\}$. The functor $\mathcal{F}$ is given by
\begin{equation*}
  \mathcal{F}_{P}=\left\{
                \begin{array}{ll}
                  X_{v}, & \hbox{$P\in \mathbb{R}_{\geq 0}^{2}\backslash C_{1}$;} \\
                  X_{w}, & \hbox{$P\in C_{1}$;} \\
                  \emptyset, & \hbox{otherwise.}
                \end{array}
              \right.
\end{equation*}
and
\begin{equation*}
  \mathcal{F}_{P\to Q}=\left\{
                \begin{array}{ll}
                  \mathrm{id}, & \hbox{$P,Q\in \mathbb{R}_{\geq 0}^{2}\backslash C_{1}$ or $P,Q\in C_{2}$;} \\
                  c_{w}, & \hbox{$P\in \mathbb{R}_{\geq 0}^{2}\backslash C_{1}$ and $Q\in C_{2}$;} \\
                  c_{w'}, & \hbox{$P\in \mathbb{R}_{\geq 0}^{2}$ and $Q\in C_{1}\backslash C_{2}$;} \\
                  \emptyset, & \hbox{otherwise.}
                \end{array}
              \right.
\end{equation*}
Here, $c_{w},c_{w'}$ are constant maps.

The generator $v$ begins to appear at $O=(0,0)$ while $w,w'$ begins to appear at $(1,1)$. The generators $v$ and $w$ begin to meet at $(1,1)$. The generators $w$ and $w'$ begin to meet at the points on $\{(x,y)|y=4/x,1\leq x\leq 4\}$. More precisely, for $Q\in C_{1}\backslash C_{2}$, we have
\begin{equation*}
  \mathcal{F}_{(0,0)\to Q}v=\mathcal{F}_{(1,1)\to Q}w=\mathcal{F}_{(1,1)\to Q}w'=w'.
\end{equation*}
A straightforward calculation shows that
\begin{equation*}
  M_{\mathcal{F}}\cong \frac{R_{\mathcal{F}}\cdot e^{(0,0)}_{v}\oplus  R_{\mathcal{F}}\cdot e^{(1,1)}_{w}\oplus  R_{\mathcal{F}}\cdot e^{(1,1)}_{w'}}{N},
\end{equation*}
where $N$ is an $R_{\mathcal{F}}$-module generated by $\mathcal{F}_{(0,0)\to (1,1)}e_{v}^{(0,0)}-e_{w}^{(1,1)}$, $\mathcal{F}_{(0,0)\to Q}e_{v}^{(0,0)}-\mathcal{F}_{(1,1)\to Q}e_{w'}^{(1,1)}$ and $\mathcal{F}_{(1,1)\to Q}e_{w}^{(1,1)}-\mathcal{F}_{(1,1)\to Q}e_{w'}^{(1,1)}$ for $Q$ on $\{(x,y)|y=1/x,1\leq x\leq 4\}$. Here, $R_{\mathcal{F}}$ is a monoid ring over $\mathbf{k}$ generated by
\begin{equation*}
  \mathcal{F}(l_{(x,y)}): \bigoplus_{P\in \mathbb{R}^{2}}\mathcal{F}_{P}\rightarrow \bigoplus_{P\in \mathbb{R}^{2}}\mathcal{F}_{P+(x,y)},\quad (x,y)\in \mathbb{R}^{2}.
\end{equation*}
This shows that $N$ is not a finitely generated $R_{\mathcal{F}}$-module.
%
\end{example}

The following example shows that the module structure of $M_{\mathcal{F}}$ contains the information of all the $1$-parameter persistence objects obtained by the filtrations parameterized by the Cayley grading.
\begin{example}
Example \ref{example:finitely} continued. Note that a monotone non-decreasing continuous function $y=f(x)$ in $\mathbb{R}^{2}$ gives a filtration of sets. Indeed, let $\mathcal{G}:(\mathbb{R},\mathbb{R}_{\geq 0})\rightarrow \mathbf{Sets}$ be a Cayley persistence set given by
\begin{equation*}
  \mathcal{G}_{x}=\mathcal{F}_{P},\quad P=(x,f(x))
\end{equation*}
and
\begin{equation*}
  \mathcal{G}_{x\to x'}=\mathcal{F}_{P\to Q},\quad P=(x,f(x)),Q=(x',f(x')).
\end{equation*}
Then the persistence set $\mathcal{G}$ is the usual $1$-parameter persistence set. Considering the two functions $y=0$ and $y=x$, we have the following barcode diagrams.
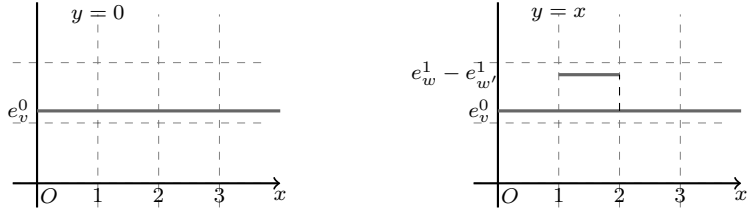
\begin{figure}[H]
\centering
    \begin{tikzpicture}[scale=1.6]
        \draw[step=.5cm,gray,ultra thin, dashed] (-0.2,-0.2) grid (1.9,1.4);%
        \begin{scope}[thick]

            \draw[->] (-0.2,0) -- (2,0);
            \draw[] (0,-0.2) -- (0,1.5);
        \end{scope}
         \draw[-,very thick,black!60] (0,0.6) -- (2,0.6);
\node [font=\fontsize{8}{6}] (node001) at (0.1,-0.1){$O$};
\node [font=\fontsize{8}{6}] (node001) at (2,-0.1){$x$};
\node [font=\fontsize{8}{6}] (node001) at (-0.15,0.6){$e_{v}^{0}$};

\node [font=\fontsize{8}{6}] (node001) at (0.5,-0.1){$1$};
\node [font=\fontsize{8}{6}] (node001) at (1,-0.1){$2$};
\node [font=\fontsize{8}{6}] (node001) at (1.5,-0.1){$3$};
\node [font=\fontsize{8}{6}] (node001) at (0.5,1.4){$y=0$};
        \end{tikzpicture}\qquad\qquad
    \begin{tikzpicture}[scale=1.6]
        \draw[step=.5cm,gray,ultra thin, dashed] (-0.2,-0.2) grid (1.9,1.4);%
        \begin{scope}[thick]

            \draw[->] (-0.2,0) -- (2,0);
            \draw[] (0,-0.2) -- (0,1.5);

        \end{scope}
         \draw[-,very  thick,black!60] (0,0.6) -- (2,0.6);
         \draw[-,very thick,black!60] (0.5,0.9) -- (1,0.9);
         \draw[-,dashed,black!100] (1,0.6) -- (1,0.9);
\node [font=\fontsize{8}{6}] (node001) at (0.1,-0.1){$O$};
\node [font=\fontsize{8}{6}] (node001) at (2,-0.1){$x$};
\node [font=\fontsize{8}{6}] (node001) at (-0.15,0.6){$e^{0}_{v}$};
\node [font=\fontsize{8}{6}] (node001) at (-0.35,0.9){$e^{1}_{w}-e^{1}_{w'}$};

\node [font=\fontsize{8}{6}] (node001) at (0.5,-0.1){$1$};
\node [font=\fontsize{8}{6}] (node001) at (1,-0.1){$2$};
\node [font=\fontsize{8}{6}] (node001) at (1.5,-0.1){$3$};
\node [font=\fontsize{8}{6}] (node001) at (0.5,1.4){$y=x$};
        \end{tikzpicture}
\caption{The barcodes of persistence sets obtained by $y=0$ and $y=x$.}
\end{figure}
It is shown that $e^{0}_{v}$ is born at $x=0$ and persists along the $x$-axis in the left figure. In the right figure, $e^{0}_{v}$ and $e^{1}_{w}$ begin to meet at $x=1$, that is, $e^{1}_{w}=\mathcal{F}_{l_{1}}e^{0}_{v}$. The generators $e^{1}_{w}$ and $e^{1}_{w'}$ begin to meet at $x=2$. If we take $e^{0}_{v},e^{1}_{w}-e^{1}_{w'}$ as generators, then the barcode diagram performs like the usual barcode diagram that $e^{1}_{w}-e^{1}_{w'}$ comes into being at $x=1$ and disappears at $x=2$.
\end{example}

Consider the case that $\mathfrak{C}$ is the category of $\mathbf{k}$-linear spaces. 
Let $\mathcal{F}_{G}=\bigoplus\limits_{x\in G}\mathcal{F}_{x}$. Then $\mathcal{F}_{G}$ itself is an $R_{\mathcal{F}}$-module given by
\begin{equation*}
 R_{\mathcal{F}}\times \mathcal{F}_{G}\rightarrow \mathcal{F}_{G},\quad (\mathcal{F}(l_{x}),\alpha)\mapsto \mathcal{F}(l_{x})\alpha
\end{equation*}
for $\alpha\in \mathcal{F}_{G}$.
\begin{corollary}\label{corollary:decomposition}
Let $\mathcal{F}:\mathrm{cat}(\mathrm{Cay}(G,S))\rightarrow \mathbf{Vec}_{\mathbf{k}}$ be a Cayley-persistence $\mathbf{k}$-linear space.
If $\mathcal{F}_{G}$ is a finitely generated $R_{\mathcal{F}}$-module, then we have a finite direct sum decomposition
\begin{equation*}
   \mathcal{F}_{G} \cong \bigoplus_{i=1}^{k} R_{\mathcal{F}}\cdot e_{i}^{x_{i}}\oplus \left(\bigoplus_{j=1}^{l}R_{\mathcal{F}}\cdot\varepsilon_{j}^{y_{j}}\right)/N
\end{equation*}
for some $k,l$, where $e_{i}^{x_{i}}\in \mathcal{F}_{x_{i}},\varepsilon_{j}^{y_{j}}\in \mathcal{F}_{y_{j}}$ and $N$ is an $R_{\mathcal{F}}$-module generated by the elements of the form $\mathcal{F}(l_{y})\varepsilon_{t}^{y_{t}}-\mathcal{F}(l_{yy_{t}y_{s}^{-1}})\varepsilon_{s}^{y_{s}}$ or $\mathcal{F}(l_{y})\varepsilon_{t}^{y_{t}}$ for some $1\leq s,t\leq l,y\in S$. Moreover, if $S$ is a finitely generated monoid, then $N$ is a finitely generated $R_{\mathcal{F}}$-module.
\end{corollary}
\begin{proof}
Since $\mathcal{F}_{G}$ is a finitely generated $R_{\mathcal{F}}$-module, there exists a finite generating set $S_{gen}$ of $\mathcal{F}_{G}$. Let $f^{x}_{1},\dots,f^{x}_{n_{x}}$ be the basis of the $\mathbf{k}$-linear space $\mathcal{F}_{x}$. Then $f^{x}_{i_{x}},x\in G,1\leq i_{x}\leq n_{x}$ gives a Hamel basis of the $\mathbf{k}$-linear space $\mathcal{F}_{G}$. Thus we can find a $\mathbf{k}$-linear basis $m_{1}^{x_{1}},\dots,m_{n}^{x_{n}}$ for $S_{gen}$, where $m_{i}^{x_{i}}\in \mathcal{F}_{x_{i}}$ for $i=1,\dots,n$.
One has
\begin{equation*}
  \mathcal{F}_{G}\subseteq R_{\mathcal{F}}m_{1}^{x_{1}}+\cdots+R_{\mathcal{F}}m_{n}^{x_{n}}.
\end{equation*}
Let $X_{a}=\{\mathcal{F}(l_{ax_{i}^{-1}})m_{i}^{x_{i}}\}_{ax_{i}^{-1}\in S}$ be a finite set. Consider the functor $\widetilde{\mathcal{F}}$ with $\widetilde{\mathcal{F}}_{a}=X_{a}$ and morphism given by
\begin{equation*}
  \widetilde{\mathcal{F}}(l_{z})=\mathcal{F}(l_{z}):X_{a}\rightarrow X_{za},\quad \mathcal{F}(l_{ax_{i}^{-1}})m_{i}^{x_{i}}\mapsto \mathcal{F}(l_{zax_{i}^{-1}})m_{i}^{x_{i}}.
\end{equation*}
It can be directly verified that $\widetilde{\mathcal{F}}:\mathrm{cat}(\mathrm{Cay}(G,S))\rightarrow \mathbf{Sets}$ is a Cayley-persistence set. We will show the morphism
\begin{equation*}
  \rho:\mathcal{F}_{G}\rightarrow M_{\widetilde{\mathcal{F}}}=\bigoplus_{x\in G}\bigoplus_{\alpha\in \widetilde{\mathcal{F}}(x)} \mathbf{k}\alpha
\end{equation*}
given by $\rho(m_{i}^{x_{i}})=m_{i}^{x_{i}}$ is an isomorphism of $R_{\mathcal{F}}$-modules. It is obvious that $\rho$ is a morphism of $R_{\mathcal{F}}$-modules and it is an injection. Let $m\in M_{\widetilde{\mathcal{F}}}$. Note that $\rho$ is a map of $\mathbf{k}$-linear spaces and $\mathcal{F}_{G}\subseteq R_{\mathcal{F}}m_{1}^{x_{1}}+\cdots+R_{\mathcal{F}}m_{n}^{x_{n}}$. Thus we need only consider the case that $m=\mathcal{F}(l_{ax_{i}^{-1}})m_{i}^{x_{i}}$ for some $m_{i}$ and $a\in G$. Then the pre-image of $m$ in $\mathcal{F}_{G}$ is $\mathcal{F}(l_{ax_{i}^{-1}})m_{i}^{x_{i}}$. The surjectivity follows. By Theorem \ref{thm:module_decomposition}, we have the desired result.
\end{proof}

\begin{example}
In our view, a generator is ``dead'' when it meets a null space or zero element. For example, let $\mathcal{F}:(\mathbb{R}^{2},\mathbb{R}^{2}_{+})\rightarrow \mathbf{Vec}_{\mathbf{k}}$ be a Cayley-persistence vector space, where $\mathbf{Vec}_{\mathbf{k}}$ is a category  of vector spaces. Let $\mathbf{V}_{v}=\mathbf{k}v\oplus \mathbf{k}v'$, $\mathbf{V}_{w}=\mathbf{k}w\oplus \mathbf{k}w'$, and let $\mathbf{0}$ be the null space.
\begin{figure}[H]
\centering
    \begin{tikzpicture}[scale=1.6]
        \draw[step=.5cm,gray,ultra thin, dashed] (-0.4 ,-0.4) grid (2.4,1.9);%
        \begin{scope}[thick]

            \draw[->] (-0.5,0) -- (2.5,0);
            \draw[->] (0,-0.5) -- (0,2);
        \end{scope}

\draw[thick](0,1)--(2.5,1);
\draw[thick](1,0)--(1,2);
\draw[thick](1,0)--(2.5,0);
\draw[thick](0,1)--(0,2);
\draw[thick](2,0)--(2,1);

\node [font=\fontsize{8}{6}] (node001) at (0.1,-0.1){$O$};
\node [font=\fontsize{8}{6}] (node001) at (2.5,-0.1){$x$};
\node [font=\fontsize{8}{6}] (node001) at (0.5,-0.1){$1$};
\node [font=\fontsize{8}{6}] (node001) at (1,-0.1){$2$};
\node [font=\fontsize{8}{6}] (node001) at (1.5,-0.1){$3$};
\node [font=\fontsize{8}{6}] (node001) at (2,-0.1){$4$};
\node [font=\fontsize{8}{6}] (node001) at (-0.5,-0.1){$-1$};

\node [font=\fontsize{8}{6}] (node001) at (0.1,2){$y$};
\node [font=\fontsize{8}{6}] (node001) at (-0.3,-0.3){$\emptyset$};
\node [font=\fontsize{8}{6}] (node001) at (-0.3,1){$\emptyset$};
\node [font=\fontsize{8}{6}] (node001) at (1,-0.3){$\emptyset$};
\node [font=\fontsize{8}{6}] (node001) at (0.5,0.5){$C_{1}$};
\node [font=\fontsize{8}{6}] (node001) at (1.5,0.5){$C_{2}$};
\node [font=\fontsize{8}{6}] (node001) at (0.5,1.5){$C_{3}$};
\node [font=\fontsize{8}{6}] (node001) at (1.5,1.5){$C_{4}$};
\node [font=\fontsize{8}{6}] (node001) at (2.3,0.5){$C_{5}$};
        \end{tikzpicture}\qquad\qquad
    \begin{tikzpicture}[scale=1.6]
        \draw[step=.5cm,gray,ultra thin, dashed] (-0.4 ,-0.4) grid (2.4,1.9);%
        \begin{scope}[thick]

            \draw[->] (-0.5,0) -- (2.5,0);
            \draw[->] (0,-0.5) -- (0,2);
        \end{scope}

\draw[thick](0,1)--(2.5,1);
\draw[thick](1,0)--(1,2);
\draw[thick](1,0)--(2.5,0);
\draw[thick](0,1)--(0,2);\draw[thick](2,0)--(2,1);
\node [font=\fontsize{8}{6}] (node001) at (0.1,-0.1){$O$};
\node [font=\fontsize{8}{6}] (node001) at (2.5,-0.1){$x$};
\node [font=\fontsize{8}{6}] (node001) at (0.5,-0.1){$1$};
\node [font=\fontsize{8}{6}] (node001) at (1,-0.1){$2$};
\node [font=\fontsize{8}{6}] (node001) at (1.5,-0.1){$3$};
\node [font=\fontsize{8}{6}] (node001) at (2,-0.1){$4$};
\node [font=\fontsize{8}{6}] (node001) at (-0.5,-0.1){$-1$};

\node [font=\fontsize{8}{6}] (node001) at (0.1,2){$y$};
\node [font=\fontsize{8}{6}] (node001) at (-0.3,-0.3){$\emptyset$};
\node [font=\fontsize{8}{6}] (node001) at (-0.3,1){$\emptyset$};
\node [font=\fontsize{8}{6}] (node001) at (1,-0.3){$\emptyset$};
\node [font=\fontsize{8}{6}] (node001) at (0.5,0.5){$\mathbf{0}$};
\node [font=\fontsize{8}{6}] (node001) at (1.5,0.5){$v,v'$};
\node [font=\fontsize{8}{6}] (node001) at (0.5,1.5){$w,w'$};
\node [font=\fontsize{8}{6}] (node001) at (1.5,1.5){$\mathbf{0}$};
\node [font=\fontsize{8}{6}] (node001) at (2.3,0.5){$u$};
        \end{tikzpicture}
\end{figure}
Let $C_{1}=\{(x,y)|0\leq x,y< 2\}$, $C_{2}=\{(x,y)|2\leq x<4,0\leq y< 2\}$, $C_{3}=\{(x,y)|0\leq x<2,y\geq 2\}$, $C_{4}=\{(x,y)|0\leq x,y\leq 2\}$, $C_{5}=\{(x,y)|x\geq 4,0\leq y< 2\}$.
The functor $\mathcal{F}$ is given by
\begin{equation*}
  \mathcal{F}_{P}=\left\{
                \begin{array}{ll}
                  \mathbf{V}_{v}, & \hbox{$P\in C_{2}$;} \\
                  \mathbf{V}_{w}, & \hbox{$P\in C_{3}$;} \\
                  \mathbf{V}_{u}, & \hbox{$P\in C_{5}$;} \\
                  \mathbf{0}, & \hbox{$P\in C_{1}\cup C_{4}$;} \\
                  \emptyset, & \hbox{otherwise.}
                \end{array}
              \right.
\end{equation*}
and
\begin{equation*}
  \mathcal{F}_{P\to Q}=\left\{
                \begin{array}{ll}
                  \mathrm{id}, & \hbox{$P,Q\in C_{2},C_{3}$;} \\
                  \varphi, & \hbox{$P\in C_{2},Q\in C_{5}$;} \\
                   0, & \hbox{$P\in C_{1}$ or $Q\in C_{4}$;} \\
                  \emptyset, & \hbox{otherwise.}
                \end{array}
              \right.
\end{equation*}
Here, $\varphi:\mathbf{V}_{v}\rightarrow \mathbf{k}u$ is given by $\varphi(v)=\varphi(v')=u$.
Then $w,w'$ begin to appear at $(0,2)$ and disappear at $(2,2)$ while $v,v'$ begin to appear at $(2,0)$, meet at $(4,0)$ and disappear at $(2,2)$.
A straightforward calculation shows that
\begin{equation*}
   \mathcal{F}_{G}\cong \frac{R_{\mathcal{F}}\cdot e^{(2,0)}_{v}\oplus R_{\mathcal{F}}\cdot e^{(2,0)}_{v'}\oplus R_{\mathcal{F}}\cdot e^{(0,2)}_{w}\oplus R_{\mathcal{F}}\cdot e^{(0,2)}_{w'}\oplus R_{\mathcal{F}}\cdot e^{(4,0)}_{u}}{N},
\end{equation*}
where $N$ is an $R_{\mathcal{F}}$-module generated by
\begin{equation*}
  \mathcal{F}_{l_{(0,2)}}e_{v}^{(2,0)},\mathcal{F}_{l_{(0,2)}}e_{v'}^{(2,0)},\mathcal{F}_{l_{(2,0)}}e_{w}^{(0,2)},\mathcal{F}_{(l_{(2,0)}}e_{w'}^{(0,2)},\mathcal{F}_{l_{(2,0)}}e_{v}^{(2,0)}-e_{u}^{(4,0)}, \mathcal{F}_{l_{(2,0)}}e_{v'}^{(2,0)}-e_{u}^{(4,0)}.
\end{equation*}
By reduction, we get
\begin{equation*}
   \mathcal{F}_{G}\cong \frac{R_{\mathcal{F}}\cdot e^{(2,0)}_{v}\oplus R_{\mathcal{F}}\cdot e^{(2,0)}_{v'}\oplus R_{\mathcal{F}}\cdot e^{(0,2)}_{w}\oplus R_{\mathcal{F}}\cdot e^{(0,2)}_{w'}}{N'},
\end{equation*}
where
\begin{equation*}
\begin{split}
  N'= & R_{\mathcal{F}}\cdot\mathcal{F}(l_{(0,2)})e_{v}^{(2,0)}\oplus R_{\mathcal{F}}\cdot\mathcal{F}(l_{(0,2)})e_{v'}^{(2,0)}\oplus R_{\mathcal{F}}\cdot\mathcal{F}(l_{(2,0)})e_{w}^{(0,2)}\\
    &\oplus R_{\mathcal{F}}\cdot \mathcal{F}(l_{(2,0)})e_{w'}^{(0,2)}\oplus R_{\mathcal{F}}\cdot \mathcal{F}(l_{(2,0)})(e_{v}^{(2,0)}-e_{v'}^{(2,0)}).
\end{split}
\end{equation*}
The numerator indicates the information of appearance of generators while the denominator contains the information of disappearance and meeting of generators.

For the two 1-parameter filtrations given by $y=0$ and $y=x^{3}$, we have the following barcode diagrams.
\begin{figure}[H]
\centering
    \begin{tikzpicture}[scale=1.6]
        \draw[step=.5cm,gray,ultra thin, dashed] (-0.2,-0.2) grid (2.4,1.4);%
        \begin{scope}[thick]

            \draw[->] (-0.2,0) -- (2.5,0);
            \draw[] (0,-0.2) -- (0,1.5);
        \end{scope}
         \draw[-,thick,black!60] (1,0.6) -- (2.5,0.6);
         \draw[-,thick,black!60] (1,0.9) -- (2,0.9);

\node [font=\fontsize{8}{6}] (node001) at (0.1,-0.1){$O$};
\node [font=\fontsize{8}{6}] (node001) at (2.5,-0.1){$x$};
\node [font=\fontsize{8}{6}] (node001) at (-0.15,0.6){$v$};
\node [font=\fontsize{8}{6}] (node001) at (-0.25,0.9){$v-v'$};
\node [font=\fontsize{8}{6}] (node001) at (0.5,-0.1){$1$};
\node [font=\fontsize{8}{6}] (node001) at (1,-0.1){$2$};
\node [font=\fontsize{8}{6}] (node001) at (1.5,-0.1){$3$};
\node [font=\fontsize{8}{6}] (node001) at (2,-0.1){$4$};
        \end{tikzpicture}\qquad\qquad
    \begin{tikzpicture}[scale=1.6]
        \draw[step=.5cm,gray,ultra thin, dashed] (-0.2,-0.2) grid (2.4,1.4);%
        \begin{scope}[thick]

            \draw[->] (-0.2,0) -- (2.5,0);
            \draw[] (0,-0.2) -- (0,1.5);

        \end{scope}
         \draw[-, thick,black!60] (0.63,0.6) -- (1,0.6);
         \draw[-, thick,black!60] (0.63,0.9) -- (1,0.9);
         \draw[dashed,black!60] (0.63,0.9) -- (0.63,0);
\node [font=\fontsize{8}{6}] (node001) at (0.1,-0.1){$O$};
\node [font=\fontsize{8}{6}] (node001) at (2.5,-0.1){$x$};
\node [font=\fontsize{8}{6}] (node001) at (-0.15,0.6){$w$};\node [font=\fontsize{8}{6}] (node001) at (-0.15,0.9){$w'$};
\node [font=\fontsize{8}{6}] (node001) at (0.5,-0.1){$1$};
\node [font=\fontsize{8}{6}] (node001) at (0.63,-0.1){$\sqrt[3]{2}$};
\node [font=\fontsize{8}{6}] (node001) at (1,-0.1){$2$};
\node [font=\fontsize{8}{6}] (node001) at (1.5,-0.1){$3$};
\node [font=\fontsize{8}{6}] (node001) at (2,-0.1){$4$};
        \end{tikzpicture}
\caption{The barcodes of persistence sets obtained by $y=0$ and $y=x^{3}$.}
\end{figure}
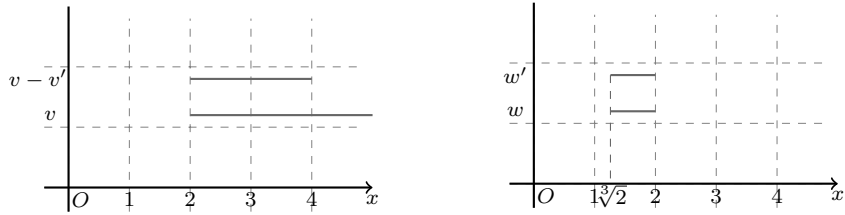
By using the representation introduced in Remark \ref{remark:meet}, we have a collection of triples
\begin{equation*}
  (e_{v}^{(2,0)},\mathbf{0},(2,2)),(e_{v'}^{(2,0)},\mathbf{0},(2,2)),(e_{w}^{(0,2)},\mathbf{0},(2,2)),(e_{w'}^{(0,2)},\mathbf{0},(2,2)),(e_{v}^{(2,0)},e_{v'}^{(2,0)},(4,0)).
\end{equation*}
This gives a finite set to represent the information of infinite many 1-parameter filtrations in the parameter space. This also shows us the survival spaces rather than some survival intervals of the generators.
\end{example}

\subsection{The Cayley persistent homology}
As an application of the module structure of $M_{\mathcal{F}}$, we introduce the Cayley persistent homology in this section. The simplicial complexes considered are always assumed to be finite.
Unless otherwise specified, we always use the reduced homology for convenience. We will first study when a Cayley persistent homology is a finitely generated module.

From now on, let $S\subseteq G$ be a monoid such that the identity element $e$ is the unique invertible element in $S$. Then the group $G$ can be regarded as a poset with partial order given by $a\leq b$ if $ba^{-1}\in S$.
\begin{definition}
A Cayley persistence object $\mathcal{F}:\mathrm{cat}(\mathrm{Cay}(G,S))\rightarrow \mathfrak{C}$ is called \emph{noetherian} if every ordered subset $X\subseteq G$ has an element $x$ such that $\mathcal{F}(l_{z}):\mathcal{F}_{x}\to \mathcal{F}_{zx}$ is the identity morphism for all $z\in S$.

A Cayley persistence object $\mathcal{F}:\mathrm{cat}(\mathrm{Cay}(G,S))\rightarrow \mathfrak{C}$ is called \emph{lower bounded} if there exists an element $a\in G$ such that $\mathcal{F}_{x}=\emptyset$ unless $x\geq a$.
\end{definition}
Generally, $\mathcal{F}$ does not have to be a finitely generated $R_{\mathcal{F}}$-module even if $\mathcal{F}_{a}$ is finite for all $a\in G$.
\begin{theorem}\label{theorem:finitely}
Let $\mathcal{F}:\mathrm{cat}(\mathrm{Cay}(G,S))\rightarrow \mathfrak{C}$ be a  noetherian and lower bounded Cayley-persistence simplicial complex. If $S$ is a finitely generated  commutative monoid, then $M_{\mathcal{F}}$ is a finitely generated $R_{\mathcal{F}}$-module.
\end{theorem}
\begin{proof}
Since $S$ is a finitely generated  commutative monoid, there exists a finite set $S_{gen}=\{s_{1},\dots,s_{p}\}$ generating $S$.
As $\mathcal{F}$ is lower bounded, there exists an element $a\in G$ such that $\mathcal{F}_{x}=\emptyset$ unless $a\leq x$.  Let $X_{i}=\{a,s_{i}a,\dots,s_{i}^{k}a,\dots\}$. Since $\mathcal{F}$ is noetherian, there exist integers $k_{1},k_{2},\dots,k_{p}$ such that
\begin{equation*}
  \mathcal{F}(l_{z}):\mathcal{F}_{s_{i}^{k_{i}}a}\to \mathcal{F}_{zs_{i}^{k_{i}}a},\quad i=1,\dots,p
\end{equation*}
are the identity morphisms for all $z\in S$. Let $b=\prod\limits_{i=1}^{p}s_{i}^{k_{i}}a$ and
\begin{equation*}
  B=\{x\in S|a\leq x\leq b\}=\{\prod\limits_{i=1}^{p}s_{i}^{l_{i}}a|l_{1}\leq k_{1},l_{2}\leq k_{2},\dots,l_{p}\leq k_{p}\}.
\end{equation*}
Then the morphism
\begin{equation*}
  \mathcal{F}_{s_{i}^{k_{i}}a}\to \mathcal{F}_{b},\quad i=1,\dots,p
\end{equation*}
is an identity map for all $1\leq i\leq p$. It follows that $\mathcal{F}_{s_{i}^{k_{i}}a}=\mathcal{F}_{s_{j}^{k_{j}}a}$ for any $1\leq i,j\leq p$. This shows that $\mathcal{F}_{x}$ is either $\mathcal{F}_{b}$ or $\emptyset$ for $x\notin B$. By the above construction, the number of elements in $B$ is finite. Note that
\begin{equation*}
  M_{\mathcal{F}}\subseteq \sum\limits_{l_{1}\leq k_{1},l_{2}\leq k_{2},\dots,l_{p}\leq k_{p}}R_{\mathcal{F}}\cdot \mathcal{F}_{(\prod\limits_{i=1}^{p}s_{i}^{l_{i}}a)},
\end{equation*}
we have that $M_{\mathcal{F}}$ is a finitely generated $R_{\mathcal{F}}$-module.
\end{proof}

Let $\mathcal{K}:\mathrm{cat}(\mathrm{Cay}(\mathbb{Z}^{n},\mathbb{Z}_{\geq 0}^{n}))\rightarrow \mathbf{Simp}$ be a Cayley-persistence simplicial complex.  By abuse of notations, let $H_{\ast}(-;\mathbf{k}):\mathbf{Simp}\rightarrow \mathbf{Vec}_{\mathbf{k}}$ be the functor given by the reduced homology.
Then $H_{\ast}(\mathcal{K};\mathbf{k}):\mathrm{cat}(\mathrm{Cay}(\mathbb{Z}^{n},\mathbb{Z}_{\geq 0}^{n}))\rightarrow \mathbf{Vec}_{\mathbf{k}}$
is a Cayley-persistence $\mathbf{k}$-linear space (or a Cayley-persistence module). Indeed, for $a,b\in \mathbb{Z}^{n}$, we have a morphism of simplicial complexes
\begin{equation*}
  \mathcal{K}_{a,b}: \mathcal{K}_{a}\rightarrow  \mathcal{K}_{b}.
\end{equation*}
It induces a morphism of $\mathbf{k}$-linear spaces
\begin{equation*}
  H_{\ast}(\mathcal{K}_{x,y};\mathbf{k}): H_{\ast}(\mathcal{K}_{x};\mathbf{k})\rightarrow  H_{\ast}(\mathcal{K}_{y};\mathbf{k}).
\end{equation*}
The \emph{$(a,b)$-Cayley persistent homology} is defined by
\begin{equation*}
  \mathrm{im}(H_{\ast}(\mathcal{K}_{x,y};\mathbf{k}))=\mathrm{im}(H_{\ast}(\mathcal{K}_{x};\mathbf{k})\rightarrow  H_{\ast}(\mathcal{K}_{y};\mathbf{k})).
\end{equation*}
We now denote $\mathbf{H}(\mathcal{K})=(H_{\ast}(\mathcal{K}_{a};\mathbf{k}))_{a\in \mathbb{Z}^{n}}$.
By Example \ref{exam:Cayley}, the ring $R_{H_{\ast}(\mathcal{K};\mathbf{k})}=\mathbf{k}[L_{1},\dots,L_{n}]$ is a polynomial ring of $n$ variables.
Here, $L_{i}=H_{\ast}(\mathcal{K}(l_{e_{i}});\mathbf{k}):\mathbf{H}(\mathcal{K})\rightarrow \mathbf{H}(\mathcal{K})$, where $e_{i}=(0,\dots,0, \overset{i}{1},0\dots,0)$.
\begin{theorem}\label{theorem:multi}
Let $\mathcal{K}:\mathrm{cat}(\mathrm{Cay}(\mathbb{Z}^{n},\mathbb{Z}_{\geq 0}^{n}))\rightarrow \mathbf{Simp}$ be a lower bounded and neotherian Cayley-persistence simplicial complex. Then we have a finite direct sum decomposition
\begin{equation*}
   \mathbf{H}(\mathcal{K}) \cong \bigoplus_{i=1}^{k} \mathbf{k}[L_{1},\dots,L_{n}]\cdot e_{i}^{x_{i}}\oplus \left(\bigoplus_{j=1}^{l}\mathbf{k}[L_{1},\dots,L_{n}]\cdot\varepsilon_{j}^{y_{j}}\right)/N
\end{equation*}
for some $k,l$, where $e_{i}^{x_{i}}\in \mathcal{K}_{x_{i}},\varepsilon_{j}^{y_{j}}\in \mathcal{K}_{y_{j}}$ and $N$ is a finitely generated $\mathbf{k}[L_{1},\dots,L_{n}]$-module generated by the elements of the form
\begin{equation*}
  L_{1}^{z_{1}}L_{2}^{z_{2}}\cdots L_{n}^{z_{n}}\varepsilon_{t}^{y_{t}}-L_{1}^{w_{1}}L_{2}^{w_{2}}\cdots L_{n}^{w_{n}}\varepsilon_{s}^{y_{s}}
\end{equation*}
for some $s,t$ such that $z+y_{t}=w+y_{s}$. Here $z=(z_{1},\dots,z_{n}),w=(w_{1},\dots,w_{n})$ are points in $\mathbb{Z}_{\geq 0}^{n}$.
\end{theorem}
\begin{proof}
By Theorem \ref{theorem:finitely} and the proof of Corollary \ref{corollary:decomposition}, $\mathbf{H}(\mathcal{K})$ is a finitely generated $\mathbf{k}[L_{1},\dots,L_{n}]$-module. By Corollary \ref{corollary:decomposition}, we have the desired result.
\end{proof}
In practical application, the filtration of simplicial complexes always begins at a given parameter and becomes stable at a finite parameter. The beginning parameter indicates that $\mathcal{K}:\mathrm{cat}(\mathrm{Cay}(\mathbb{Z}^{n},\mathbb{Z}_{\geq 0}^{n}))\rightarrow \mathbf{Simp}$ is lower bounded, while ``stable'' implies that $\mathcal{K}$ is noetherian.
\begin{example}
Let $V$ be a weighted data-set of points in a Euclidean space, that is, a finite point set $V=\{x_{1},x_{2},\dots,x_{n}\}$ in a Euclidean space with a weight function $w:V\rightarrow \mathbb{R}$. Then we have a Vietoris-Rips complex \cite{carlsson2006algebraic} derived from a distance $\delta$. More precisely, a Vietoris-Rips complex is an abstract simplicial complex  whose $p$-simplices correspond to the sets of $(p+1)$ points which are pairwise within distance $\delta$. Let $V_{a}=\{x\in V|w(x)\leq a\}$. Then we have a Vietoris-Rips complex $\mathcal{K}_{a,\delta}$ for given $a\in \mathbb{R}$ and $\delta\geq 0$. Since $V$ is finite, we can find finite many real numbers
\begin{equation*}
  a_{0},a_{1},\dots,a_{k},\quad \delta_{0},\delta_{1},\dots,\delta_{l}
\end{equation*}
such that the collection $\mathcal{K}_{a_{s},\delta_{t}},0\leq s\leq k,0\leq t\leq l$ contains the information of $\mathcal{K}_{a,\delta},a\geq 0,\delta\in \mathbb{R}$.
It is obvious that $\mathcal{K}:\mathrm{cat}(\mathrm{Cay}(\mathbb{Z}^{2},\mathbb{Z}_{\geq 0}^{2}))\rightarrow \mathbf{Simp}$ given by
\begin{equation*}
  \mathcal{K}(s,t)=\left\{
                     \begin{array}{ll}
                       \emptyset, & \hbox{$s< k$ or $t< 0$;} \\
                       \mathcal{K}_{a_{s},\delta_{t}}, & \hbox{$0\leq s\leq k,0\leq t\leq l$;} \\
                       \mathcal{K}_{a_{k},\delta_{l}}, & \hbox{$s>k$ or $t>l$}
                     \end{array}
                   \right.
\end{equation*}
is a lower bounded  and noetherian persistence simplicial complex. This gives the usual $2$-parameter persistent homology.
\begin{figure}[H]
\centering
    \begin{tikzpicture}[scale=1.6]
        \draw[step=1cm,gray,ultra thin, dashed] (-0.2,-0.2) grid (2.4,1.4);%
        \begin{scope}[thick]

            \draw[->] (-0.2,0) -- (2.5,0);
            \draw[->] (0,-0.2) -- (0,1.5);
        \end{scope}

\draw[fill=black!100](0,0) circle(1pt);
\draw[fill=black!100](2,0) circle(1pt);
\draw[fill=black!100](0,1) circle(1pt);
\draw[fill=black!100](2,1) circle(1pt);

\node [font=\fontsize{8}{6}] (node001) at (0.1,-0.1){$O$};
\node [font=\fontsize{8}{6}] (node001) at (2.5,-0.1){$x$};
\node [font=\fontsize{8}{6}] (node001) at (-0.15,1.5){$y$};
\node [font=\fontsize{8}{6}] (node001) at (0.4,0.15){$w(P_{1})=1$};
\node [font=\fontsize{8}{6}] (node001) at (2.15,0.15){$w(P_{2})=2$};
\node [font=\fontsize{8}{6}] (node001) at (0.4,1.15){$w(P_{3})=3$};
\node [font=\fontsize{8}{6}] (node001) at (2.15,1.15){$w(P_{4})=1$};

 \end{tikzpicture}
\end{figure}
For example, let $V=\{x_{1},x_{2},x_{3},x_{4}\}$, where $x_{1}=(0,0),x_{2}=(2,0),x_{3}=(0,1),x_{4}=(2,1)$ are points in $\mathbb{R}^{2}$. Consider the weight function $w:V\rightarrow \mathbb{R}$ given by
\begin{equation*}
  w(P_{1})=1,\quad  w(P_{2})=2,\quad  w(P_{3})=3,\quad  w(P_{4})=1.
\end{equation*}
Then we have a $2$-parameter filtration given as the following diagram.
\begin{figure}[H]
\centering
    \begin{tikzpicture}[scale=2]

        \begin{scope}[thick]
        \draw[step=1cm,black!40] (-0.5,0) grid (4,3.5);%
            \draw[black!40] (-0.5,3.5) -- (4,3.5);
         \draw[-,black!40] (-0.5,0) -- (-0.5,3.5);
         \draw[-,black!40] (-0.5,3.5) -- (0,3);
        \end{scope}

\node [font=\fontsize{8}{6}] (node001) at (-0.3,3.15){$s$};
\node [font=\fontsize{8}{6}] (node001) at (-0.2,3.35){$t$};
\node [font=\fontsize{8}{6}] (node001) at (-0.25,2.5){$s_{0}=1$};
\node [font=\fontsize{8}{6}] (node001) at (-0.25,1.5){$s_{1}=2$};
\node [font=\fontsize{8}{6}] (node001) at (-0.25,0.5){$s_{2}=3$};
\node [font=\fontsize{8}{6}] (node001) at (0.5,3.25){$t_{0}=0$};
\node [font=\fontsize{8}{6}] (node001) at (1.5,3.25){$t_{1}=1$};
\node [font=\fontsize{8}{6}] (node001) at (2.5,3.25){$t_{2}=2$};
\node [font=\fontsize{8}{6}] (node001) at (3.5,3.25){$t_{3}=\sqrt{5}$};


\draw[fill=black!100](0.1,0.3) circle(0.5pt);\draw[fill=black!100](0.9,0.3) circle(0.5pt);\draw[fill=black!100](0.1,0.7) circle(0.5pt);\draw[fill=black!100](0.9,0.7) circle(0.5pt);
\node [font=\fontsize{8}{6}] (node001) at (0.3,0.7){$u^{(2,0)}$};

\draw[fill=black!100](1.1,0.3) circle(0.5pt);\draw[fill=black!100](1.9,0.3) circle(0.5pt);\draw[fill=black!100](1.1,0.7) circle(0.5pt);\draw[fill=black!100](1.9,0.7) circle(0.5pt);
\draw[-,black!100] (1.1,0.3) -- (1.1,0.7);\draw[-,black!100] (1.9,0.3) -- (1.9,0.7);

\draw[fill=black!100](2.1,0.3) circle(0.5pt);\draw[fill=black!100](2.9,0.3) circle(0.5pt);\draw[fill=black!100](2.1,0.7) circle(0.5pt);\draw[fill=black!100](2.9,0.7) circle(0.5pt);
\draw[-,black!100] (2.1,0.3) -- (2.1,0.7);\draw[-,black!100] (2.9,0.3) -- (2.9,0.7);\draw[-,black!100] (2.1,0.3) -- (2.9,0.3);\draw[-,black!100] (2.1,0.7) -- (2.9,0.7);
\node [font=\fontsize{8}{6}] (node001) at (2.5,0.5){$e^{(2,2)}$};

\draw[fill=black!100](3.1,0.3) circle(0.5pt);\draw[fill=black!100](3.9,0.3) circle(0.5pt);\draw[fill=black!100](3.1,0.7) circle(0.5pt);\draw[fill=black!100](3.9,0.7) circle(0.5pt);
\draw[-,black!100] (3.1,0.3) -- (3.1,0.7);\draw[-,black!100] (3.9,0.3) -- (3.9,0.7);\draw[-,black!100] (3.1,0.3) -- (3.9,0.3);\draw[-,black!100] (3.1,0.7) -- (3.9,0.7);
\draw[fill=black!30](3.1,0.3) -- (3.1,0.7)-- (3.9,0.7)--(3.9,0.3) -- (3.1,0.3);\draw[-,black!100] (3.9,0.3)-- (3.1,0.7);\draw[-,black!100]  (3.1,0.3) -- (3.9,0.7);

\draw[fill=black!100](0.1,1.3) circle(0.5pt);\draw[fill=black!100](0.9,1.3) circle(0.5pt);\draw[fill=black!100](0.9,1.7) circle(0.5pt);
\node [font=\fontsize{8}{6}] (node001) at (0.7,1.2){$w^{(1,0)}$};

\draw[fill=black!100](1.1,1.3) circle(0.5pt);\draw[fill=black!100](1.9,1.3) circle(0.5pt);\draw[fill=black!100](1.9,1.7) circle(0.5pt);
\draw[-,black!100] (1.9,1.3) -- (1.9,1.7);

\draw[fill=black!100](2.1,1.3) circle(0.5pt);\draw[fill=black!100](2.9,1.3) circle(0.5pt);\draw[fill=black!100](2.9,1.7) circle(0.5pt);
\draw[-,black!100] (2.9,1.3) -- (2.9,1.7);\draw[-,black!100] (2.1,1.3) -- (2.9,1.3);

\draw[fill=black!100](3.1,1.3) circle(0.5pt);\draw[fill=black!100](3.9,1.3) circle(0.5pt);\draw[fill=black!100](3.9,1.7) circle(0.5pt);
\draw[-,black!100] (3.9,1.3) -- (3.9,1.7);\draw[-,black!100] (3.1,1.3) -- (3.9,1.7);\draw[-,black!100] (3.1,1.3) -- (3.9,1.3);
\draw[fill=black!30](3.9,1.3) -- (3.9,1.7)--(3.1,1.3)--(3.9,1.3);

\draw[fill=black!100](0.1,2.3) circle(0.5pt);\draw[fill=black!100](0.9,2.7) circle(0.5pt);
\node [font=\fontsize{8}{6}] (node001) at (0.5,2.2){\footnotesize{Base point}};\node [font=\fontsize{8}{6}] (node001) at (0.8,2.8){$v^{(0,0)}$};

\draw[fill=black!100](1.1,2.3) circle(0.5pt);\draw[fill=black!100](1.9,2.7) circle(0.5pt);

\draw[fill=black!100](2.1,2.3) circle(0.5pt);\draw[fill=black!100](2.9,2.7) circle(0.5pt);

\draw[fill=black!100](3.1,2.3) circle(0.5pt);\draw[fill=black!100](3.9,2.7) circle(0.5pt);
\draw[-,black!100] (3.1,2.3) -- (3.9,2.7);
\end{tikzpicture}
\end{figure}
By Theorem \ref{theorem:multi}, we have a finite direct sum decomposition
\begin{equation*}
   \mathbf{H}(\mathcal{K}) \cong \left( \mathbf{k}[L_{1},L_{2}]\cdot v^{(0,0)}\oplus\mathbf{k}[L_{1},L_{2}]\cdot w^{(1,0)}\oplus\mathbf{k}[L_{1},L_{2}]\cdot u^{(2,0)}\oplus\mathbf{k}[L_{1},L_{2}]\cdot e^{(2,2)}\right)/N,
\end{equation*}
where $N$ is a $\mathbf{k}[L_{1},L_{2}]$-module generated by
\begin{equation*}
  L_{2}^{3}v^{(0,0)},\quad L_{2}w^{(1,0)}-L_{1}L_{2}v^{(0,0)},\quad L_{2}^{2}w^{(1,0)},\quad L_{2}u^{(2,0)},\quad L_{2}e^{(2,2)}.
\end{equation*}
Here, $L_{1}=H_{\ast}(\mathcal{K}(l_{(1,0)});\mathbf{k})$ and $L_{2}=H_{\ast}(\mathcal{K}(l_{(0,1)});\mathbf{k})$.
\end{example}

\section{The algebraic structure on Cayley-persistence}\label{section:algebra}
\subsection{Persistent cohomology algebra}
Let $\mathcal{K}:\mathrm{cat}(\mathrm{Cay}(G,S))\rightarrow \mathbf{Simp}$ be a Cayley-persistence simplicial complex such that the homology $H_{\ast}(\mathcal{K}_{a};\mathbf{k})$ is of finite dimension  for all $a\in G$. In this section, the cohomology considered is also reduced. Then we have a Cayley-copersistence module
\begin{equation*}
  H^{\ast}(\mathcal{K};\mathbf{k}):\mathrm{cat}(\mathrm{Cay}(G,S))\rightarrow \mathbf{Vec}_{\mathbf{k}},
\end{equation*}
that is, $H^{\ast}(\mathcal{K}_{a};\mathbf{k})$ is a finite dimensional linear space  for $a\in G$ and $H^{\ast}(\mathcal{K}_{a,b};\mathbf{k})H^{\ast}(\mathcal{K}_{b,c};\mathbf{k})=H^{\ast}(\mathcal{K}_{a,c};\mathbf{k})$ for all $ba^{-1},cb^{-1}\in S$. Here, $H^{\ast}(\mathcal{K}_{a,b};\mathbf{k})$ is the morphism $H^{\ast}(\mathcal{K}_{b};\mathbf{k})\rightarrow H^{\ast}(\mathcal{K}_{a};\mathbf{k})$ induced by $a\rightarrow b$. For convenience, we denote $H^{\ast}_{a}=H^{\ast}(\mathcal{K}_{a};\mathbf{k})$. Let
\begin{equation*}
  L_{x}=H^{\ast}(\mathcal{K}_{a,xa};\mathbf{k}):H^{\ast}_{xa}\rightarrow H^{\ast}_{a},\quad a\in G,x\in S.
\end{equation*}
We can regard $L_{x}$ as a $G$-graded map on $\bigoplus\limits_{a\in G}H^{\ast}_{a}$.
Now, for each $a\in G$, there is a cup product given by
\begin{equation*}
  \cup:H_{a}^{p}\cup H_{a}^{q}\rightarrow H_{a}^{p+q}.
\end{equation*}
We will introduce the persistent product space.
\begin{definition}
For $a\in G,x\in S$, the \emph{$(a,xa)$-persistent cup-space} is defined to be
\begin{equation*}
  \mathrm{im}(H_{xa}^{\ast}\times H_{xa}^{\ast}\stackrel{\cup}{\rightarrow} H_{xa}^{\ast}\stackrel{L_{x}}{\rightarrow} H_{a}^{\ast}),
\end{equation*}
denoted by $P^{\ast}_{a,xa}$. Here, $L_{x}:H_{xa}^{\ast}\rightarrow H_{a}^{\ast}$ is induced by $a\rightarrow xa$.
\end{definition}
In particular, when $x=e$ is the identity element, we denote $P^{\ast}_{a}=P^{\ast}_{a,a}$ and call it the cup-space at $a$. It is obvious that $P^{\ast}_{a}=H_{a}^{\ast}\cup H_{a}^{\ast}$.
Note that the cup product on $H^{\ast}_{a}$ induces a cup product on $H^{\ast}_{a,xa}$. The following result gives a relationship between persistent cup-space and persistent cohomology.
\begin{proposition}
$P^{\ast}_{a,xa}=H^{\ast}_{a,xa}\cup H^{\ast}_{a,xa}$, where $H^{\ast}_{a,xa}$ is the $(a,xa)$-persistent cohomology.
\end{proposition}
\begin{proof}
By \cite[Proposition 3.10]{Hatcher}, we have a commutative diagram
\begin{equation*}
\xymatrix{
  H^{\ast}_{xa}\times H^{\ast}_{xa}\ar@{->}[r]^{\cup}\ar@{->}[d]_{(L_{x},L_{x})}& H^{\ast}_{xa}\ar@{->}[d]^{L_{x}} \\
   H^{\ast}_{a}\times H^{\ast}_{a}\ar@{->}[r]^{\cup} & H^{\ast}_{a}.
}
\end{equation*}
It follows that $P^{\ast}_{a,xa}=L_{x}(H^{\ast}_{xa}\cup H^{\ast}_{xa})= (L_{x}H^{\ast}_{xa})\cup(L_{x}H^{\ast}_{xa})=H^{\ast}_{a,xa}\cup H^{\ast}_{a,xa}$.
\end{proof}
\begin{definition}
For $a\in G,x,y\in S$, the \emph{$(a,xa\cup ya)$-persistent cup-space} is defined to be
\begin{equation*}
  \mathrm{im}(H_{xa}^{\ast}\times H_{ya}^{\ast}\stackrel{(L_{x},L_{y})}{\longrightarrow} H_{a}^{\ast}\times H_{a}^{\ast}\stackrel{\cup}{\longrightarrow} H_{a}^{\ast}),
\end{equation*}
denoted by $P^{\ast}_{a,xa\cup ya}$. For $\beta\in H_{xa}^{\ast}$ and $\gamma\in H_{ya}^{\ast}$, the \emph{cup product of $\beta,\gamma$ at $a$} is defined by
\begin{equation*}
  \beta\cup_{a}\gamma=(L_{x}\beta)\cup (L_{y}\gamma).
\end{equation*}
\end{definition}
By definition, we have $P^{\ast}_{a,xa\cup ya}=H^{\ast}_{a,xa}\cup H^{\ast}_{a,ya}$.
It is obvious that $P^{\ast}_{a,xa\cup xa}=P^{\ast}_{a,xa}$. Now, for $a\in G$, we set
\begin{equation*}
  \mathbf{H}_{a}=\bigoplus\limits_{x\in S}H^{\ast}_{xa},\quad \mathbf{P}_{a}=\sum_{x,y\in S} H_{xa}^{\ast}\cup_{a}H_{ya}^{\ast}.
\end{equation*}
\begin{proposition}
For $a\in G$, $(\mathbf{H}_{a},\cup_{a})$ is a graded commutative $\mathbf{k}$-algebra. Moreover, we have $\mathbf{P}_{a}=\mathbf{H}_{a}\cup_{a} \mathbf{H}_{a}$.
\end{proposition}
\begin{proof}
Note that $\cup_{a}$ is indeed a product on $\mathbf{H}_{a}$ with the associated law and graded commutativity inherited from the cup product. It is obvious that $\mathbf{P}_{a}\subseteq\mathbf{H}_{a}\cup_{a} \mathbf{H}_{a}$. For any $\beta=\beta_{1}+\cdots+\beta_{k},\gamma=\gamma_{1}+\cdots+\gamma_{l}\in \mathbf{H}_{a}$, we have
\begin{equation*}
  \beta\cup_{a}\gamma=\sum_{i=1}^{k}\sum_{j=1}^{l}\beta_{i}\cup_{a}\gamma_{j}\in \sum_{x,y\in S} H_{xa}^{\ast}\cup_{a}H_{ya}^{\ast}.
\end{equation*}
It follows that $\mathbf{H}_{a}\cup_{a} \mathbf{H}_{a}\subseteq \mathbf{P}_{a}$.
\end{proof}

The structure of cohomology ring or algebra is also useful in applications as a topological feature. For $a,b\in G$ with $a^{-1}b\in S$, the $(a,b)$-persistent \emph{cup-number} of $\mathcal{K}$ is defined to be
\begin{equation*}
  c_{a,b}(\mathcal{K})=\dim P^{\ast}_{a,b}(\mathcal{K};\mathbf{k}).
\end{equation*}
The persistent cup-number is also a topological invariant which has the potential for application. In the next section, we will further consider the richer algebraic structure of the persistence cohomology algebra.

\subsection{The product on persistence module}\label{Persistence algebra over module}

From now on, let $S\subseteq G$ be a monoid such that the identity element $e$ is the unique invertible element in $S$. Then the group $G$ can be regarded as a poset with partial order given by $a\leq b$ if $ba^{-1}\in S$. The category $\mathrm{cat}(\mathrm{Cay}(G,S))$ is exactly the category given by the poset $G$. In this section, we assume the poset $G$ has finite product \cite{barr1999category,awodey2010category}, for example, a lattice.

\begin{lemma}
For objects $b,c\in \mathrm{cat}(\mathrm{Cay}(G,S))$, the product $b\times c$ is unique. Moreover, we have $b\times c=c\times b$.
\end{lemma}
\begin{proof}
Recall that the product in a category is unique up to isomorphism.
\begin{equation*}
  \xymatrix{
  b\times c\ar@{->}[r] \ar@{->}[d] & b\\
  c   & a\ar@{->}[l]\ar@{->}[u]\ar@{.>}[ul]\\
  }
\end{equation*}
If $a$ is also the product of $b$ and $c$. Then we have morphisms $b\times c\rightarrow a$ and $a\rightarrow b\times c$. It follows that
\begin{equation*}
  b\times c=xa,\quad a=y(b\times c),\quad x,y\in S.
\end{equation*}
This implies that $xy=yx=1$, which shows $x=y=1$ since $x,y\in S$. Thus the product $b\times c$ is unique.

On the other hand, by the universal property, there exist elements $x,y\in S$ such that
\begin{equation*}
  c\times b=x(b\times c),\quad b\times c=y(c\times b).
\end{equation*}
A straightforward calculation shows $b\times c=c\times b$.
\end{proof}
The above lemma says that $b\times c$ is the unique maximal element $\leq b,c$ under the partial order.
\begin{example}
Recall the Cayley digraph $(\mathbb{Z}^{n},\mathbb{Z}_{\geq 0}^{n})$ considered in Example \ref{exam:Cayley}($i$). Let
\begin{equation*}
  x=(x_{1},\dots,x_{n}),y=(y_{1},\dots,y_{n})\in \mathbb{Z}^{n}.
\end{equation*}
Note that $(0,\dots,0)$ is the unique invertible element in $\mathbb{Z}_{\geq 0}^{n}$, one has
\begin{equation*}
  x\times y=(\min(x_{1},y_{1}),\dots,\min(x_{n},y_{n})).
\end{equation*}
For a more intuitive interpretation, consider the Cayley digraph $(\mathbb{R}^{2},\mathbb{R}^{2}_{+})$, where $\mathbb{R}_{\geq 0}=\{x\in \mathbb{R} |x\geq 0\}$.
\begin{figure}[H]
\centering
    \begin{tikzpicture}[scale=1.6]
        \draw[step=.5cm,gray,ultra thin, dashed] (-1.4,-0.9) grid (1.4,1.4);%
        \begin{scope}[thick]

            \draw[->] (-1.5,0) -- (1.5,0);
            \draw[->] (0,-1) -- (0,1.5);
        \end{scope}

           \draw[->,thick,dashed,black!60] (1.1,0.3) -- (1.1,-1);
           \draw[->,thick,dashed,black!60] (1.1,0.3) -- (-1.3,0.3);
           \draw[->,thick,dashed,black!60] (-0.2,0.9) -- (-0.2,-1);
           \draw[->,thick,dashed,black!60] (-0.2,0.9) -- (-1.3,0.9);
\draw[fill=black!100](1.1,0.3) circle(1pt);
\draw[fill=black!100](-0.2,0.9) circle(1pt);
\draw[fill=black!100](-0.2,0.3) circle(1pt);

\node [font=\fontsize{8}{6}] (node001) at (0.1,-0.1){$O$};
\node [font=\fontsize{8}{6}] (node001) at (1.5,-0.1){$x$};
\node [font=\fontsize{8}{6}] (node001) at (0.1,1.5){$y$};
\node [font=\fontsize{8}{6}] (node001) at (-0.3,1){$b$};
\node [font=\fontsize{8}{6}] (node001) at (1.2,0.4){$c$};
\node [font=\fontsize{8}{6}] (node001) at (-0.4,0.2){$b\times c$};
        \end{tikzpicture}
\end{figure}
We take $b=(-0.2,0.9), c=(1.1,0.3)$, then $b\times c=(-0.2,0.3)$. Obviously, one has
\begin{equation*}
  b(b\times c)^{-1}=(0,0.6)\in \mathbb{R}_{\geq 0}^{2},\quad c(b\times c)^{-1}=(1.3,0)\in \mathbb{R}_{\geq 0}^{2}.
\end{equation*}
Let $(\mathbb{R}^{2},S)$ be another Cayley digraph with $S=\langle\mathbb{R}^{2}_{+},(-1,0)\rangle$. Here, $S$ is an abelian monoid generated by $\mathbb{R}^{2}_{+}$ and $(-1,0)$. By a direct calculation, we have $S=\mathbb{R}\times \mathbb{R}_{\geq 0}$. It follows that
\begin{equation*}
  b\times c=(x,0.3),\quad x\in \mathbb{R},
\end{equation*}
which is unique up to isomorphism $\rho_{x,y}:(x,0.3)\rightarrow (y,0.3)$.
\end{example}

\begin{definition}
For $b,c\in G$, the \emph{persistence-cup product} of $\beta\in H_{b}^{\ast}$ and $\gamma\in H_{c}^{\ast}$ is defined by
\begin{equation*}
  \beta\cdot\gamma=\beta\cup_{b\times c}\gamma.
\end{equation*}
\end{definition}
Indeed, the persistence-cup product defined above is a product.
\begin{lemma}\label{lemma:associated}
Let $\beta\in H_{b}^{\ast},\gamma\in H^{\ast}_{c},\delta\in H^{\ast}_{d}$. Then we have
\begin{equation*}
  (\beta\cdot\gamma)\cdot \delta=\beta\cdot(\gamma\cdot \delta).
\end{equation*}
\end{lemma}
\begin{proof}
For convenience, we denote $\phi_{b}^{c}=L_{b(b\times c)^{-1}}:H^{\ast}_{b}\rightarrow H^{\ast}_{b\times c}$. We will prove
\begin{equation*}
  (\beta\cdot\gamma)\cdot \delta=(\beta\cup_{b\times c\times d}\gamma) \cup_{b\times c\times d} \delta.
\end{equation*}
Consider the following diagram.
\begin{equation*}
  \xymatrix{
  H^{\ast}_{b}\times H^{\ast}_{c}\times H^{\ast}_{d}\ar@{->}[d]^{(\phi_{b}^{c},\phi_{c}^{b},\mathrm{id})}&&\\
  H^{\ast}_{b\times c}\times H^{\ast}_{b\times c}\times H^{\ast}_{d}\ar@{->}[r]^{(\mathrm{id},\mathrm{id},\phi_{d}^{b\times c})}\ar@{->}[d]^{(\cup,\mathrm{id})}
 &H^{\ast}_{b\times c}\times H^{\ast}_{b\times c}\times H^{\ast}_{b\times c\times d}\ar@{->}[r]^{(\phi_{b\times c}^{d},\phi_{b\times c}^{d},\mathrm{id})}\ar@{->}[d]^{(\cup,\mathrm{id})}
 &H^{\ast}_{b\times c\times d}\times H^{\ast}_{b\times c\times d}\times H^{\ast}_{b\times c\times d}\ar@{->}[d]^{(\cup,\mathrm{id})}\\
  H^{\ast}_{b\times c}\times H^{\ast}_{d}\ar@{->}[r]^{(\mathrm{id},\phi_{d}^{b\times c})}&H^{\ast}_{b\times c}\times H^{\ast}_{b\times c\times d}\ar@{->}[r]^{(\phi_{b\times c}^{d},\mathrm{id})}&H^{\ast}_{b\times c\times d}\times H^{\ast}_{b\times c\times d}\ar@{->}[d]^{\cup}\\
 &&H^{\ast}_{b\times c\times d}.
 }
\end{equation*}
This first square is commutative by definition. The second square is commutative by \cite[Proposition 3.10]{Hatcher}. Hence, by a straightforward calculation, we obtain
\begin{equation*}
  \begin{split}
    (\beta\cdot\gamma)\cdot \delta=  & \cup\circ  (\phi^{d}_{b\times c},\mathrm{id}) \circ  (\mathrm{id},\phi_{d}^{b\times c})  \circ (\cup,\mathrm{id}) \circ(\phi_{b}^{c},\phi_{c}^{b},\mathrm{id}) \\
     = & \cup\circ (\cup,\mathrm{id})\circ  (\phi^{d}_{b\times c},\phi^{d}_{b\times c},\mathrm{id}) \circ  (\mathrm{id},\mathrm{id},\phi_{d}^{b\times c})   \circ(\phi_{b}^{c},\phi_{c}^{b},\mathrm{id})\\
     =& (\beta\cup_{b\times c\times d}\gamma )\cup_{b\times c\times d} \delta.
  \end{split}
\end{equation*}
Similarly, we have $\beta\cdot(\gamma\cdot \delta)=\beta\cup_{b\times c\times d}(\gamma \cup_{b\times c\times d} \delta)$. But $(\beta\cup_{b\times c\times d}\gamma )\cup_{b\times c\times d} \delta=\beta\cup_{b\times c\times d}(\gamma \cup_{b\times c\times d} \delta)$, which gives the desired result.
\end{proof}

Let $\mathbf{H}=\bigoplus\limits_{a\in G}H^{\ast}_{a}$. Then $\mathbf{H}$ is a $\mathbf{k}$-algebra with respect to the persistence-cup product.
However, the persistence-cup product on $\mathbf{H}$ is not so friendly to us:
\begin{itemize}
  \item[$(i)$] As a $\mathbf{k}$-algebra, $\mathbf{H}$ can be an infinite-dimensional algebra even if $\mathbf{H}$ is finitely generated as a persistence module. Moreover, it is difficult for us to study the persistence structure on cohomology since the poset category $\mathrm{cat}(\mathrm{Cay}(G,S))$ may not have a terminal object.
  \item[$(ii)$] There may be too many redundant multiplications on $\mathbf{H}$.
For example, consider the $\mathrm{Cay}(\mathbb{Z}\times \mathbb{Z},\mathbb{Z}_{\geq 0}\times \mathbb{Z}_{\geq 0})$-grading.
Let $\beta\in H^{\ast}_{(0,1)}$ and $\gamma\in H^{\ast}_{(N,0)}$ for $N>>0$. Assume that we have a nontrivial product
\begin{equation*}
  \beta\cup_{(0,1)\times (N,0)} \gamma= \beta\cup_{(0,0)} \gamma.
\end{equation*}
For $0<k<N$, we have
\begin{equation*}
  H^{\ast}_{(0,1)}\times H^{\ast}_{(N,0)}\stackrel{(\mathrm{id},L_{(N-k,0)})}{\longrightarrow} H^{\ast}_{(0,1)}\times H^{\ast}_{(k,0)}\stackrel{(L_{(0,1)},L_{(k,0)})}{\longrightarrow} H^{\ast}_{(0,0)}\times H^{\ast}_{(0,0)}\stackrel{\cup}{\longrightarrow} H^{\ast}_{(0,0)}.
\end{equation*}
It follows that
\begin{equation*}
  \cup\circ (L_{(0,1)},L_{(k,0)})\circ (\mathrm{id},L_{(N-k,0)})=\cup\circ (L_{(0,1)},L_{(N-k.0)}),
\end{equation*}
which shows that $\beta\cup_{(0,0)}\gamma=\beta\cup_{(0,0)}L_{(N-k,0)}\gamma$. Thus there are too many multiplications $\beta\cup_{(0,0)}L_{(N-k,0)}\gamma,k=1,2,\cdots,N$ which are essentially the same.
\end{itemize}

\begin{theorem}\label{theorem:representation}
Let $G$ be an abelian group, and let $\mathrm{End}_{\mathbf{k}}(\mathbf{H})$ be the monoid of $\mathbf{k}$-algebra endomorphisms $\mathbf{H}\rightarrow \mathbf{H}$. Then we have a morphism of monoids $L:S\rightarrow \mathrm{End}_{\mathbf{k}}(\mathbf{H}),x\mapsto L_{x}$.
\end{theorem}
\begin{proof}
It suffices to prove
\begin{equation*}
  L_{x}(\beta\cdot\gamma)=(L_{x}\beta)\cdot(L_{x}\gamma).
\end{equation*}
Firstly, we will prove $(x^{-1}b)\times (x^{-1}c)=x^{-1}(b\times c)$ for $b,c\in G,x\in S$. Consider the following two diagrams.
\begin{equation*}
  \xymatrix{
  x^{-1}b\times x^{-1}c\ar@{->}[d]\ar@{->}[r]&x^{-1}c\\
  x^{-1}b&x^{-1}(b\times c)\ar@{->}[l]\ar@{.>}[lu]\ar@{->}[u]
  }\quad\quad
\xymatrix{
  x(x^{-1}b\times x^{-1}c)\ar@{->}[r]\ar@{.>}[rd]\ar@{->}[d]&c\\
   b&b\times c\ar@{->}[l]\ar@{->}[u]
}
\end{equation*}
The dashed arrows are given by the universal property of product. Thus there exist elements $y,z\in S$ such that
\begin{equation*}
  (x^{-1}b\times x^{-1}c)=yx^{-1}(b\times c),\quad b\times c=zx(x^{-1}b\times x^{-1}c),
\end{equation*}
which implies that $yx^{-1}zx=1$ and $zxyx^{-1}=1$. Since $G$ is abelian, one has $y=z=1$ as $y,z\in S$. This shows that $(x^{-1}b)\times (x^{-1}c)=x^{-1}(b\times c)$.
\begin{equation*}
  \xymatrix{
 H_{b}^{\ast}\times H_{c}^{\ast}\ar@{->}[rrr]^{(\phi_{b}^{c},\phi_{c}^{b})}\ar@{->}[d]^{(L_{x},L_{x})}&&&H_{b\times c}^{\ast}\times H_{b\times c}^{\ast}\ar@{->}[r]^{\cup}\ar@{->}[d]^{(L_{x},L_{x})}& H_{b\times c}^{\ast}\ar@{->}[d]^{L_{x}}\\
 H_{x^{-1}b}^{\ast}\times H_{x^{-1}c}^{\ast}\ar@{->}[rrr]^{(\phi_{x^{-1}b}^{x^{-1}c},\phi_{x^{-1}c}^{x^{-1}b})}&&&H_{x^{-1}(b\times c)}^{\ast}\times H_{x^{-1}(b\times c)}^{\ast}\ar@{->}[r]^{\cup}&  H_{x^{-1}(b\times c)}^{\ast}
 }
\end{equation*}
The first square is commutative since
\begin{equation*}
  \begin{split}
    (\phi_{x^{-1}b}^{x^{-1}c},\phi_{x^{-1}c}^{x^{-1}b})\circ (L_{x},L_{x})= &(L_{x^{-1}b(x^{-1}b\times x^{-1}c)^{-1}}, L_{x^{-1}c(x^{-1}b\times x^{-1}c)^{-1}})\circ (L_{x},L_{x})\\
     = & (L_{b(x^{-1}b\times x^{-1}c)^{-1}}, L_{c(x^{-1}b\times x^{-1}c)^{-1}})\\
     = & (L_{b[x^{-1}(b\times c)]^{-1}}, L_{c[x^{-1}(b\times c)]^{-1}})\\
     = & (L_{b(b\times c)^{-1}x}, L_{c(b\times c)^{-1}x})\\
     = & (L_{x},L_{x})\circ (L_{b(b\times c)^{-1}},L_{c(b\times c)^{-1}})\\
     = & (L_{x},L_{x})\circ (\phi_{b}^{c},\phi_{b}^{c}).
  \end{split}
\end{equation*}
The second square is commutative by \cite[Proposition 3.10]{Hatcher}. Thus we have
\begin{equation*}
  \begin{split}
    L_{x}(\beta\cdot\gamma)= & L_{x}\circ\cup\circ (\phi_{b}^{c},\phi_{b}^{c}) \\
      =&\cup \circ(\phi_{x^{-1}b}^{x^{-1}c},\phi_{x^{-1}c}^{x^{-1}b})\circ (L_{x},L_{x})\\
      =&(L_{x}\beta)\cdot(L_{x}\gamma).
  \end{split}
\end{equation*}
This completes the proof.
\end{proof}
The above theorem says that the persistence of $\mathbf{H}$ reveals the monoid structure of $S$.
In the next section, we will show the multiplication on $\mathbf{H}$ as a persistence module.
\subsection{Twisted algbera}
Let $\mathbb{L}_{S}$ be the set of elements $L_{x},x\in S$.
For cohomology morphisms, the composition is given by
\begin{equation*}
  L_{x}\circ L_{y}=L_{yx}.
\end{equation*}
We define a product on $\mathbb{L}_{S}$ by
\begin{equation*}
  L_{x}\cdot L_{y}=L_{y}\circ L_{x}.
\end{equation*}
Then $\mathbb{L}_{S}$ is a monoid and the monoid ring $\mathbf{k}[\mathbb{L}_{S}]$ is a $G$-graded ring with the grading given by $\deg L_{x}=x$. Regard $\mathbf{H}$ as a right $\mathbf{k}[\mathbb{L}_{S}]$-module given by
\begin{equation*}
  \mathbf{H}\times \mathbb{L}_{S}\rightarrow \mathbf{H},\quad (\alpha,L_{x})\mapsto \alpha\cdot L_{x}=L_{x}\alpha.
\end{equation*}
Give a $G$-graded on $\mathbf{H}$ by assigning grading to elements in $H^{\ast}_{a}$ as
\begin{equation*}
  \deg \alpha=a^{-1},\quad \alpha\in H^{\ast}_{a},a\in G.
\end{equation*}
\begin{lemma}
The $\mathbf{k}$-linear space $\mathbf{H}$ is a $G$-graded right $\mathbf{k}[\mathbb{L}_{S}]$-module.
\end{lemma}
\begin{proof}
Consider the morphisms $H_{xya}^{\ast}\stackrel{L_{x}}{\rightarrow} H_{ya}^{\ast}  \stackrel{L_{y}}{\rightarrow}  H_{a}^{\ast}$.
For $\alpha\in H_{xya}^{\ast}$, we have
\begin{equation*}
  \alpha \cdot L_{x}\cdot L_{y}=L_{y}L_{x}\alpha=L_{xy}\alpha=\alpha \cdot L_{xy}.
\end{equation*}
On the other hand, note that $H^{\ast}_{xa}\cdot L_{x}\subseteq H^{\ast}_{a}$ and $(xa)^{-1}x=a^{-1}$. Hence, $\mathbf{H}$ is a $G$-graded right $\mathbf{k}[\mathbb{L}_{S}]$-module.
\end{proof}

Note that $\mathbf{H}$ does not have to be a $\mathbf{k}[\mathbb{L}_{S}]$-algebra. To describe the multiplication on $\mathbf{H}$ as a persistence module, we introduce the twisted algebra. There are different variants of twisted algebra with the twistor evolved from the multiplication, the differential or the associative law of an algebra.
\begin{definition}
Let $R$ be a commutative ring with unit. An $R$-module $A$ is an \emph{$R$-twisted algebra} if there is a twisted multiplication $\cdot:A\times A \rightarrow A$ and a bilinear function $f:R\times R\rightarrow R$ satisfying
\begin{enumerate}
  \item[$(i)$] $\alpha\cdot(\beta\cdot\gamma)=(\alpha\cdot\beta)\cdot\gamma$ for any $\alpha,\beta,\gamma\in A$.
  \item[$(ii)$] $(\lambda \alpha)\cdot(\mu \beta)=f(\lambda ,\mu)(\alpha \cdot \beta)$ for any $\lambda,\mu\in R,\alpha,\beta\in A$.
\end{enumerate}
Let $R=\bigoplus\limits_{a\in G}R_{a}$ be a $G$-graded ring. A \emph{$G$-graded $R$-twisted algebra} $A=\bigoplus\limits_{a\in G}A_{a}$ is a $G$-graded $R$-module with twisted multiplication $\cdot$ and bilinear functions $f_{a,b}:R\times R\rightarrow R,a,b\in G$ satisfying
\begin{enumerate}
  \item[$(i)$] $\alpha\cdot(\beta\cdot\gamma)=(\alpha\cdot\beta)\cdot\gamma$ for any $\alpha,\beta,\gamma\in A$.
  \item[$(ii)$] $(\lambda \alpha)\cdot(\mu \beta)=f_{a,b}(\lambda ,\mu)(\alpha \cdot \beta)$ for any $\lambda,\mu\in R,\alpha\in A_{a},\beta\in A_{b}$.
\end{enumerate}
\end{definition}
Specially, if we take $f(\lambda,\mu)=\lambda\mu$ as the ring multiplication, then the $R$-twisted algebra is the usual $R$-algebra. Now, we will show that $\mathbf{H}$ is a $G$-graded $\mathbf{k}[\mathbb{L}_{S}]$-twisted algebra with the twisted multiplication determining the algebra structure of $\mathbf{H}$.

\begin{theorem}\label{theorem:twist}
Let $G$ be an abelian group.
The persistence-cup product on $\mathbf{H}$ is uniquely determined by the persistence-cup product of $\mathbf{k}[\mathbb{L}_{S}]$-module generators. Moreover, we have
\begin{itemize}
  \item[$(i)$] $(\mathbf{H},\cdot)$ is a $G$-graded $\mathbf{k}[\mathbb{L}_{S}]$-twisted algebra.
  \item[$(ii)$] Let $\mathbf{P}=\bigcup\limits_{a,b\in G} H^{\ast}_{b}\cdot H^{\ast}_{c}$. Then $\mathbf{P}$ is a right $\mathbf{k}[\mathbb{L}_{S}]$-module and $\mathbf{P}=\mathbf{H}\cdot \mathbf{H}$.
\end{itemize}
\end{theorem}
\begin{proof}
For given $\alpha\in H_{a},\beta\in H_{b},x,y\in S$, we will first prove that there exists an element $z\in S$ such that
\begin{equation*}
  (\alpha \cdot L_{x})\cdot (\beta\cdot L_{y})=(\alpha\cdot\beta)\cdot L_{z}.
\end{equation*}
Consider the following diagram. By the universal property of $a\times b$, we have a morphism $x^{-1}a\times y^{-1}b\to a\times b$. This shows that $a\times b=z(x^{-1}a\times y^{-1}b)$ for some $z\in S$.
\begin{equation*}
  \xymatrix{
  &a&\\
  x^{-1}a\ar@{->}[ru]&a\times b\ar@{->}[u]\ar@{->}[r]&b\\
  x^{-1}a\times y^{-1}b\ar@{.>}[ru]\ar@{->}[u]\ar@{->}[r]&y^{-1}b\ar@{->}[ru]&
  }
\end{equation*}
Thus we have
\begin{equation*}
  \begin{split}
    (\alpha \cdot L_{x})\cdot (\beta\cdot L_{y})= &  (\alpha \cdot L_{x})\cup_{x^{-1}a\times y^{-1}b} (\beta\cdot L_{y}) \\
     =& (\alpha \cdot L_{a(x^{-1}a\times y^{-1}b)^{-1}})\cup (\beta\cdot L_{b(x^{-1}a\times y^{-1}b)^{-1}})\\
     =& (\alpha \cdot L_{a(a\times b)^{-1}z})\cup (\beta\cdot L_{b(a\times b)^{-1}z})\\
     =& (\alpha \cdot\beta)\cdot L_{z}.
  \end{split}
\end{equation*}
This shows that the $\mathbf{k}$-algebra structure of $\mathbf{H}$ can be determined by the persistence-cup product on the $\mathbf{k}[\mathbb{L}_{S}]$-module structure of $\mathbf{H}$.

Assume that $\lambda=\sum\limits_{x\in G}c_{x}L_{x},\mu=\sum\limits_{y\in G}c'_{y}L_{y}\in \mathbf{k}[\mathbb{L}_{S}]$, where $c_{x},c'_{y}\in \mathbf{k}$ for all $x,y\in G$. Let
\begin{equation*}
  f_{a,b}(\lambda ,\mu)=\sum\limits_{x,y\in G} c_{x}c'_{y}L_{(a\times b)(x^{-1}a\times y^{-1}b)^{-1}}.
\end{equation*}
Then we have $(\alpha\lambda )\cdot (\beta\mu )=(\alpha\cdot\beta)f_{a,b}(\lambda ,\mu)$. By Lemma \ref{lemma:associated}, $\mathbf{H}$ is a $G$-graded $\mathbf{k}[\mathbb{L}_{S}]$-twisted algebra. Since $\mathbf{H}$ is a right $\mathbf{k}[\mathbb{L}_{S}]$-module, $\mathbf{P}$ is also a right $\mathbf{k}[\mathbb{L}_{S}]$-module. A direct calculation shows $\mathbf{P}=\mathbf{H}\cdot \mathbf{H}$, which completes the proof.
\end{proof}

\begin{remark}
Note that the elements in $\mathbf{H}$ have another grading given by the cohomology dimension. The persistence-cup product also inherits the graded commutativity of the cup product, i.e.,
\begin{equation*}
  \alpha\cdot \beta=(-1)^{ij}\beta\cdot \alpha,\quad \alpha\in H_{a}^{i},\beta\in H_{b}^{j}.
\end{equation*}
\end{remark}

\subsection{Persistence algebra on ordered groups}\label{subsection:ordergroup}
In this part, the poset category $\mathrm{cat}(\mathrm{Cay}(G,S))$ is assumed to be totally ordered, that is, for any $a,b\in G$, there exists a non-identity element $x\in S$ such that $b=xa$ or $a=xb$. In this case, we say the Cayley grading is totally ordered. For example, $(\mathbb{Z},\mathbb{Z}_{\geq 0})$ and $(\mathbb{R},\mathbb{R}_{\geq 0})$ give the totally ordered grading, which are corresponding to the usual 1-parameter persistence theory.

Recall that a \emph{left-ordered group}\cite{levi1942ordered,clay2016ordered} $(G,\leq)$ is a group $G$ equipped with an order $\leq$ such that $a\leq b$ implies $ca\leq cb$ for all $a,b,c\in G$. Dually, $(G,\leq)$ is \emph{a right-ordered group} if $a\leq b$ implies $ac\leq bc$ for all $a,b,c\in G$. We say that $(G,\leq)$ is \emph{(bi-)ordered group} if it is left-orderable and right-orderable. For a bi-ordered group, a subset $G^{+}=\{a\in G|e\leq a\}$ is a \emph{positive cone}. Here, $e$ denotes the identity element in $G$. In general, a group $G$ is not always an ordered group. In this section, our persistence theory will be built on ordered groups.

\begin{lemma}
Let $G$ be an abelian group. The poset category $\mathrm{cat}(\mathrm{Cay}(G,S))$ is assumed to be totally ordered. Then $(G,\leq)$ is a bi-ordered group with positive cone $S$. Here, the bi-order $\leq$ is given by $a\leq b$ if there exists an element $x\in S$ such that $b=xa$
\end{lemma}
\begin{proof}
It can be verified that $\leq$ is indeed an order. If $a\leq b$, then we will have $b=xa$. Since $G$ is abelian, it follows that $cb=xca$. Thus one obtains $ca\leq cb$. Let $G^{+}=\{a\in G|e\leq a\}$. By definition, $a\in G^{+}$ if and only if $a\in S$.
\end{proof}
Ordered groups are important objects of study in the fields of algebra, number theory, combinatorics and topology, etc.

\begin{example}
\begin{enumerate}
  \item[$(i)$] In addition to $\mathbb{Z}$, $\mathbb{R}$ and $\mathbb{Q}$, the algebraic number field $\mathbb{A}$ and the ring $\mathbb{Z}[1,\sqrt{2}]$ are also bi-ordered groups with the order from reals numbers. Moreover, an additive subgroup of $\mathbb{R}$ is bi-orderable.
  \item[$(ii)$] Free groups are bi-orderable.
  \item[$(iii)$] An ordering braid group $\mathcal{B}_{n}$ is a group with generators $\sigma_{1},\dots,\sigma_{n-1}$ and relations
\begin{eqnarray*}
  \sigma_{i}\sigma_{j} &=& \sigma_{j}\sigma_{i},\quad |i-j|>1 \\
  \sigma_{i}\sigma_{i+1}\sigma_{i} &=& \sigma_{i+1}\sigma_{i}\sigma_{i+1},\quad 1\leq i\leq n-2.
\end{eqnarray*}
Then $\mathcal{B}_{n}$ is left-orderable \cite{clay2016ordered}.
\end{enumerate}
\end{example}

Let $S_{+}=S\setminus\{e\}$. It is worth noting that $G$ is the disjoint union of $S_{+}$, $S_{+}^{-1}$ and $\{e\}$. Now, we will identify the totally ordered Cayley digraphs $\mathrm{cat}(\mathrm{Cay}(G,S))$ with the ordered group $(G,\leq)$ and frequently use the order $\leq$ for convenience.

From now on, for a Cayley digraph $(G,S)$, the set $S$ considered is assumed to be a Hausdorff topological monoid with the topology generated by the open sets of the following forms
\begin{equation*}
  \{x|a<x<b\},\{x|a<x\},\{x|x<b\},\quad a,b\in S.
\end{equation*}
Here, $a<x$ means that $a\leq x$ but $a\neq x$. By Lemma \ref{lemma:topology}, the morphism of monoids
\begin{equation*}
  \phi:S\rightarrow \mathcal{R}_{S},\quad  x\mapsto L_{x}
\end{equation*}
induces a topology on $\mathcal{R}_{S}$. Moreover, $\mathcal{R}_{S}$ is also a topological monoid. Since $\phi$ is a quotient map, $V\subseteq \mathcal{R}_{S}$ is open if and only if $\phi^{-1}(V)$ is open. Equivalently, $V\subseteq \mathcal{R}_{S}$ is closed if and only if $\phi^{-1}(V)$ is closed.
\begin{lemma}\label{lemma:closure}
Let $U$ be a subset of $S$. Then we have $\phi(\overline{U})=\overline{\phi(U)}$.
\end{lemma}
\begin{proof}
Note that $\phi(U)\subseteq \phi(\overline{U})$. Since $\phi(\overline{U})$ is closed, we have $\overline{\phi(U)}\subseteq \phi(\overline{U})$. On the other hand, suppose that $\phi(\overline{U})\not\subseteq\overline{\phi(U)}$. Then there exists an element $x\in \overline{U}$ such that $\phi(x)\notin \overline{\phi(U)}$. Since $S$ is Hausdorff, there is an open neighborhood $V$ of $\phi(x)$ such that $V\cap \phi(U)=\emptyset$. Note that $\phi^{-1}(V)$ is open, then $\phi^{-1}(V)$ is an open neighborhood of $x$ and $\phi^{-1}(V)\cap U=\emptyset$.
This contradicts the fact $x\in \overline{U}$, which leads to $\phi(\overline{U})\subseteq \overline{\phi(U)}$.
\end{proof}

\begin{lemma}\label{lemma:inf}
Let $G$ be an ordered group, and let $S$ be the positive cone of $G$.
Suppose $I$ is a closed subset of $S$ such that $SI=I$. Then there exists an element $a\in I$ such that
\begin{equation*}
  I=\{b\in S|a\leq b\}=Sa.
\end{equation*}
\end{lemma}
\begin{proof}
Obviously, the identity element $e$ is the lower bound of $S$ with respect to the order $\leq$. Then there is an infimum of $I$ and we denote $a=\inf(I)$. If $a\in I$, then for any $b\in S,a<b$, we have $b\in I$.
It follows that $\{b\in S|a\leq b\}\subseteq I$. Note that $a$ is the infimum of $I$, we obtain $I=\{b\in S|a\leq b\}\subseteq I$.
Suppose $a\notin I$. Then, for any $b\in S,a<b$, there exists an element $a'\in S$ such that $a'<b$. Otherwise, $b$ is the infimum of $I$. Thus we have $b\in S$ since $SI=I$. It follows that
\begin{equation*}
  \{c\in S|a<c\}\subseteq I.
\end{equation*}
Note that $I_{a}=\{c\in S|a\leq c\}$ is a closed subset of $S$ containing $I$. We have
\begin{equation*}
 I_{a}\backslash\{a\}=\{c\in S|a<c\}\subseteq I\subseteq I_{a}.
\end{equation*}
It follows that $I_{a}=I$. We complete the proof.
\end{proof}

\begin{theorem}\label{thm:decomclousure}
Let $G$ be an ordered abelian group. Let $\mathbf{H}$ be a finitely generated $\mathbf{k}[\mathbb{L}_{S}]$-module.
Then we have a finite direct sum decomposition
\begin{equation*}
   \mathbf{H}\cong \bigoplus_{i=1}^{k}  e^{i}_{x_{i}} \cdot \mathbf{k}[\mathbb{L}_{S}]\oplus \bigoplus_{j=1}^{l}\frac{\varepsilon^{j}_{y_{j}}\cdot \mathbf{k}[\mathbb{L}_{S}]}{\mathbf{k}[I_{j}]}
\end{equation*}
for some $k,l$. Here, $\overline{I_{j}}$ is the closure of $I_{j}$ and  $\mathbf{k}[\overline{I_{j}}]=\varepsilon^{j}_{y_{j}}\cdot L_{w_{j}}\cdot R$ for some $w_{j}\in S,j=1,\dots,l$.
\end{theorem}

\begin{proof}
By Corollary \ref{corollary:decomposition}, we have a finite direct sum decomposition for the version of right module
\begin{equation}\label{proof:decomposition}
   \mathbf{H}\cong \bigoplus_{i=1}^{k}  e^{i}_{x_{i}} \cdot \mathbf{k}[\mathbb{L}_{S}]\oplus \left(\left(\bigoplus_{j=1}^{l}\varepsilon^{j}_{y_{j}}\cdot \mathbf{k}[\mathbb{L}_{S}]\right)/N\right)
\end{equation}
for some $k,l$, where $e^{i}_{x_{i}}\in H^{\ast}_{x_{i}},\varepsilon^{j}_{y_{j}}\in H^{\ast}_{y_{j}}$ and $N$ is a right $\mathbf{k}[\mathbb{L}_{S}]$-module generated by the elements of the form $\varepsilon^{t}_{y_{t}}\cdot L_{y}-\varepsilon^{s}_{y_{s}}\cdot L_{y_{s}y_{t}^{-1}y}$ or $\varepsilon^{t}_{y_{t}}\cdot L_{y}$ for some $1\leq s,t\leq l,y\in S$.
We may assume that $y_{1}\leq y_{2}\leq \cdots\leq y_{l}$. Let
\begin{equation*}
  \omega_{st}^{y}=\varepsilon^{t}_{y_{t}}\cdot L_{y}-\varepsilon^{s}_{y_{s}}\cdot L_{y_{s}y_{t}^{-1}y},\quad 1\leq t<s\leq l,y\in S
\end{equation*}
and
\begin{equation*}
  \omega_{l+1,t}^{y}=\varepsilon^{t}_{y_{t}}\cdot L_{y},\quad 1\leq t\leq l,y\in S.
\end{equation*}
We denote
\begin{equation*}
  I_{st}=\{y\in S|\pi(\omega_{st}^{y})=0\},\quad 1\leq  t<s \leq l+1,
\end{equation*}
where $\pi$ is defined in the proof of Theorem \ref{thm:module_decomposition}. Note that $I_{st}$ can be an empty set for some $s,t$.
Let $\mathcal{R}_{I_{st}}=\{L_{y},y\in I_{st}\}$. By Lemma \ref{lemma:closure}, we have that
\begin{equation*}
  \overline{\mathcal{R}_{I_{st}}}=\mathcal{R}_{\overline{I_{st}}}.
\end{equation*}
Let $\mathcal{U}=\bigcup\limits_{s,t}\mathcal{R}_{I_{st}}$. It follows that
\begin{equation*}
  \overline{\mathcal{U}}=\overline{\bigcup\limits_{s,t}\mathcal{R}_{I_{st}}}=\bigcup\limits_{s,t}\overline{\mathcal{R}_{I_{st}}}=\bigcup\limits_{s,t}\mathcal{R}_{\overline{I_{st}}}.
\end{equation*}
One has
\begin{equation*}
  \mathbf{k}[\overline{\mathcal{U}}]=\mathbf{k}[\bigcup\limits_{s,t}\mathcal{R}_{\overline{I_{st}}}]=\sum\limits_{s,t}\mathbf{k}[\mathcal{R}_{\overline{I_{st}}}].
\end{equation*}
If $I_{st}$ is nonempty, by Lemma \ref{lemma:inf}, we have
\begin{equation*}
  \overline{I_{st}}=\{b\in S|a_{st}\leq b\}=Sa_{st}.
\end{equation*}
for some $a_{st}\in S$. It follows that $\mathbf{k}[\overline{\mathcal{U}}]$ is a finitely generated right $\mathbf{k}[\mathbb{L}_{S}]$-module. Moreover, the generating set is given by the elements of the following form
\begin{equation*}
  \omega_{s_{1}t_{1}}^{a_{s_{1}t_{1}}},\omega_{s_{2}t_{2}}^{a_{s_{2}t_{2}}},\dots,\omega_{s_{m}t_{m}}^{a_{s_{m}t_{m}}},
\end{equation*}
where $1\leq  t_{i}<s_{i} \leq l+1$ for $1\leq i\leq m$. We will prove the generating set can be chosen such that $t_{i}\neq t_{j}$ for any $1\leq i,j\leq m$. Suppose that $t_{i}=t_{j}$. We may assume that $a_{s_{i}t_{i}}\geq a_{s_{j}t_{j}}$ and set $a_{s_{i}t_{i}}=z a_{s_{j}t_{j}}$.
If $y_{s_{i}}\geq y_{s_{j}}$, we have
\begin{equation*}
  \omega_{s_{i}t_{i}}^{a_{s_{i}t_{i}}}-\omega_{s_{j}t_{j}}^{a_{s_{j}t_{j}}}\cdot L_{z}=(\varepsilon_{y_{s_{j}}}^{s_{j}}-\varepsilon_{y_{s_{i}}}^{s_{i}}\cdot L_{y_{s_{i}} y_{s_{j}}^{-1}})\cdot L_{y_{s_{j}}y_{t_{j}}^{-1}a_{s_{i}t_{i}}}=\omega_{s_{j}s_{i}}^{y_{s_{j}}y_{t_{j}}^{-1}a_{s_{i}t_{i}}}.
\end{equation*}
Thus the elements
\begin{equation*}
  \omega_{s_{1}t_{1}}^{a_{s_{1}t_{1}}},\dots,\omega_{s_{i-1}t_{i-1}}^{a_{s_{i-1}t_{i-1}}},\omega_{s_{j}s_{i}}^{y_{s_{j}}y_{t_{j}}^{-1}a_{s_{i}t_{i}}},\omega_{s_{i+1}t_{i+1}}^{a_{s_{i+1}t_{i+1}}},\dots,\omega_{s_{m}t_{m}}^{a_{s_{m}t_{m}}}
\end{equation*}
consist a generating set of $\mathbf{k}[\overline{\mathcal{U}}]$.
If $y_{s_{i}}\leq y_{s_{j}}$, the set
\begin{equation*}
  \omega_{s_{1}t_{1}}^{a_{s_{1}t_{1}}},\dots,\omega_{s_{j-1}t_{j-1}}^{a_{s_{j-1}t_{j-1}}},\omega_{s_{i}s_{j}}^{y_{s_{i}}y_{t_{i}}^{-1}a_{s_{j}t_{j}}},\omega_{s_{j+1}t_{j+1}}^{a_{s_{j+1}t_{j+1}}},\dots,\omega_{s_{m}t_{m}}^{a_{s_{m}t_{m}}}
\end{equation*}
is a generating set of $\mathbf{k}[\overline{\mathcal{U}}]$. Note that $s_{i},s_{j}> t_{i}=t_{j}$. By finite steps of the above progress, we can reduce the finite generating set of $\mathbf{k}[\overline{\mathcal{U}}]$ as $\mathbf{k}[\mathbb{L}_{S}]$-module to the generating set
\begin{equation*}
  \omega_{s_{1}t_{1}}^{a_{s_{1}t_{1}}},\omega_{s_{2}t_{2}}^{a_{s_{2}t_{2}}},\dots,\omega_{s_{m}t_{m}}^{a_{s_{m}t_{m}}},
\end{equation*}
such that $t_{i}\neq t_{j}$ for any $1\leq i,j\leq m$. Now, for $1\leq t\leq l$, let
\begin{equation*}
  \tilde{\varepsilon}_{y_{t}}^{t}=\left\{
                                    \begin{array}{ll}
                                      \omega_{s_{i}t_{i}}^{1}, & \hbox{$t=t_{i}$ for some $i$;} \\
                                      \varepsilon_{y_{t}}^{t}, & \hbox{otherwise.}
                                    \end{array}
                                  \right.
\end{equation*}
Then $\{\tilde{\varepsilon}^{j}_{y_{j}}\}_{1\leq j\leq l}$ is also a $\mathbf{k}[\mathbb{L}_{S}]$-module basis for $\bigoplus\limits_{j=1}^{l}\varepsilon^{j}_{y_{j}}\cdot \mathbf{k}[\mathbb{L}_{S}]$. Moreover, we have
\begin{equation*}
  \left(\bigoplus_{j=1}^{l}\varepsilon^{j}_{y_{j}}\cdot \mathbf{k}[\mathbb{L}_{S}]\right)/N\cong  \left(\bigoplus_{j=1}^{l}\tilde{\varepsilon}^{j}_{y_{j}}\cdot \mathbf{k}[\mathbb{L}_{S}]\right)/\mathbf{k}[\mathcal{U}].
\end{equation*}
Here, $\mathbf{k}[\overline{\mathcal{U}}]$ is generated by
\begin{equation*}
  \tilde{\varepsilon}^{t_{1}}_{y_{t_{1}}}\cdot L_{a_{s_{1}t_{1}}},\tilde{\varepsilon}^{t_{2}}_{y_{t_{2}}}\cdot L_{a_{s_{2}t_{2}}},\dots,\tilde{\varepsilon}^{t_{m}}_{y_{t_{m}}}\cdot L_{a_{s_{m}t_{m}}}.
\end{equation*}
Note that $\mathbf{k}[\mathcal{U}]$ is a submodule of $\bigoplus\limits_{i=1}^{m} \tilde{\varepsilon}^{t_{i}}_{y_{t_{i}}}$. By the isomorphism (\ref{proof:decomposition}), we have
\begin{equation}
   \mathbf{H}\cong \bigoplus_{i=1}^{k}  e^{i}_{x_{i}} \cdot \mathbf{k}[\mathbb{L}_{S}]\oplus \bigoplus_{j\neq t_{1},\dots,t_{m}}\tilde{\varepsilon}^{j}_{y_{j}}\cdot \mathbf{k}[\mathbb{L}_{S}]\oplus \left(\left(\bigoplus_{i=1}^{m}\tilde{\varepsilon}^{t_{i}}_{y_{t_{i}}}\cdot \mathbf{k}[\mathbb{L}_{S}]\right)/\mathbf{k}[\mathcal{U}]\right).
\end{equation}
Rewriting the above isomorphism, we obtain the desired result.
\end{proof}

In particular, if $S\cong \mathbb{N}$ as monoid, the interval topology on $S$ is identified with a discrete topology. We have the following corollary.
\begin{corollary}[\cite{Zomorodian2005}]\label{corollary:posetgroup}
Let $G$ be an ordered abelian group. Let $\mathbf{H}$ be a finitely generated $\mathbf{k}[\mathbb{L}_{S}]$-module. If the monoid $S$ is isomorphic to $\mathbb{N}$, then we have
\begin{equation*}
   \mathbf{H}\cong \left(\bigoplus_{i=1}^{k}  e^{i}_{x_{i}} \cdot \mathbf{k}[\mathbb{L}_{S}]\right)\oplus \left(\bigoplus_{j=1}^{l}\frac{\varepsilon^{j}_{y_{j}}\cdot \mathbf{k}[\mathbb{L}_{S}]}{\varepsilon^{j}_{y_{j}}\cdot L_{w_{j}}\cdot \mathbf{k}[\mathbb{L}_{S}]}\right)
\end{equation*}
for some $k,l$ and $w_{j}\in S,j=1,\dots,l$.
\end{corollary}

\begin{remark}
The above corollary is essentially the classic result of A. Zomorodian and G. Carlsson for the case that  $G=\mathbb{Z}$ \cite[Theorem 2.1]{Zomorodian2005}. Suppose that $G=\mathbb{Z}[1,\sqrt{2}],S=\mathbb{Z}[1,\sqrt{2}]^{+}$. It is worth noting that the interval topology on $\mathbb{Z}[1,\sqrt{2}]^{+}$ is not the discrete topology. Thus we do not always have $\mathbf{k}[I_{j}]=\varepsilon^{j}_{y_{j}}\cdot L_{w_{j}}\cdot R$ to reduce the decomposition in Theorem \ref{thm:decomclousure}.
\end{remark}
Corollary \ref{corollary:posetgroup} says that $\mathbf{H}=\mathbf{F}\oplus \mathbf{T}$, where $\mathbf{F}$ is the free component and $\mathbf{T}$ is corresponding to the torsion part. For cohomology, the free generator $e^{i}_{x_{i}}$ denotes the generator  come into existence at $-\infty$ and survives to the parameter $x_{i}$. And the torsion $\varepsilon^{j}_{y_{j}}$ denotes the generator appeared at time $w_{j}^{-1}y_{j}$ and  persists until the parameter $y_{j}$.
\begin{example}
The persistent cohomology is indeed a copersistence abelian group. It is  difficult to understand the time from $-\infty$ in the practical sense. One interpretation is that the persistent cohomology prefers to describe the evolutionary process of progressive smaller space. For example, let $\{K_{i}\}_{i\in \mathbb{Z}_{\geq 0}}$ be a family of simplicial complexes such that
$f_{ij}:K_{j}\hookrightarrow K_{i}$ is an inclusion for $i\leq  j$. Then we have a morphism
\begin{equation*}
  f_{i,j}^{\ast}:H^{\ast}(K_{i})\rightarrow H^{\ast}(K_{j}),\quad i\leq j.
\end{equation*}
The $(i,j)$-persistent cohomology $H^{\ast}_{i,j}$ is defined to be $\mathrm{im} (f^{\ast}_{i,j})$. In this example, we change the direction of time to interpret the persistent cohomology, that is, as the parameter $t$ becomes larger from a given parameter $t_{0}$ to $+\infty$, the simplicial complex $K_{t}$ becomes smaller.
\end{example}
For the integrity and strictness, we outline the persistent cohomology theory. Firstly, we restate the definition of $\mathcal{P}$-interval given by A. Zomorodian and G. Carlsson for the copersistent version. We consider the case that $S$ is a finitely generated monoid.
\begin{definition}[$\mathcal{P}$-interval]
Let $\widehat{G}$ be the completion and cocompletion of  $\mathrm{cat}(\mathrm{Cay}(G,S))$ with the initial object $-\infty$ and terminal object $\infty$.
A \emph{$\mathcal{P}$-interval} $(a, b)$ is  an ordered
pair
\begin{equation*}
  a\leq b,\quad a,b\in \widehat{G}.
\end{equation*}
\end{definition}
We can endow the generators $e^{i}_{x_{i}}$ and $\varepsilon^{j}_{y_{j}}$ with $\mathcal{P}$-intervals as
\begin{equation*}
  (-\infty,x_{i}),\quad  (w_{j}^{-1}y_{j},y_{j}),
\end{equation*}
respectively. We define $Q(x,y)=(L_{x}\cdot \mathbf{k}[\mathbb{L}_{S}])/(L_{y}\cdot \mathbf{k}[\mathbb{L}_{S}])$. By Corollary \ref{corollary:posetgroup}, there is a bijection between the finite sets of $\mathcal{P}$-intervals and the finitely generated graded modules over $\mathbf{k}[\mathbb{L}_{S}]$.
The $\mathcal{P}$-intervals $(-\infty ,x_{i}),(w_{j}^{-1}y_{j},y_{j})$ form the \emph{barcode} of $\mathbf{H}$.
We say that a $\mathcal{P}$-interval $(x,y)\supseteq (x',y')$ if $x\leq x'$ and $y'\leq y$. The fundamental characterization of barcodes for cohomology is as follows.
\begin{theorem}[\cite{Zomorodian2005}]\label{theorem:carlsson2}
Let $\mathbf{H}^{n}=\bigoplus\limits_{a\in G}H^{n}_{a}$. Then the rank of $H^{n}_{a,b}$ is equal to the number of  $\mathcal{P}$-intervals  in the barcode of $\mathbf{H}^{n}$ containing $(a,b)$.
\end{theorem}
\begin{proof}
A $\mathcal{P}$-interval $(u,v)\supseteq (a,b)$ gives a direct sum component $e_{v}\cdot L_{vu^{-1}}\cdot R$ of $\mathbf{H}^{n}$, which represents a generator in $\mathbf{H}^{n}$ coming into existence at $u$ and disappearing at $v$. Thus the generators  survived from the time $a$ to the time $b$ are one-one corresponding to the intervals containing $(a,b)$.
\end{proof}

Now, we consider the case that $S$ dose not have to be a finitely generated monoid. We will use that notation ``survival space'' which is essential from the idea of $\mathcal{P}$-interval.
For an element $\alpha\in \mathbf{H}$, let $\alpha(x)$ denote the $x$-degree component of the Cayley graded right $\mathbf{k}[\mathbb{L}_{S}]$-module $\alpha\cdot \mathbf{k}[\mathbb{L}_{S}]$. Then we have $\alpha=\bigoplus\limits_{x\in G}\alpha(x)$.
The \emph{survival space} of $\alpha$ is defined as
\begin{equation*}
  I(\alpha)=\{x|\alpha(x)\neq 0,x\in G\}.
\end{equation*}
\begin{lemma}
$I(e^{i}_{x_{i}})=\{x|a\leq x_{i}\}$ and  $\overline{I(\varepsilon^{j}_{y_{j}})}=\{x|w_{j}y_{j}\leq x\leq y_{j}\}$.
\end{lemma}
\begin{proof}
It is a direct result of Theorem \ref{thm:decomclousure}.
\end{proof}
For convenience, we rewrite the generators $e_{x_{1}}^{1},\dots,e_{x_{k}}^{k},\varepsilon_{y_{1}}^{1},\dots,\varepsilon_{y_{l}}^{l}$ of $\mathbf{H}$ as $\alpha^{1},\alpha^{2},\dots,\alpha^{q}$.
By Theorem \ref{thm:decomclousure}, we can write
\begin{equation*}
  \overline{I(\alpha^{i})}=[a_{i},b_{i}],\quad i=1,\dots,q,
\end{equation*}
where $a_{i},b_{i}\in \widehat{G}$.
\begin{lemma}\label{lemma:interval}
Let $\alpha=\sum\limits_{i\in J}\lambda_{i}\alpha^{i}\cdot L_{z_{i}}$, where $J$ is a finite index set and $\lambda_{i}\neq 0\in \mathbf{k}, L_{z_{i}}\in \mathbf{k}[\mathbb{L}_{S}]$.
Then we obtain $I(\alpha)=\bigcup\limits_{i\in J}I(\alpha^{i})$.
\end{lemma}
\begin{proof}
If $x\in \bigcup\limits_{i\in J}I(\alpha^{i})$, then $x\in I(\alpha^{i})$ for some $i$. Suppose that $\alpha(x)=0$. It follows that
\begin{equation*}
  \alpha(x)=\sum\limits_{i\in J}\lambda_{i}\alpha^{i}(z_{i}x)\cdot L_{z_{i}}=\sum\limits_{i\in J}\lambda_{i}\alpha^{i}(x)=0.
\end{equation*}
Thus we have that $\alpha^{i}(x)=0$ by the linear independence. This contradicts the fact that $x\in I(\alpha^{i})$.
Conversely, if $x\notin \bigcup\limits_{i\in J}I(\alpha^{i})$, then $\alpha^{i}(x)=0$ for all $i$. Hence one has $\alpha(x)=0$.
\end{proof}

By Theorem \ref{theorem:twist}, the persistence-cup product on the $G$-graded $\mathbf{k}[\mathbb{L}_{S}]$-twisted algebra $\mathbf{H}$ is determined by the product on the basis. The following theorem says that the persistence-cup product can read some persistent information of a cup product element. For a subset $U\in G$, we denote
\begin{equation*}
  \sup(U)=\coprod\limits_{x\in U}x,\quad\inf(U)=\prod\limits_{x\in U}x
\end{equation*}
the supremum and infimum of $U$ with respect to the order $\leq$ in $\hat{G}$, respectively.
\begin{theorem}\label{theorem:main}
Let $G$ be an ordered abelian group. If $\alpha^{i}\cdot\alpha^{j}$ is nontrivial, then
\begin{equation*}
  \min(b_{i},b_{j})\leq \sup(I(\alpha^{i}\cdot\alpha^{j})),\quad \max(a_{i},a_{j})\leq \inf(I(\alpha^{i}\cdot\alpha^{j})).
\end{equation*}
\end{theorem}
\begin{proof}
$(i)$ Suppose that
\begin{equation*}
  \alpha^{i}\cdot\alpha^{j}=\sum\limits_{k\in J}\lambda_{k}\alpha^{k}\cdot L_{z_{k}},\quad\lambda_{k}\in \mathbf{k}, z_{k}\in S
\end{equation*}
for some finite index set $J$. By Lemma \ref{lemma:interval}, we have $I(\alpha^{i}\cdot\alpha^{j})=\bigcup\limits_{i\in J}I(\alpha^{i})$.
Note that the grading of $\alpha^{i}\cdot\alpha^{j}$ is
\begin{equation*}
  \deg (\alpha^{i}\cdot\alpha^{j})=b_{i}\times b_{j}=\min(b_{i},b_{j})=z_{k}^{-1}b_{k}
\end{equation*}
for any $k\in J$. We have that
\begin{equation*}
  \sup(I(\alpha^{i}\cdot\alpha^{j}))=\coprod_{k\in J}b_{k}=\coprod_{k\in J}\min(b_{i},b_{j}) z_{k}=\min(b_{i},b_{j})\coprod_{k\in J} z_{k}.
\end{equation*}
Since $J$ is finite, we have $\min(b_{i},b_{j})^{-1}\sup(I(\alpha^{i}\cdot\alpha^{j}))=\coprod\limits_{k\in J} z_{k}\in S$. It follows that
\begin{equation*}
  \min(b_{i},b_{j})\leq \sup(I(\alpha^{i}\cdot\alpha^{j})).
\end{equation*}

$(ii)$ Suppose that $\alpha^{i}\cdot L_{x}=0$. By the proof of Theorem \ref{theorem:twist}, we have
\begin{equation*}
  0=(\alpha^{i}\cdot L_{x})\cdot \alpha^{j}=(\alpha^{i}\cdot\alpha^{j})\cdot L_{(x^{-1}b_{i}\times b_{j})^{-1}(b_{i}\times b_{j})}.
\end{equation*}
It follows that $(x^{-1}b_{i}\times b_{j})^{-1}(b_{i}\times b_{j})> c_{k}^{-1}z_{k}^{-1}b_{k}=c_{k}^{-1}(b_{i}\times b_{j})$ for any $c_{k}\in I(\alpha_{k})$. One has that $x^{-1}b_{i}\times b_{j}<c_{k}$ for any $c_{k}\in I(\alpha_{k})$.

We first consider the case $I(\alpha^{i})=\{x\in G| a_{i}<x\leq b_{i}\}$. Then we have $\alpha^{i}\cdot L_{a_{i}^{-1}b_{i}}=0$.
It follows that $a_{i}\times b_{j}<c_{k}$ for any $c_{k}\in I(\alpha^{k})$. Hence, we obtain $a_{i}\times b_{j}\leq a_{k}$ for any $k\in J$.  Note that $\alpha^{i}\cdot\alpha^{j}$ is nontrivial. We have $a_{i}\times b_{j}=a_{i}$, which shows that $a_{i}\leq a_{k}$ for any $k\in J$. It follows that $a_{i}\leq \prod\limits_{k\in J} a_{k}=\inf (I(\alpha^{i}\cdot\alpha^{j}))$.
Now, we consider the case $I(\alpha^{i})=\{x\in G| a_{i}\leq x\leq b_{i}\}$. Then we have $\alpha^{i}\cdot L_{x}=0$ for any $x>a_{i}^{-1}b_{i}$.
It follows that $x^{-1}b_{i}\times b_{j}<c_{k}$ for any $c_{k}\in I(\alpha^{k})$. Since $x^{-1}b_{i}<a_{i}$, we have $x^{-1}b_{i}<c_{k}$ for any $c_{k}\in I(\alpha^{k})$ and $x>a_{i}^{-1}b_{i}$.
It follows that $a_{i}\leq c_{k}$ for any $c_{k}\in I(\alpha^{k})$, which leads to $a_{i}\leq \inf (I(\alpha^{i}\cdot\alpha^{j}))$.
Similarly, we have $a_{j}\leq \inf (I(\alpha^{i}\cdot\alpha^{j}))$. Hence, we obtain
\begin{equation*}
  \max(a_{i},a_{j})\leq \inf (I(\alpha^{i}\cdot\alpha^{j})).
\end{equation*}
The desired result follows.
%
\end{proof}

\begin{example}
Consider the filtrations of torus as follows.
\begin{itemize}
  \item[$(i)$] Grow from a rectangle to a cylinder and then to a torus.
   \begin{figure}[H]
\centering
  \subfigure[$\mathcal{K}_{0}$]{
\begin{minipage}[t]{0.30\textwidth}
\centering
\includegraphics[width=1.6in]{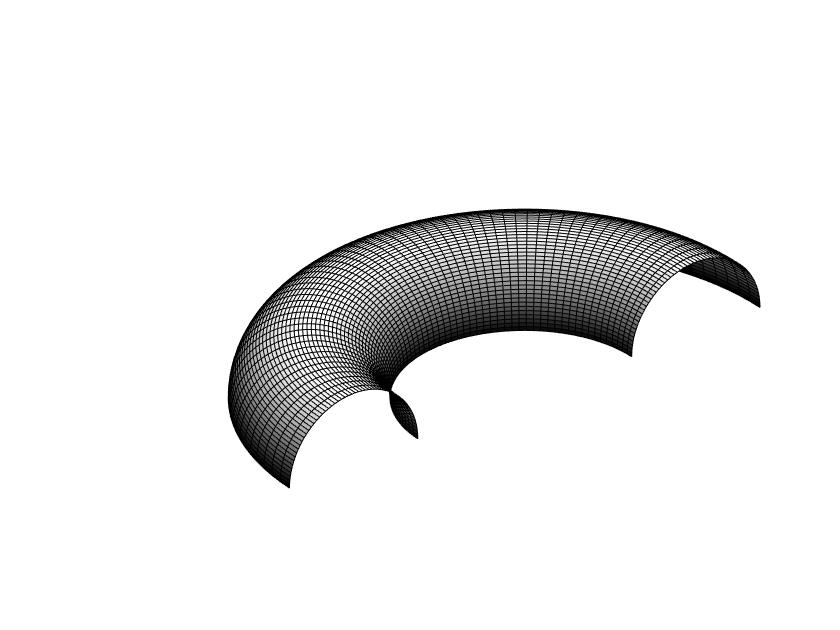}
\end{minipage}
}  \subfigure[$\mathcal{K}_{1}$]{
\begin{minipage}[t]{0.30\textwidth}
\centering
\includegraphics[width=1.6in]{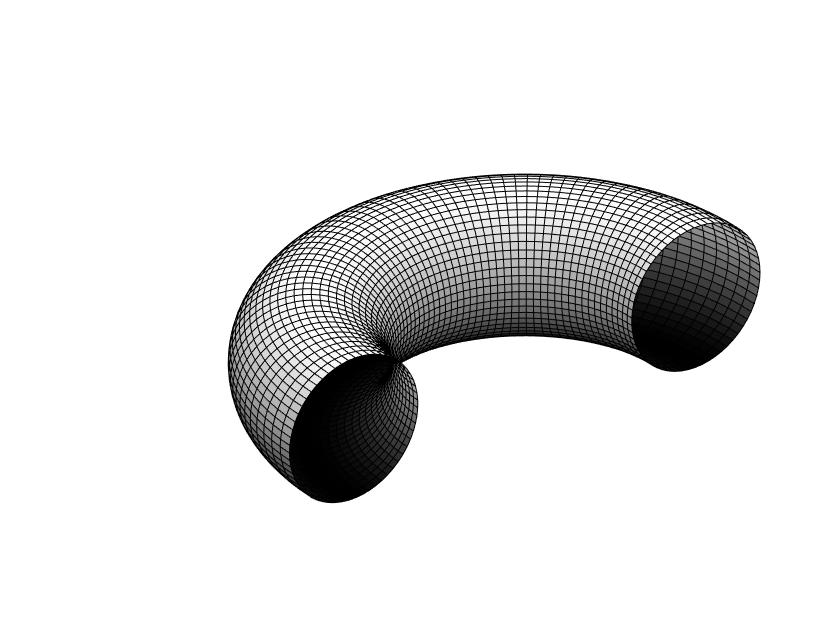}
\end{minipage}
}  \subfigure[$\mathcal{K}_{2}$]{
\begin{minipage}[t]{0.30\textwidth}
\centering
\includegraphics[width=1.6in]{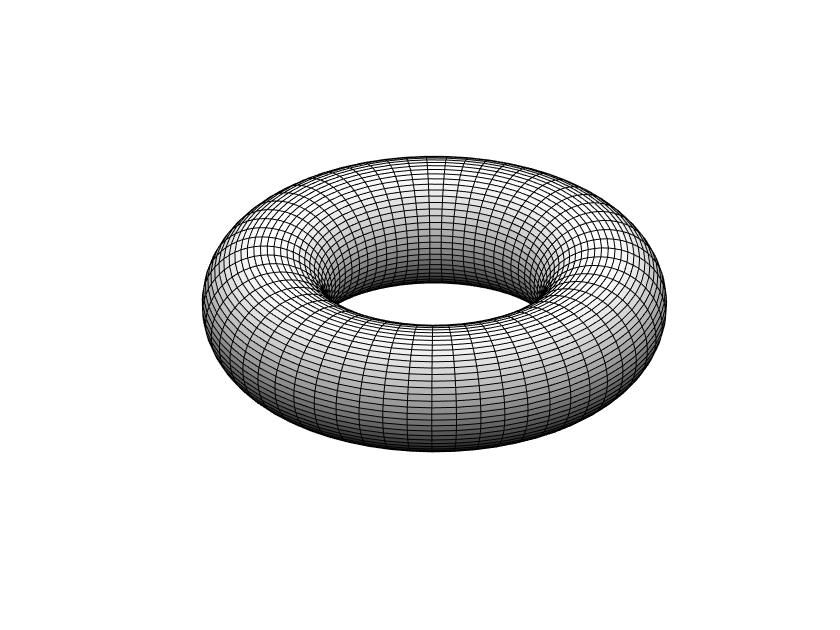}
\end{minipage}
}
\end{figure}
We set $\mathcal{K}_{0}\simeq pt$, $\mathcal{K}_{1}\simeq S^{1}$, $\mathcal{K}_{2}\simeq S^{1}\times S^{1}$, and $\mathcal{K}_{3}\simeq S^{1}\times S^{1}$.
  \item[$(ii)$] Grow from a rectangle to an incomplete torus, and finally come into a torus.
   \begin{figure}[H]
\centering
  \subfigure[$\mathcal{K}'_{0}$]{
\begin{minipage}[t]{0.30\textwidth}
\centering
\includegraphics[width=1.6in]{tori2.jpg}
\end{minipage}
}  \subfigure[$\mathcal{K}'_{1}$]{
\begin{minipage}[t]{0.30\textwidth}
\centering
\includegraphics[width=1.6in]{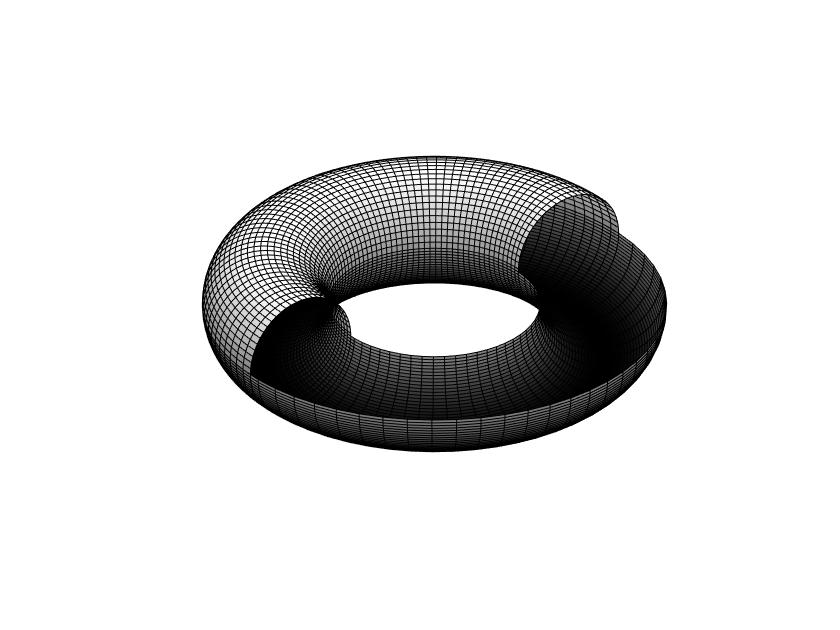}
\end{minipage}
}  \subfigure[$\mathcal{K}'_{2}$]{
\begin{minipage}[t]{0.30\textwidth}
\centering
\includegraphics[width=1.6in]{tori1.jpg}
\end{minipage}
}
\end{figure}
We set $\mathcal{K}'_{0}\simeq pt$, $\mathcal{K}'_{1}\simeq S^{1}\vee S^{1}$, $\mathcal{K}'_{2}\simeq S^{1}\times S^{1}$, and $\mathcal{K}'_{3}\simeq S^{1}\times S^{1}$.
  \item[$(iii)$] Start with a torus and gradually fill the horizontal and vertical circles.
   \begin{figure}[H]
\centering
  \subfigure[$\mathcal{K}''_{0}$]{
\begin{minipage}[t]{0.30\textwidth}
\centering
\includegraphics[width=1.6in]{tori1.jpg}
\end{minipage}
}  \subfigure[$\mathcal{K}''_{1}$]{
\begin{minipage}[t]{0.30\textwidth}
\centering
\includegraphics[width=1.6in]{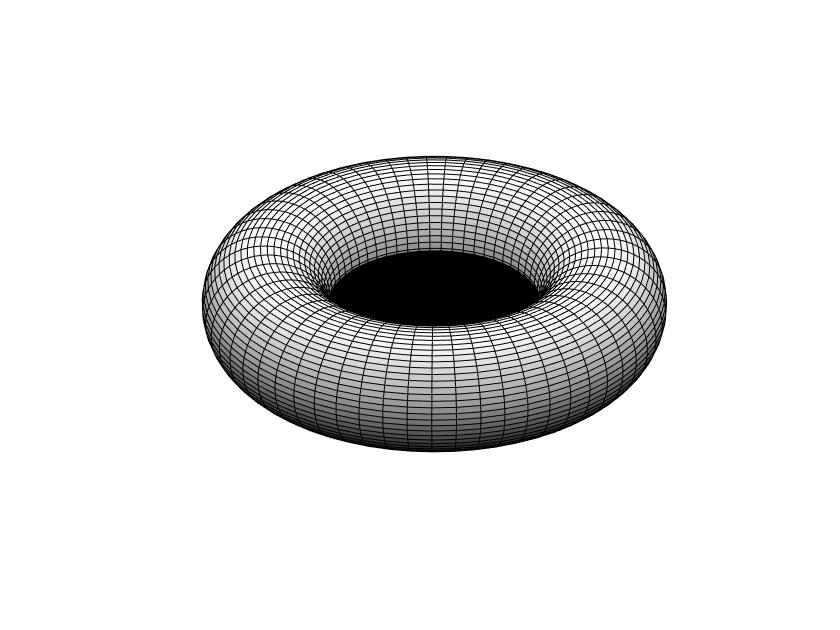}
\end{minipage}
}  \subfigure[$\mathcal{K}''_{2}$]{
\begin{minipage}[t]{0.30\textwidth}
\centering
\includegraphics[width=1.6in]{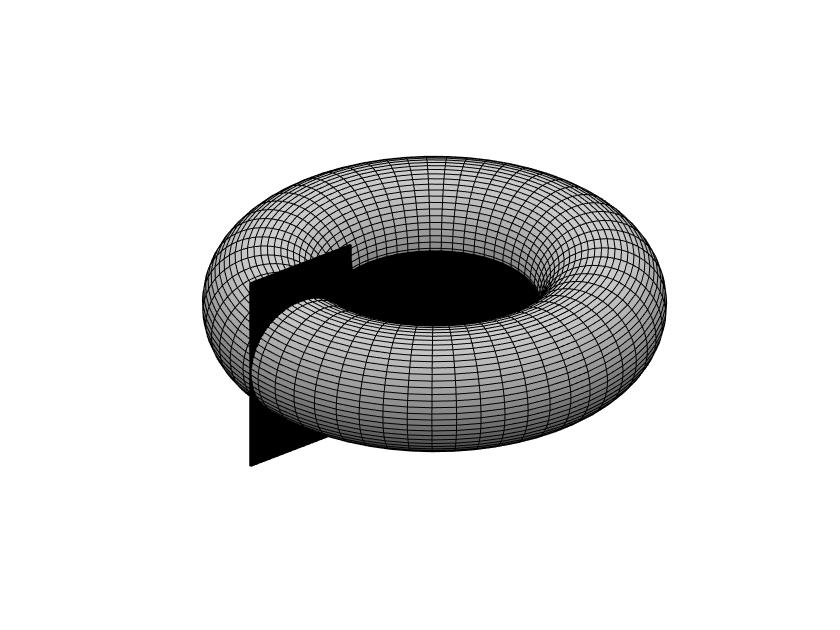}
\end{minipage}
}
\end{figure}
We set $\mathcal{K}''_{0}\simeq S^{1}\times S^{1}$, $\mathcal{K}''_{1}\simeq S^{1}\vee S^{2}$, $\mathcal{K}''_{2}\simeq S^{2}$, and $K''_{3}\simeq S^{2}$.
\end{itemize}
\begin{figure}[H]
\centering
\subfigure[The barcode of $H^{\ast}(\mathcal{K}$)]{
    \begin{tikzpicture}[scale=1.6]
        \draw[step=.5cm,gray,ultra thin, dashed] (-0.2,-0.2) grid (1.9,1.4);%
        \begin{scope}[thick]

            \draw[->] (-0.2,0) -- (2,0);
            \draw[] (0,-0.2) -- (0,1.5);
        \end{scope}
         \draw[-,thick,black!60] (0.5,0.4) -- (1.5,0.4);
         \draw[-,thick,black!60] (1,0.7) -- (1.5,0.7);
         \draw[-,thick,black!60] (1,1.1) -- (1.5,1.1);

\node [font=\fontsize{8}{6}] (node001) at (0.1,-0.1){$O$};
\node [font=\fontsize{8}{6}] (node001) at (2,-0.1){$t$};
\node [font=\fontsize{8}{6}] (node001) at (-0.15,0.4){$\alpha^{1}$};
\node [font=\fontsize{8}{6}] (node001) at (-0.15,0.7){$\alpha^{2}$};
\node [font=\fontsize{8}{6}] (node001) at (0.6,1.2){$\alpha^{1}\cdot\alpha^{2}$};

\node [font=\fontsize{8}{6}] (node001) at (0.5,-0.1){$1$};
\node [font=\fontsize{8}{6}] (node001) at (1,-0.1){$2$};
\node [font=\fontsize{8}{6}] (node001) at (1.5,-0.1){$3$};
        \end{tikzpicture}}\qquad
\subfigure[The barcode of $H^{\ast}(\mathcal{K}'$)]{
    \begin{tikzpicture}[scale=1.6]
        \draw[step=.5cm,gray,ultra thin, dashed] (-0.2,-0.2) grid (1.9,1.4);%
        \begin{scope}[thick]

            \draw[->] (-0.2,0) -- (2,0);
            \draw[] (0,-0.2) -- (0,1.5);
        \end{scope}
         \draw[-,thick,black!60] (0.5,0.4) -- (1.5,0.4);
         \draw[-,thick,black!60] (0.5,0.7) -- (1.5,0.7);
         \draw[-,thick,black!60] (1,1.1) -- (1.5,1.1);

\node [font=\fontsize{8}{6}] (node001) at (0.1,-0.1){$O$};
\node [font=\fontsize{8}{6}] (node001) at (2,-0.1){$t$};
\node [font=\fontsize{8}{6}] (node001) at (0.3,0.4){$\alpha^{1}$};
\node [font=\fontsize{8}{6}] (node001) at (0.3,0.7){$\alpha^{2}$};
\node [font=\fontsize{8}{6}] (node001) at (0.6,1.2){$\alpha^{1}\cdot\alpha^{2}$};

\node [font=\fontsize{8}{6}] (node001) at (0.5,-0.1){$1$};
\node [font=\fontsize{8}{6}] (node001) at (1,-0.1){$2$};
\node [font=\fontsize{8}{6}] (node001) at (1.5,-0.1){$3$};
        \end{tikzpicture}}\qquad
\subfigure[The barcode of $H^{\ast}(\mathcal{K}''$)]{
    \begin{tikzpicture}[scale=1.6]
        \draw[step=.5cm,gray,ultra thin, dashed] (-0.2,-0.2) grid (1.9,1.4);%
        \begin{scope}[thick]

            \draw[->] (-0.2,0) -- (2,0);
            \draw[] (0,-0.2) -- (0,1.5);
        \end{scope}
         \draw[-,thick,black!60] (0,0.4) -- (1,0.4);
         \draw[-,thick,black!60] (0,0.7) -- (0.5,0.7);
         \draw[-,thick,black!60] (0,1.1) -- (1.5,1.1);

\node [font=\fontsize{8}{6}] (node001) at (0.1,-0.1){$O$};
\node [font=\fontsize{8}{6}] (node001) at (2,-0.1){$t$};
\node [font=\fontsize{8}{6}] (node001) at (-0.15,0.4){$\alpha^{1}$};
\node [font=\fontsize{8}{6}] (node001) at (-0.15,0.7){$\alpha^{2}$};
\node [font=\fontsize{8}{6}] (node001) at (1.5,1.2){$\alpha^{1}\cdot\alpha^{2}$};

\node [font=\fontsize{8}{6}] (node001) at (0.5,-0.1){$1$};
\node [font=\fontsize{8}{6}] (node001) at (1,-0.1){$2$};
\node [font=\fontsize{8}{6}] (node001) at (1.5,-0.1){$3$};
        \end{tikzpicture}}
\caption{The persistence-cup products on tori of different filtrations.}
\end{figure}
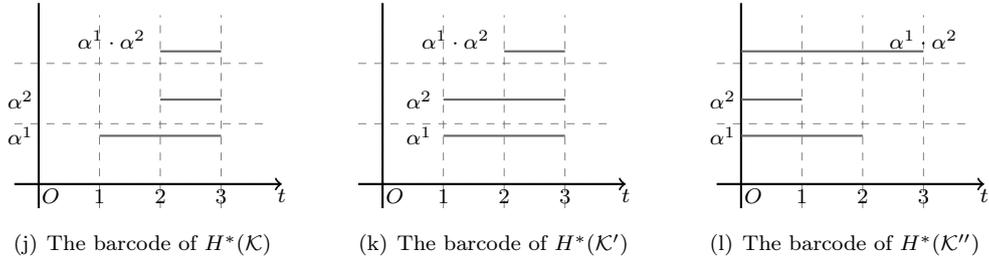
Let $t_{b},t_{d}$ denote the birth time and death time, respectively. Then the above examples satisfy
\begin{equation*}
  \max(t_{b}(\alpha^{1}),t_{b}(\alpha^{2}))\leq t_{b}(\alpha^{1}\cdot\alpha^{2}),\quad \min(t_{d}(\alpha^{1}),t_{d}(\alpha^{2}))\leq t_{d}(\alpha^{1}\cdot\alpha^{2}),
\end{equation*}
which are consistent with the results in Theorem \ref{theorem:main}.
\end{example}

\section{Persistence products on manifolds}\label{section:duality}
For a further application of persistence algebra, we study the persistence-cup product and persistence-cap product on manifolds. Moreover, we introduce the persistence Poincar\'{e} duality which is derived from the usual  Poincar\'{e} duality theorem.
In this section, the ground field is $\mathbf{k}$ and all the manifolds considered are assumed to be compact orientable $n$-manifold without boundary.
\subsection{The persistence-cap product}
Let $G$ be an abelian group, and  let $S\subseteq G$ be a monoid such that the identity element $e$ is the unique invertible element in $S$. Let $\mathcal{M}:\mathrm{cat}(\mathrm{Cay}(G,S))\rightarrow \mathbf{Mani}$ be a Cayley-persistence manifold. The homology and cohomology considered are unreduced in this section.
Then we obtain the persistent homology
\begin{equation*}
  H_{\ast}^{a,b}=\mathrm{im}(f_{\ast}^{a,b}:H_{\ast}(\mathcal{M}_{a})\rightarrow H_{\ast}(\mathcal{M}_{b})),\quad a,b\in G
\end{equation*}
and the persistent cohomology
\begin{equation*}
  H^{\ast}_{a,b}=\mathrm{im}(f^{\ast}_{a,b}:H^{\ast}(\mathcal{M}_{b})\rightarrow H^{\ast}(\mathcal{M}_{a})),\quad a,b\in G.
\end{equation*}
For convenience, we denote $H_{\ast}^{a}=H_{\ast}(\mathcal{M}_{a})$ and $H^{\ast}_{a}=H^{\ast}(\mathcal{M}_{a})$.

\begin{proposition}\label{proposition:isomorphism}
For $p\geq 0$, we have $\mathbf{k}$-linear isomorphisms
\begin{equation*}
  H_{b}^{p}\cong H_{a,b}^{p}\oplus (H^{b}_{p}/H^{a,b}_{p})^{\sharp},\quad H^{a}_{p}\cong H_{p}^{a,b}\oplus (H^{p}_{a}/H^{p}_{a,b})^{\sharp}.
\end{equation*}
Here, $(-)^{\sharp}=\mathrm{Hom}(-,\mathbf{k})$.
\end{proposition}
\begin{proof}
Consider the following diagram.
\begin{equation*}
  \xymatrix{
  0\ar@{->}[r]&(H_{p}^{b}/H_{p}^{a,b})^{\sharp}\ar@{->}[r]\ar@{->}[d]&  (H_{p}^{b})^{\sharp}\ar@{->}[r]^{(f^{a,b}_{p})^{\sharp}}\ar@{->}[d]^{\cong}&   (H_{p}^{a})^{\sharp}\ar@{->}[r]\ar@{->}[d]^{\cong}& (\ker f_{p}^{a,b})^{\sharp}\ar@{->}[r]\ar@{->}[d]&0\\
   0\ar@{->}[r]&\ker f_{a,b}^{p} \ar@{->}[r]&H_{b}^{p} \ar@{->}[r]^{f_{a,b}^{p}}&H_{a}^{p} \ar@{->}[r]& H_{a}^{p}/H_{a,b}^{p}\ar@{->}[r]&0
  }
\end{equation*}
The first horizontal sequence is exact since $\mathrm{Hom}(-,\mathbf{k})$ is left exact and $\mathbf{k}$ is a field. Note that the second and third vertical arrows are isomorphism, one has
\begin{equation}\label{equation:isomorphism}
  (H^{b}_{p}/H^{a,b}_{p})^{\sharp}\cong \ker f_{a,b}^{p} ,\quad (\ker f^{a,b}_{p})^{\sharp}\cong H_{a}^{p}/H_{a,b}^{p}.
\end{equation}
Thus we have a short exact sequence
\begin{equation*}
  0\rightarrow (H^{b}_{p}/H^{a,b}_{p})^{\sharp}\rightarrow H_{b}^{p} \rightarrow H_{a,b}^{p} \rightarrow 0,
\end{equation*}
which is also split as $\mathbf{k}$-linear spaces. This gives $H_{b}^{p}\cong H_{a,b}^{p}\oplus (H^{b}_{p}/H^{a,b}_{p})^{\sharp}$. Similarly, we have $H^{a}_{p}\cong H_{p}^{a,b}\oplus (H^{p}_{a}/H^{p}_{a,b})^{\sharp}$.
\end{proof}

\begin{remark}
This result is also applicable for a persistence CW-complex of finite type over the field $\mathbf{k}$.
\end{remark}
Let $\beta_{a,b}^{p}=\dim H_{a,b}^{p},\beta^{a,b}_{p}=\dim H^{a,b}_{p}$ denote the Betti numbers corresponding to homology and cohomology, respectively.
\begin{corollary}\label{corollary:Betti}
$\beta_{a,b}^{p}=\beta^{a,b}_{p}$.
\end{corollary}
\begin{proof}
It is obtained by a straightforward calculation of the dimensions of linear spaces.
\end{proof}
This shows that computing persistent homology and persistent cohomology are consistent in some sense, which is essentially the universal coefficient theorem of the persistent version.

Now, to study the persistence-cap product, we need to study the dual space in the persistence case. For $a\in G,x\in S$, define the morphism
\begin{equation*}
  H_{xa}^{\ast}\times H^{a}_{\ast}\rightarrow \mathbf{k}
\end{equation*}
given by $\langle\beta,\Phi \rangle=(\beta\cdot L_{x})(\Phi)$ for $\beta\in H_{xa}^{\ast}$ and $\Phi\in H^{a}_{\ast}$. Note that the map
\begin{equation*}
  L^{x}:H_{\ast}^{a}\rightarrow H_{\ast}^{xa}
\end{equation*}
induced by $\mathcal{M}_{a}\rightarrow \mathcal{M}_{xa}$ is the adjoint operator of $L_{x}$. Moreover, we have $(\beta\cdot L_{x})(\Phi)=\beta(L^{x}\Phi)$.
Let $\mathbb{L}^{S}$ be the monoid given by the elements $L^{x},x\in S$, and let $\mathbb{L}_{S}$ be the monoid given by the elements $L_{x},x\in S$. Then the homology $\{H_{\ast}^{a}\}_{a\in G}$ is a left $\mathbf{k}[\mathbb{L}^{S}]$-module while the cohomology $\{H^{\ast}_{a}\}_{a\in G}$ is a right $\mathbf{k}[\mathbb{L}_{S}]$-module. The cohomology $\{H^{\ast}_{a}\}_{a\in G}$ is an adjoint module of $\{H_{\ast}^{a}\}_{a\in G}$.

\begin{definition}
For $p\leq q$ and $a\in G,x\in S$, the \emph{persistence-cap product} $\vec{\cap}:H^{a}_{q}\times H^{p}_{xa}\rightarrow H_{q-p}^{a}$ is defined by
\begin{equation*}
  \Phi\vec{\cap} \beta=\Phi\cap (\beta\cdot L_{x})\quad \Phi\in H^{a}_{q},\beta\in H^{p}_{xa}.
\end{equation*}
\end{definition}
Let $\omega^{a}\in H_{n}^{a}$ be the fundamental class for $a\in G$. Let $\beta\in H_{xa}^{\ast},x\in S$.
We have
\begin{equation*}
  \omega^{a}\vec{\cap}\beta=\omega^{a}\cap (\beta\cdot L_{x})=D(\beta\cdot L_{x}).
\end{equation*}
Here, $D:H^{k}_{a}\rightarrow H_{n-k}^{a}$ is defined by $D(\alpha)=\omega^{a}\cap \alpha$ for all $0\leq k\leq n$.
Similarly, the persistence-cap product is dual to the persistence-cup product.
\begin{lemma}
For $a,b,c\in G$ satisfying $ba^{-1},ca^{-1}\in S$, let $\beta\in H_{b}^{\ast},\gamma\in H_{c}^{\ast}$ and $\Phi\in H^{a}_{\ast}$. Then we have
\begin{equation*}
  \langle\beta\cdot \gamma,\Phi \rangle=\langle\gamma, \Phi \vec{\cap}\beta\rangle.
\end{equation*}
\end{lemma}
\begin{proof}
By a straightforward computation, we have
\begin{equation*}
  \begin{split}
    \langle\gamma, \Phi \vec{\cap}\beta\rangle=& \langle\gamma,\Phi\cap (\beta\cdot L_{ba^{-1}})\rangle\\
      =&   (\gamma\cdot L_{ca^{-1}})(\Phi\cap (\beta\cdot L_{ba^{-1}}))\\
      =& ((\beta\cdot L_{ba^{-1}})\cup (\gamma\cdot L_{ca^{-1}}))(\Phi)\\
      =& \left((\beta\cdot\gamma)\cdot L_{(b\times c)a^{-1}}\right)(\Phi)\\
      =& \langle\beta\cdot \gamma,\Phi\rangle.
  \end{split}
\end{equation*}
Here, we use the fact that $(\alpha'\cup\beta')(\Phi')=\beta'(\Phi'\cap\alpha')$ for $\alpha',\beta'\in H^{a}_{\ast},\Phi'\in H_{a}^{\ast}$.
\end{proof}
For $b,c\in G$, let $\beta\in H_{b}^{\ast},\gamma\in H_{c}^{\ast}$.
We denote
\begin{equation*}
  (\beta,\gamma)_{D} =\langle\beta\cdot\gamma,\omega^{b\times c}\rangle.
\end{equation*}
Note that $(\beta,\gamma)_{D}$ is $\mathbf{k}$-bilinear. Moreover, we have
 \begin{equation*}
  (\beta,\gamma)_{D}=\langle\beta\cdot\gamma,\omega^{b\times c}\rangle=\langle\gamma ,\omega^{b\times c}\vec{\cap} \beta\rangle=\langle\gamma ,D(\beta\cdot L_{b(b\times c)^{-1}})\rangle.
 \end{equation*}
The following proposition indicates that $(\beta,\gamma)_{D}$ is also $\mathbf{k}[\mathbb{L}_{S}]$-bilinear under certain condition.
\begin{proposition}\label{proposition:adj1}
For $b,c\in G, x\in S$, we have
\begin{equation*}
  (\beta\cdot L_{x},\gamma)_{D}=\langle\beta\cdot \gamma , \omega^{x^{-1}b\times c}\rangle,\quad \beta\in H^{\ast}_{b},\gamma\in H^{\ast}_{c}.
\end{equation*}
If $b=c$, we have $(\beta\cdot L_{x},\gamma)_{D}=(\beta,\gamma\cdot L_{x})_{D}$ for any $\beta,\gamma\in H^{\ast}_{b}$.
\end{proposition}
\begin{proof}
A direct calculation shows that
\begin{equation*}
  \begin{split}
    (\beta\cdot L_{x},\gamma)_{D}= & \langle(\beta\cdot L_{x})\cdot \gamma,\omega^{x^{-1}b\times c}\rangle \\
      =& \langle(\beta\cdot \gamma)\cdot L_{x},\omega^{x^{-1}b\times c}\rangle \\
      =&((\beta\cdot \gamma)\cdot L_{x})(\omega^{x^{-1}b\times c})\\
      =&\langle \beta\cdot \gamma,\omega^{x^{-1}b\times c}\rangle .
  \end{split}
\end{equation*}

\begin{equation*}
   (\beta\cdot L_{x},\gamma)_{D}=\langle(\beta\cdot \gamma)\cdot L_{x},\omega^{x^{-1}b\times c}\rangle  =((\beta\cdot \gamma)\cdot L_{x})(\omega^{x^{-1}b\times c})=\langle \beta\cdot \gamma,\omega^{x^{-1}b\times c}\rangle .
\end{equation*}
If $b=c$, then $b\times x^{-1}b=x^{-1}b\times b=x^{-1}b$. Thus we have
\begin{equation*}
  (\beta\cdot L_{x},\gamma)_{D}=\langle(\beta\cdot \gamma)\cdot L_{x},\omega^{x^{-1}b}\rangle=\langle\beta\cdot (\gamma\cdot L_{x}),\omega^{x^{-1}b}\rangle=(\beta,\gamma\cdot L_{x})_{D}.
\end{equation*}
This completes the proof.
\end{proof}

\subsection{The persistence Poincar\'{e} duality}
In this section, we will further study the persistence duality with respect to the Poincar\'{e} duality. Moreover, we show that the persistent morphism on fundamental classes of homology is crucial to describing the persistent (co)homology.

\begin{proposition}\label{proposition:adj2}
For $a,b\in G$ with $ba^{-1}\in S$, let $f^{a,b}_{\ast}:H_{\ast}^{a}\rightarrow H_{\ast}^{b}$ be a map induced by $\mathcal{M}_{a}\rightarrow \mathcal{M}_{b}$ and $f^{a,b}_{n}(\omega^{a})=\lambda^{a,b}\omega^{b},\lambda^{a,b}\in \mathbf{k}$. Here, $\omega^{a},\omega^{b}$ are the fundamental classes in $H^{a}_{n},H^{b}_{n}$, respectively.
Then we have
\begin{equation*}
  \lambda^{a,b}(\beta,\gamma)_{D}=(\beta\cdot L_{ba^{-1}},\gamma)_{D},\quad\beta,\gamma\in H_{b}^{\ast}.
\end{equation*}
\end{proposition}
\begin{proof}
For $a\in G,x=ba^{-1}\in S$, we denote $L^{x}:H_{\ast}^{a}\rightarrow H_{\ast}^{xa}$ and $L_{x}:H_{a}^{\ast}\rightarrow H_{xa}^{\ast}$. Consider the following diagram.
\begin{equation*}
  \xymatrix{
   H_{q}^{a} \ar@{->}[d]^{L^{x}}&\times &H_{a}^{p}\ar@{->}[r]^{\cap} &H_{q-p}^{a}\ar@{->}[d]^{L^{x}}\\
     H_{q}^{xa}& \times &H_{xa}^{p}\ar@{->}[u]^{L_{x}}\ar@{->}[r]^{\cap} &H_{q-p}^{xa}
  }
\end{equation*}
Then we have \cite{Hatcher}
\begin{equation*}
  L^{x}\Phi\cap \beta=L^{x}(\Phi\cap (\beta\cdot L_{x})),\quad \Phi\in H_{q}^{a},\beta\in H^{p}_{xa}.
\end{equation*}
Take $\Phi=\omega^{a}$, we obtain
\begin{equation*}
   (L^{x}\omega^{a})\cap \beta=L^{x}(D(\beta\cdot L_{x})),\quad \beta\in H^{p}_{xa}.
\end{equation*}
\begin{equation*}
  \xymatrix{
 H^{p}_{xa}\ar@{->}[r]^{D}\ar@{->}[d]_{L_{x}}&H_{n-p}^{xa}\\
 H^{p}_{a}\ar@{->}[r]^{D}&H_{n-p}^{a}\ar@{->}[u]_{L^{x}}\\
 }
\end{equation*}
Suppose that $L^{x}\omega^{a}=\lambda^{a,xa}\omega^{xa}$. Thus one has
\begin{equation*}
  \lambda^{a,xa} D(\beta)=L^{x}(D(\beta\cdot L_{x})).
\end{equation*}
Note that $L^{x}$ is the adjoint operator of $L_{x}$, it follows that
\begin{equation*}
  \langle\gamma,\lambda^{a,xa} D(\beta)\rangle=\langle\gamma, L^{x}(D(\beta\cdot L_{x}))\rangle=\langle\gamma\cdot L_{x}, D(\beta\cdot L_{x})\rangle=\langle\gamma, D(\beta\cdot L_{x})\rangle,\quad \gamma\in H^{p}_{b}.
\end{equation*}
Hence, we have $\lambda^{a,xa}(\beta,\gamma)_{D}=(\beta\cdot L_{x},\gamma)_{D}$. This completes the proof.
\end{proof}
Recall the Poincar\'{e} duality theorem.
\begin{theorem}[\cite{Hatcher}]
Let $R$ be a commutative ring with unit.
If $M$ is a closed $R$-orientable $n$-manifold with fundamental class $[M]\in H_{n}(M;R)$, then the map $D:H^{k}(M; R)\rightarrow H_{n-k}(M; R)$ defined by $D(\alpha) = [M]\cap \alpha$ is an isomorphism for all $k$.
\end{theorem}
The following result exhibits the Poincar\'{e} duality theorem of the persistence version.
\begin{theorem}\label{theorem:duality}
For $a,b\in G$ with $ba^{-1}\in S$, let $f^{a,b}_{\ast}:H_{\ast}^{a}\rightarrow H_{\ast}^{b}$ be a map induced by $\mathcal{M}_{a}\rightarrow \mathcal{M}_{b}$ and $f^{a,b}_{n}(\omega^{a})=\lambda^{a,b}\omega^{b},\lambda^{a,b}\in \mathbf{k}$. Here, $\omega^{a},\omega^{b}$ are the fundamental classes in $H^{a}_{n},H^{b}_{n}$, respectively.
\begin{itemize}
  \item[$(i)$] If $\lambda^{a,b}\neq 0$, then the map $$D^{a,b}=f^{a,b}_{\ast}\circ D:H^{p}_{a,b}\rightarrow H_{n-p}^{a,b}$$ is an isomorphism for all $p$. Moreover, we have $\beta_{a,b}^{p}=\beta_{a,b}^{n-p}$.
  \item[$(ii)$] If $\lambda^{a,b}=0$, then the  $(a,b)$-persistent cup-space $P_{a,b}^{n}=0$. Moreover, we have
  \begin{equation*}
  \beta_{p}^{a,b}+\beta_{n-p}^{a,b}\leq \beta_{p}^{a},\quad \beta^{p}_{a,b}+\beta^{n-p}_{a,b}\leq \beta^{p}_{a}.
  \end{equation*}
\end{itemize}
\end{theorem}
\begin{proof}
For convenience, Let $x=ba^{-1}$. By the proof of Proposition \ref{proposition:adj2}, we have
\begin{equation*}
  \lambda^{a,xa} D(\beta)=L^{x}(D(\beta\cdot L_{x})),\quad \beta\in H^{\ast}_{b}.
\end{equation*}

$(i)$ When $\lambda^{a,xa}\neq 0$. By Poincar\'{e} duality theorem, the $\mathbf{k}$-linear map
\begin{equation*}
  L^{x}\circ D\circ L_{x}:H^{p}_{xa}\rightarrow H^{p}_{a,xa}\rightarrow H_{n-p}^{a,xa}
\end{equation*}
is an isomorphism. Consider the linear map
\begin{equation*}
   L^{x}\circ D:H^{p}_{xa}\cdot L_{x}=H^{p}_{a,xa}\rightarrow H_{n-p}^{a,xa}.
\end{equation*}
For $\beta\in H_{n-p}^{a,xa}$, let $\alpha=(\lambda^{a,xa})^{-1} (D^{-1}\beta)\cdot L_{x}$. Then $L^{x}\circ D (\alpha)=\beta$ and $L^{x}\circ D$ is a surjection. On the other hand, if $L^{x}\circ D(\alpha)=0$ for some $\alpha\in H^{p}_{a,xa}\cdot L_{x}$, then there exists an element $\beta\in H^{p}_{xa}$ such that $\alpha=\beta\cdot L_{x}$. Thus we have
\begin{equation*}
  L^{x}\circ D(\beta\cdot L_{x})=\lambda^{a,xa} D(\beta)=0.
\end{equation*}
Since $\lambda^{a,xa}\neq 0$ and $D$ is an isomorphism, we obtain $\beta=0$. It follows that $\alpha=0$, which shows that $L^{x}\circ D$ is an injection. By Corollary \ref{corollary:Betti}, we have $\beta_{a,b}^{p}=\beta_{a,b}^{n-p}$.

$(ii)$ When $\lambda^{a,xa}=0$. Let $\omega_{xa}\cdot L_{x}=\mu\omega_{a}$. Here, $\omega_{a},\omega_{b}$ are the dual generators of $\omega^{a},\omega^{b}$ in $H_{a}^{n},H_{b}^{n}$, respectively. In view of $L^{x}(D(\omega_{xa}\cdot L_{x}))=0$, we obtain
\begin{equation*}
  0=\mu L^{x} D(\omega_{x})=\mu L^{x}(1)=\mu.
\end{equation*}
This implies that
\begin{equation*}
  P_{a\to b}^{n}=\mathrm{im}(H_{xa}^{\ast}\times H_{xa}^{n-\ast}\stackrel{\cup}{\rightarrow} H_{xa}^{n}\stackrel{L_{x}}{\rightarrow} H_{a}^{n})=0.
\end{equation*}
On the other hand, note that $L^{x}(D(\alpha\cdot L_{x}))=0$ for any $\alpha\in H^{a}_{\ast}$.  One has
\begin{equation*}
 D(H^{p}_{a,xa})= D(H^{p}_{xa}\cdot L_{x})\subseteq  H_{n-p}^{a}\cap \ker L^{x}.
\end{equation*}
By Poincar\'{e} duality theorem, we have
\begin{equation*}
  \dim H^{p}_{a,xa}\leq \dim H_{n-p}^{a}\cap \ker L^{x}.
\end{equation*}
By Equation (\ref{equation:isomorphism}), we obtain $\dim H_{a}^{p}=\dim H_{a,xa}^{p}+\dim H_{p}^{a}\cap\ker L^{x}$. It follows that
\begin{equation*}
  \dim H_{a,xa}^{p}+\dim H_{a,xa}^{n-p}\leq \dim H_{a}^{p}.
\end{equation*}
This gives $\beta^{p}_{a,xa}+\beta^{n-p}_{a,xa}\leq \beta^{p}_{a}$.
By Corollary \ref{corollary:Betti}, one obtains
\begin{equation*}
  \beta_{p}^{a,xa}+\beta_{n-p}^{a,xa}\leq \beta_{p}^{a}.
\end{equation*}
The desired result follows.
\end{proof}

\medskip
\noindent {\bf Acknowledgement}.
This work is supported in part by the Natural Science Foundation of China (NSFC grant no. 11971144), High-level Scientific Research Foundation of Hebei Province and the start-up research fund from BIMSA.
This work is also supported in part by the Tianjin Natural Science Foundation (Grant No. 19JCYBJC30200).
\bibliographystyle{plain}
\bibliography{per_cup}

\bigskip

Wanying Bi

Affiliation: $^1$School of Mathematical Sciences, Hebei Normal University, 050024, China.

$^2$Yanqi Lake Beijing Institute of Mathematical Sciences and Applications, 101408, China.

e-mail: 13244522050@163.com

\medskip

Jingyan Li

Affiliation: Yanqi Lake Beijing Institute of Mathematical Sciences and Applications, 101408, China.

e-mail: jingyanli@bimsa.cn

\medskip

Jian Liu

Affiliation: $^1$School of Mathematical Sciences, Hebei Normal University, 050024, China.

$^2$Yanqi Lake Beijing Institute of Mathematical Sciences and Applications, 101408, China.

e-mail: liujian@bimsa.cn

\medskip

Jie Wu

Affiliation: Yanqi Lake Beijing Institute of Mathematical Sciences and Applications, 101408, China.

e-mail: wujie@bimsa.cn

\medskip

\end{CJK*}
 \end{document}